\documentclass{c-art2}     
\volume{000}
\startingpage{1}

\usepackage{amsthm}

\authorheadline{M.~Adler,  J.~Del\'epine, and  P.~van Moerbeke}
\titleheadline{Airy process with outliers}

\usepackage{mathptmx}

\usepackage{graphicx}

\newtheorem{theorem}{Theorem}[section]
\newtheorem{lemma}[theorem]{Lemma}
\newtheorem{corollary}[theorem]{Corollary}
\newtheorem{proposition}[theorem]{Proposition}
\theoremstyle{definition}

\theoremstyle{remark}
\newtheorem{remark}[theorem]{Remark}

\newcommand{\rg}{\rightarrow}
\newcommand{\lrg}{\longrightarrow}
\newcommand{\Rg}{\Rightarrow}
\newcommand{\DF}{\Longleftrightarrow}
\newcommand{\df}{\longleftrightarrow}
\newcommand{\pp}{\ldots}
\newcommand{\TT}{\tilde\tau}
\newcommand{\AR}{\mathcal {A}}
\newcommand{\CR}{\mathcal {C}}
\newcommand{\DR}{\mathcal {D}}
\newcommand{\FR}{\mathcal {F}}
\newcommand{\GR}{\mathcal{G}}
\newcommand{\HR}{\mathcal {H}}
\newcommand{\JR}{\mathcal{J}}
\newcommand{\LR}{\mathcal {L}}
\newcommand{\OR}{\mathcal{O}}
\newcommand{\PR}{\mathcal{P}}
\newcommand{\SR}{\mathcal{S}}
\newcommand{\VR}{\mathcal{V}}
\newcommand{\WR}{\mathcal{W}}
\newcommand{\BB}{\mathcal {B}}
\newcommand{\BT}{\mathcal{T}}
\newcommand{\XR}{\mathcal{X}}
\newcommand{\BC}{\mathbb{C}}
\newcommand{\BD}{\mathbb{D}}
\newcommand{\BF}{\mathbb{F}}
\newcommand{\BH}{\mathbb{H}}
\newcommand{\BL}{\mathbb{L}}
\newcommand{\BP}{\mathbb{P}}
\newcommand{\BQ}{\mathbb{Q}}
\newcommand{\BS}{\mathbb{S}}
\newcommand{\BV}{\mathbb{V}}
\newcommand{\BX}{\mathbb{X}}
\newcommand{\BY}{\mathbb{Y}}
\newcommand{\BZ}{\mathbb{Z}}
\newcommand{\Sg}{\Sigma}
\newcommand{\iy}{\infty}
\newcommand{\pl}{\partial}
\newcommand{\al}{\alpha}

\newcommand{\tp}{\tilde\partial}
\newcommand{\gs}{{\bf s}}
\newcommand{\no}{\nonumber}
\newcommand{\oiint}{\frac{1}{(2\pi
i)^2}\oint\limits_{\iy}\!\!\oint\limits_{\iy}}
\newcommand{\tx}{\tilde x}
\newcommand{\ty}{\tilde y}
\newcommand{\ba}{{\backslash}}

\newcommand{\acht}{{\rm ht}}

\newcommand{\shade}{/\!\!/\!\!/\!\!/\!\!
}
\newcommand{\bshade}{\backslash\!\!
\backslash\!\!\backslash\!\!\backslash
}

\def\inn#1#2{\left\langle#1\,\left\vert\,#2\right.\right\rangle}
\def\INN{\langle\cdot\,\vert\,\cdot\rangle}

\newcommand{\vv}{||}

\def\be#1\ee{\begin{equation}#1\end{equation}}
\def\bea#1\eea{\begin{eqnarray}#1\end{eqnarray}}
\def\bean#1\eean{\begin{eqnarray*}#1\end{eqnarray*}}

\newcommand{\Pf}{\operatorname{\rm Pfaff}}
\newcommand{\Tr}{\operatorname{\rm Tr}}
\newcommand{\Mat}{\operatorname{\rm Mat}}
\newcommand{\OneOrTwo}{{\left\{\substack{1\\2}\right\}}}

\newcommand{\om}{\omega}
\newcommand{\vp}{\varphi}
\newcommand{\la}{\langle}
\newcommand{\ra}{\rangle}
\newcommand{\ga}{\gamma}
\newcommand{\Ga}{\Gamma}
\newcommand{\dt}{\delta}
\newcommand{\Dt}{\Delta}
 \newcommand{\vr}{\varepsilon}
\newcommand{\sg}{\sigma}
\newcommand{\BR}{\mathbb{R}}
\newcommand{\lb}{\lambda}
\newcommand{\Lb}{\Lambda}
\newcommand{\tr}{\mbox{tr}}
\newcommand{\dis}{\displaystyle}

\newcommand{\Bn}{\mathfrak{n}}
\newcommand{\Bk}{\mathfrak{k}}
\newcommand{\BG}{\mathbb{G}}
\newcommand{\BK}{\mathbb{K}}
\newcommand{\BJ}{\mathbb{J}}
\newcommand{\BI}{\mathbb{I}}
\newcommand{\Span}{\operatorname{span}}
\newcommand{\diag}{\operatorname{diag}}
\newcommand{\Res}{\operatorname{Res}}

\newcommand\ig[1]{{}}

\begin{document}                        


\title{Dyson's Nonintersecting Brownian Motions \\ with a Few Outliers}

\author{Mark Adler}{Brandeis University}
\author{Jonathan Del\'epine}{Universit\'e Catholique de Louvain}
 \author{Pierre van Moerbeke}{Universit\'e Catholique de Louvain \& Brandeis University}





 \begin{abstract}
 Consider $n$ non-intersecting \mbox{Brownian} particles
on $\BR$ (Dyson Brownian motions), all starting from the
 origin at time $t=0$, and forced to return to $x=0$
 at time $t=1$. For large $n$, the average mean density of particles
has its support, for each $0<t<1$, on the interval $\pm
\sqrt{2nt(1-t)}$. The Airy process $\mathcal{A}(\tau)$
is
 defined as the motion of these non-intersecting
 Brownian motions for large $n$, but viewed from the
 curve $\mathcal{C}:\quad y=\sqrt{2nt(1-t)}$ with
 an appropriate space-time rescaling.
 Assume now a finite number $r$ of these particles are
forced to a different target point, say
$a=\rho_0\sqrt{n/2}>0$. Does it affect the Brownian
fluctuations along the curve $\mathcal{C}$ for large
$n$? In this paper, we show that no new process appears
as long as one considers points $(y,t)\in \mathcal{C},
\mbox{ such that } 0<t<({1+\rho_0^2})^{-1}
 $, which is the $t$-coordinate of
the point of tangency of the tangent to the curve
passing through $(\rho_0\sqrt{n/2},1)$. At this point of tangency 
the fluctuations obey a new statistics, which we call
the {\bf Airy process with $r$ outliers} $\mathcal
{A}^{(r)}(\tau)$ (in short: {\bf $r$-Airy process}).
   The log of the probability
   that at time $\tau$ none of the particles in the 
 cloud exceeds $x$ is given by the Fredholm
 determinant of a new kernel (extending the Airy kernel)
 and it satisfies a non-linear PDE in $x$ and $\tau$,
 from which the asymptotic behavior of the process can be
 deduced for $\tau\rg -\iy$. This kernel is closely
 related to one found by Baik, Ben Arous and P\'ech\'e in
 the context of multivariate statistics.

 \end{abstract}

\maketitle



  \tableofcontents

\ig{
\newcommand{\MAT}[1]{\left(\begin{array}{*#1c}}
\newcommand{\mat}{\end{array}\right)}
\newcommand{\qed}{\leavevmode\unskip\nobreak\penalty200\hskip2pt\null
\nobreak\hfill\rule{1.1ex}{1.1ex}
 \medbreak }

\newcommand{\rg}{\rightarrow}
\newcommand{\lrg}{\longrightarrow}
\newcommand{\Rg}{\Rightarrow}
\newcommand{\DF}{\Longleftrightarrow}
\newcommand{\df}{\longleftrightarrow}
\newcommand{\pp}{\ldots}
\newcommand{\TT}{\tilde\tau}
\newcommand{\AR}{\mathcal {A}}
\newcommand{\CR}{\mathcal{C}}
\newcommand{\DR}{\mathcal{D}}
\newcommand{\FR}{\mathcal{F}}
\newcommand{\GR}{\mathcal{G}}
\newcommand{\HR}{\mathcal {H}}
\newcommand{\JR}{\mathcal{J}}
\newcommand{\LR}{\mathcal{L}}
\newcommand{\MR}{\mathcal{M}}
\newcommand{\NR}{\mathcal{N}}
\newcommand{\OR}{\mathcal{O}}
\newcommand{\PR}{\mathcal{P}}
\newcommand{\SR}{\mathcal{S}}
\newcommand{\VR}{\mathcal{V}}
\newcommand{\WR}{\mathcal{W}}
\newcommand{\BB}{\mathcal{B}}
\newcommand{\BT}{\mathcal{T}}
\newcommand{\XR}{\mathcal{X}}
\newcommand{\BC}{\mathbb{C}}
\newcommand{\BD}{\mathbb{D}}
\newcommand{\BF}{\mathbb{F}}
\newcommand{\BH}{\mathbb{H}}
\newcommand{\BL}{\mathbb{L}}
\newcommand{\BP}{\mathbb{P}}
\newcommand{\BQ}{\mathbb{Q}}
\newcommand{\BS}{\mathbb{S}}
\newcommand{\BV}{\mathbb{V}}
\newcommand{\BX}{\mathbb{X}}
\newcommand{\BY}{\mathbb{Y}}
\newcommand{\BZ}{\mathbb{Z}}
\newcommand{\Sg}{\Sigma}
\newcommand{\iy}{\infty}
\newcommand{\pl}{\partial}
\newcommand{\al}{\alpha}

\newcommand{\tp}{\tilde\partial}
\newcommand{\gs}{{\bf s}}
\newcommand{\no}{\nonumber}
\newcommand{\oiint}{\frac{1}{(2\pi
i)^2}\oint\limits_{\iy}\!\!\oint\limits_{\iy}}
\newcommand{\tx}{\tilde x}
\newcommand{\ty}{\tilde y}
\newcommand{\ba}{{\backslash}}

\newcommand{\acht}{{\rm ht}}

\newcommand{\shade}{/\!\!/\!\!/\!\!/\!\!
}
\newcommand{\bshade}{\backslash\!\!
\backslash\!\!\backslash\!\!\backslash
}

\def\inn#1#2{\left\langle#1\,\left\vert\,#2\right.\right\rangle}
\def\INN{\langle\cdot\,\vert\,\cdot\rangle}

\newcommand{\vv}{||}

\newenvironment{proof}{\medskip\noindent{\it Proof:\/} }{\qed}
\newenvironment
        {remark}{\medskip\noindent\underline{\it Remark:\/} }{\medbreak}
\newenvironment
        {example}{\medskip\noindent\underline{\it Example:\/} }{\medbreak}

\newcommand{\om}{\omega}
\newcommand{\vp}{\varphi}
\newcommand{\la}{\langle}
\newcommand{\ra}{\rangle}
\newcommand{\ga}{\gamma}
\newcommand{\Ga}{\Gamma}
\newcommand{\dt}{\delta}
\newcommand{\Dt}{\Delta}
 \newcommand{\vr}{\varepsilon}
\newcommand{\sg}{\sigma}
\newcommand{\BR}{\mathbb{R}}
\newcommand{\lb}{\lambda}
\newcommand{\Lb}{\Lambda}
\newcommand{\tr}{\mbox{tr}}
\newcommand{\dis}{\displaystyle}

\newcommand{\Bn}{\mathfrak{n}}
\newcommand{\Bk}{\mathfrak{k}}
\newcommand{\BG}{\mathbb{G}}
\newcommand{\BK}{\mathbb{K}}
\newcommand{\BJ}{\mathbb{J}}
\newcommand{\BI}{\mathbb{I}}
\newcommand{\Span}{\operatorname{span}}
\newcommand{\diag}{\operatorname{diag}}
\newcommand{\Res}{\operatorname{Res}}




\def\be#1\ee{\begin{equation}#1\end{equation}}
\def\bea#1\eea{\begin{eqnarray}#1\end{eqnarray}}
\def\bean#1\eean{\begin{eqnarray*}#1\end{eqnarray*}}

\newcommand{\Pf}{\operatorname{\rm Pfaff}}
\newcommand{\Tr}{\operatorname{\rm Tr}}
\newcommand{\Mat}{\operatorname{\rm Mat}}
\newcommand{\OneOrTwo}{{\left\{\substack{1\\2}\right\}}}





\catcode `!=11

\newdimen\squaresize
\newdimen\thickness
\newdimen\Thickness
\newdimen\ll! \newdimen \uu! \newdimen\dd! \newdimen \rr! \newdimen
\temp!

\def\sq!#1#2#3#4#5{%
\ll!=#1 \uu!=#2 \dd!=#3 \rr!=#4
\setbox0=\hbox{%
 \temp!=\squaresize\advance\temp! by .5\uu!
 \rlap{\kern -.5\ll!
 \vbox{\hrule height \temp! width#1 depth .5\dd!}}%
%
 \temp!=\squaresize\advance\temp! by -.5\uu!
 \rlap{\raise\temp!
 \vbox{\hrule height #2 width \squaresize}}%
%
 \rlap{\raise -.5\dd!
 \vbox{\hrule height #3 width \squaresize}}%
%
 \temp!=\squaresize\advance\temp! by .5\uu!
 \rlap{\kern \squaresize \kern-.5\rr!
 \vbox{\hrule height \temp! width#4 depth .5\dd!}}%
%
 \rlap{\kern .5\squaresize\raise .5\squaresize
 \vbox to 0pt{\vss\hbox to 0pt{\hss $#5$\hss}\vss}}%
}
 \ht0=0pt \dp0=0pt \box0
}

\def\vsq!#1#2#3#4#5\endvsq!{\vbox to \squaresize{\hrule
width\squaresize height 0pt%
\vss\sq!{#1}{#2}{#3}{#4}{#5}}}

\newdimen \LL! \newdimen \UU! \newdimen \DD! \newdimen \RR!

\def\vvsq!{\futurelet\next\vvvsq!}
\def\vvvsq!{\relax
  \ifx     \next l\LL!=\Thickness \let\continue=\skipnexttoken!
  \else\ifx\next u\UU!=\Thickness \let\continue=\skipnexttoken!
  \else\ifx\next d\DD!=\Thickness \let\continue=\skipnexttoken!
  \else\ifx\next r\RR!=\Thickness \let\continue=\skipnexttoken!
  \else\def\continue{\vsq!\LL!\UU!\DD!\RR!}%
  \fi\fi\fi\fi
  \continue}

\def\skipnexttoken!#1{\vvsq!}

\def\place#1#2#3{\vbox to 0pt{\vss
\rlap{\kern#1\squaresize
  \raise#2\squaresize\hbox{$#3$}}
\vss}}

\def\Young#1{\LL!=\thickness \UU!=\thickness \DD! = \thickness \RR! =
\thickness \vbox{\smallskip\offinterlineskip
\halign{&\vvsq! ##
\endvsq!\cr #1}}}

\def\blank{\omit\hskip\squaresize}
\catcode `!=12

\squaresize = 35pt \thickness = 1pt \Thickness = 3pt

}

Dyson \cite{Dyson} made the important observation that
putting dynamics into random matrix models leads to finitely many
non-intersecting Brownian motions (on ${\mathbb R}$) for the eigenvalues.
Applying scaling limits to the random matrix models,
combined with Dyson's dynamics, then leads to
interesting infinitely many diffusions for the
eigenvalues. This paper studies a model, which stems
from multivariate statistics and which interpolates
between the {\bf Airy} and {\bf Pearcey} processes.

Consider $n$ non-intersecting \mbox{Brownian} particles on the
real line $\BR$,
$$
-\iy <x_1(t)<\ldots<x_n(t)<\iy,$$
 with (local) Brownian transition probability given by
\be p(t ; x, y)  :=  \frac{1}{\sqrt{\pi t}} e^{-
\frac{(y-x)^{2}}{t}}, \label{Normal}\ee  all starting
from the origin $x=0$ at time $t=0$, and forced to
return to $x=0$ at time $t=1$.
For very large $n$, the average mean density of particles
has its support, for each $0<t<1$, on the interval
$(-\sqrt{2nt(1-t)},\sqrt{2nt(1-t)})$, as sketched in Figure \ref{fig1}.

The {\bf Airy process} ${A}(\tau)$ is defined as the
motion of these non-intersecting Brownian motions for
large $n$, but viewed from an observer on the (right
hand) edge-curve $$\mathcal{C}:\quad
\{y=\sqrt{2nt(1-t)}>0\}$$ of the set of particles, with
space stretched by the customary GUE-edge rescaling $
n^{1/6}$ and time rescaled by the factor $ n^{1/3}$ in
tune with the Brownian motion space-time rescaling; this
is to say that in this new scale, slowed down
microscopically, the left-most particles appear
infinitely far and the time horizon $t=1$ lies in the
very remote future. Thus, the Airy process describes the
fluctuations of the Brownian particles near the
edge-curve $\mathcal{C}$, looked at through a magnifying glass, as shown in Figure  \ref{fig1}.

 The Airy process was introduced by Pr\"ahofer and Spohn \cite{Spohn} and further
  investigated in
  \cite{Johansson2,Johansson3,TW-Dyson,AvM-Airy-Sine}.
  Notice that
in this work the Airy process is not viewed as the
motion of the largest particle (point process), but as
the motion of the cloud of particles, which will be
described as a determinantal process. Giving
a pathwise description of this motion remains an open
problem.

\begin{figure}
\hspace*{-2.9cm}\includegraphics[width=180mm,height=120mm]{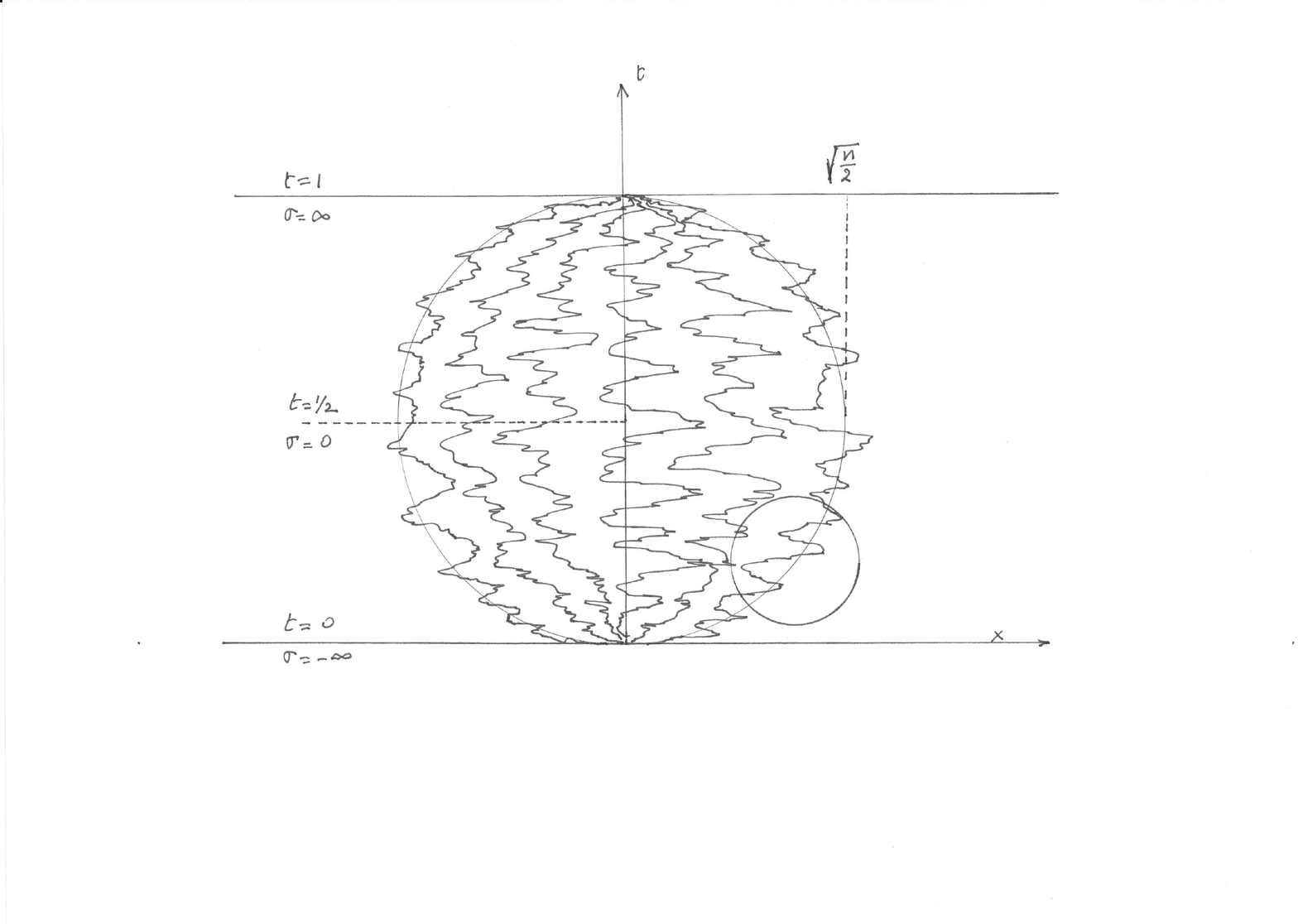}

\vspace{-3.5cm}

\caption{Airy process}\label{fig1}
\end{figure}

 \vspace*{-4cm}

$$\hspace*{9cm}\mathcal{C}:=\{x=\sqrt{2nt(1-t)}\}$$

\vspace*{3cm}

Assume now that, among those $n$ paths, $0\leq r\leq n$
are forced to reach a given final target $a \geq 0$,
\,while the $(n-r)$ remaining particles return to the
position $x=0$. Consider the probability that, at any
given time $0 < t < 1$ , all of the particles avoid a
window $E \subset \BR$, namely (the superscript in the
probability $\BP_{Br}^{0a}$ here and later refers to the
target points)
\be  \BP_{Br}^{0a}\left(\mbox{all  $x_j(t)\in
E^c$}\right)  := \! \BP\!\!
\left(\!\!\!\begin{tabular}{c|c}
& all $x_j(0) =0$\\
all $x_j(t)\in E^c$ &
 $r$  right paths end up at  $a$  at $t=1$\\
  &
$(n-r)$ paths end up at $0$ at $t=1$
\end{tabular}\!\!\!\right).
  \label{Brownian target}\ee
%
%
%
%
%
%
%
Does the fact that a finite number $r$ of particles are
forced to a different target point, in particular the
target point $a=\rho_0\sqrt{n/2}>0$ for some arbitrary
parameter $\rho_0$, affect the Brownian fluctuations
along the curve $\mathcal{C}$ for very large $n$? It is
understood here that near the points of the curve (under
consideration), one uses the same scaling as the Airy
process. In this paper, we show that no new process
appears as long as one considers points \be (y,t)\in
\mathcal{C}, \mbox{ such that }
0<t<\frac{1}{1+\rho_0^2}.
 \label{curve}\ee
Observe that $t=t_0=(1+\rho_0^2)^{-1}$ corresponds to the point of tangency $(y_0,t_0)$ of the tangent to the curve  passing through the
   point $(\rho_0\sqrt{n/2},1)$ of the $(t=1)$-axis; i.e.
 \be
 (y_0,t_0)= \left(\frac{\rho_0 \sqrt{2n}}{1+\rho_0^2},
 \frac{1}{ 1+\rho_0^2 }\right)\in  \mathcal{C}
.\label{tangency}
 \ee
   At this point of tangency $(y_0,t_0)$ the fluctuations obey a new
statistics, which we call the {\bf Airy process with
$r$ outliers} $\mathcal {A}^{(r)}(\tau)$
  (in short: {\bf $r$-Airy process}); see Figure 0.2
    In particular, for
the target point $a=\sqrt{n/2}$, the $r$-Airy process
occurs at the the maximum of the edge curve
$\mathcal{C}$; i.e., $\rho_0=1$ and thus $t_0=1/2$.
 Notice, the $r$-Airy process is an extension of the Airy
process; the $r$-Airy process coincides with the Airy
process, when the
  target point $a$ coincides with $0$ or,
  what is the same, when $r=0$. The Airy process
  is stationary, whereas the $r$-Airy process ceases to be stationary.
%
\begin{figure}
\hspace*{-1cm}\includegraphics[width=170mm,height=105mm]{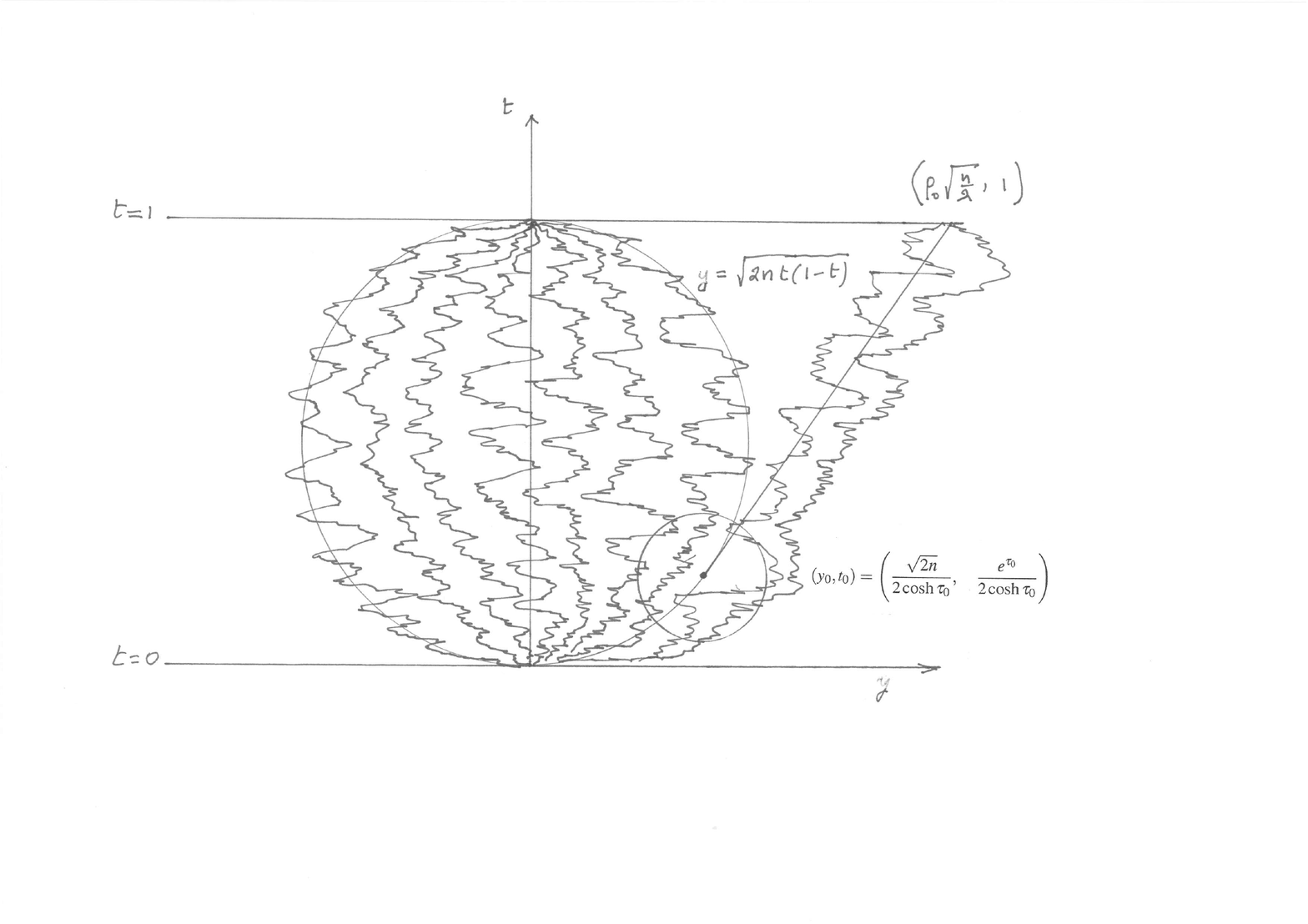}

\vspace*{-2.9cm}

\caption{The $r$-Airy process}\label{fig2}
\end{figure}


Given a target point $a=\rho_0\sqrt{n/2}$,
  the point of tangency of the tangent to the curve
  $\mathcal{C}$, passing through $a$, can be written, in
accordance with (\ref{tangency}),
   \be  (y_0,t_0)=\left(\frac{\rho_0 \sqrt{2n}}{1+\rho_0^2},
 \frac{1}{ 1+\rho_0^2 }\right)=(y_0,t_0)= \left(\frac{\sqrt{2n}    }
   {2\cosh\tau_0 }  ,\quad  \frac{e^{ \tau_0 }}
 {2\cosh \tau_0    }\right)\in \mathcal{C}
  ,\label{y0,t0}\ee
upon introducing a new parameter $\tau_0$, whose
significance as a new time parameter will be clear from
section \ref{sect1} (formula (\ref{time-change})), and
which is defined by
   \be \label{rho0}  e^{ -\tau_0}:=\rho_0=\sqrt{\frac{1-t_0}{t_0}}.\ee
Given a subset $E\subset {\mathbb R}$, the $r$-Airy
process in the new time $\tau$ will be obtained by taking the following limit:%
  \bea
  \lefteqn{ {~~~~~~~~~\BP(\mathcal {A}^{(r)}(\tau)\cap E
  =\emptyset)}}
   \label{r-AiryP}\\&&
  \! := \!
   \lim_{n\rightarrow \iy}
\BP^{(0,\rho_0\sqrt{n/2})}_{Br}\left(\mbox{all}~
x_i\left( \frac{1}
 {1+e^{-2(\tau_0+\frac{\tau}{n^{1/3}})}}\right) \in
  \frac{\sqrt{2n}+\frac{E^c}{\sqrt{2}n^{1/6}}   }
   {2\cosh(\tau_0+\frac{\tau}{n^{1/3}})}
   \right)
 .\no\eea
 Notice that for $\tau=0$ and upon ignoring the set $E^c$, space and time on the right hand side of the formula above
  equal $(y_0,t_0)$ as in (\ref{y0,t0}). In section \ref{sect2}, the limit above will be shown to exist, independently of the parameter $\rho_0$, at the same time establishing some universality.

  In order to state Theorem \ref{Theo1} below, define the functions\footnote{$C$ is a contour running from $\iy
 e^{5i\pi/6}$ to $\iy
 e^{i\pi/6}$, such that $-i\tau$ lies above the contour.}
 \be
 A^{\pm}_r(u;\tau):=
  \int_C   e^{\frac{1}{3} ia^3+iau }
 \left(\mp ia-\tau\right)^{\pm r}
 \frac{da}{2\pi}
  \label{k-Airy0}\ee
and the standard
Airy function $A(u):=A^-_0(u;\tau)=A^+_0(u;\tau)$,  satisfying the differential equation
 $
A'' (x)=xA(x).$
Given these functions, the Airy and $r$-Airy kernels 
 are defined by (\ref{r-kernel}) below and will be used in Theorem \ref{Theo1}:
   \begin{equation}\label{r-kernel}\begin{split}
    K^{(0)}(u,v)&= \int_0^{\iy}
 dw A(w+u)A(w+v)=\frac{A(u)A'(v)-A'(u)A(v)}{u-v}
 \\
 K^{(r)}_\tau(u,v)&=\int_0^{\iy}
 dw A_r^-(w+u; \tau)A_r^+(w+v; \tau). 
 \end{split}\end{equation}
 We now state:

\begin{theorem} \label{Theo1}
Consider non-intersecting Brownian motions as above with
$r$ particles forced to a target point
$\rho_0\sqrt{n/2}>0$ at time $t=1$. For large $n$, the
average mean density of particles as a function of $t$
has its support on a region bounded to the right by the
curve $ \mathcal{C}$ (defined just below (\ref{Normal})). The
tangent line to $ \mathcal{C}$, passing through
$(\rho_0\sqrt{n/2},1)$ has its point of tangency at
$(y_0,t_0)$, given by (\ref{tangency}). Letting $n\rg
\iy$ and given an arbitrary point $(y,t)\in \CR$ with
$t\leq t_0$, a phase transition occurs at $t=t_0$, which
is conveniently expressed in terms of the new parameters
defined by\footnote{Remember from (\ref{rho0}), one has
$\rho_0=e^{-\tau_0}$.} $t=(1+e^{-2\sg})^{-1}$ and
$t_0=(1+e^{-2\tau_0})^{-1}$; one has the following limit:
%
%
   %
%
%
 \bea \lefteqn{\label{0.3}
   \qquad\lim_{n\rightarrow \iy}
\BP^{(0,\rho_0\sqrt{n/2})}_{Br}\left(\mbox{all}~
x_i\left( \frac{1}
 {1+e^{-2(\sigma+\frac{\tau}{n^{1/3}})}}\right) \in
  \frac{\sqrt{2n}+\frac{E^c}{\sqrt{2}n^{1/6}}   }
  {2\cosh(\sigma+\frac{\tau}{n^{1/3}})}
   \right) } \\
  %
 %
 \no\\
&=&\left\{\begin{array}{lll}
 \BP(\mathcal {A}(\tau)\cap E=\emptyset)=
 \det\left(I-K^{(0)}\right)_E
  &~~\mbox{for}\quad  0\leq t<t_0
  &\\
  & \quad\quad(\mbox{ i.e., }0 \leq\sigma<\tau_0) ,\\
 \\
 \BP(\mathcal {A}^{(r)}(\tau)\cap E=\emptyset)=
  \det\left(I-K_\tau^{(r)}\right)_E
  &\mbox{for}\quad t=t_0 \quad&\\
  & \qquad(\mbox{ i.e., }
  \sigma=\tau_0)
 . \end{array}\right. 
   \no%
   \eea
These Fredholm determinants are genuine probability distributions for each value of $\tau$.  Notice that for $0 \leq\sigma<\tau_0 $, the
limit above is independent of the time $\tau$ and the number of outliers $r$, unlike the case $\sigma=\tau_0$.

\end{theorem}

\begin{figure}
\includegraphics[width=95mm,height=90mm]{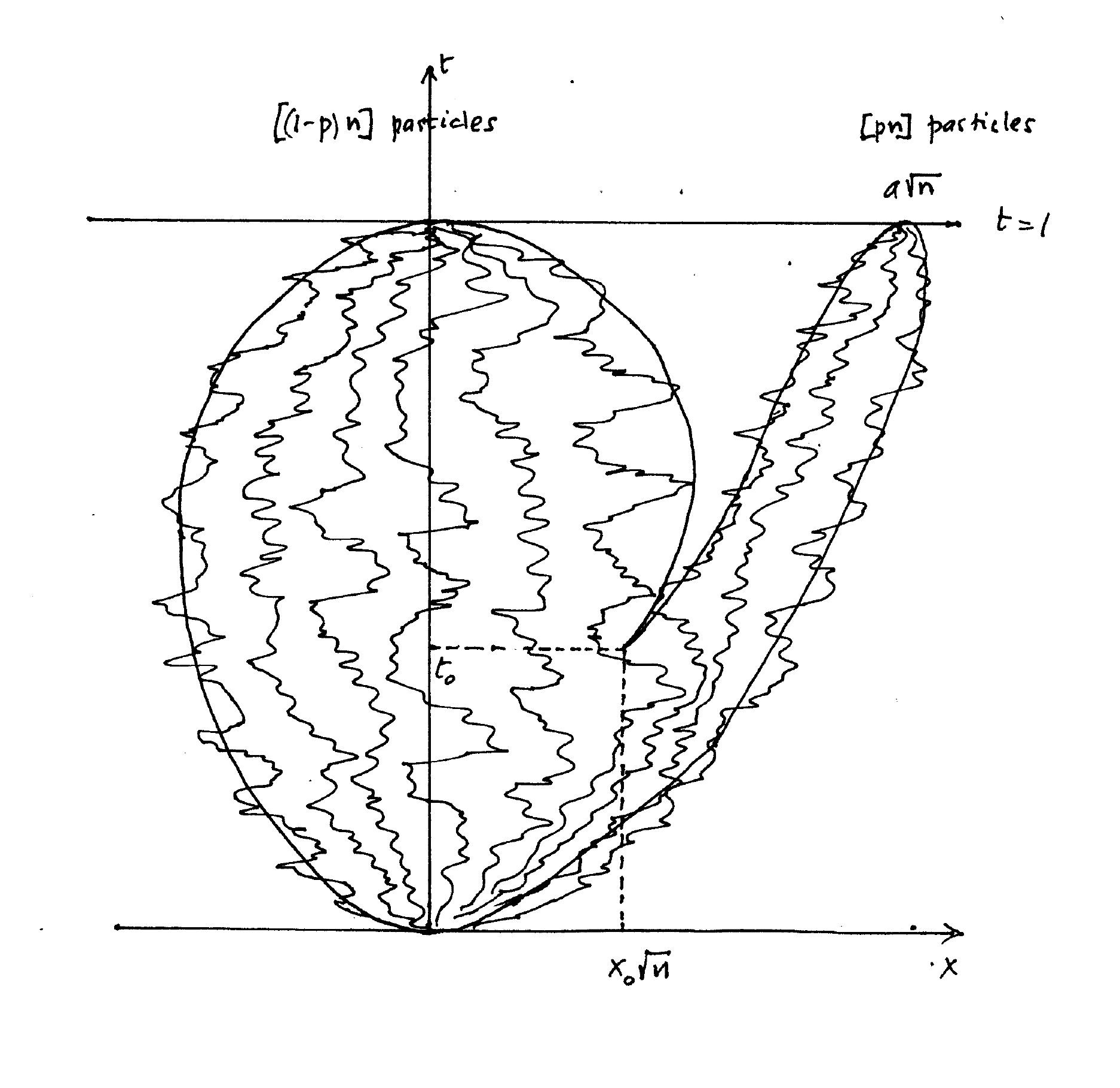}

\vspace*{-1cm}

\caption{Pearcey process}\label{fig3}
\end{figure}


  \bigbreak

As is clear from Figure 0.2 and for $t>t_0$, the $r$
outliers separate from the bulk as a group; thus after
an appropriate shift to take into account their
displacement from the curve $\mathcal{C}$ and using a
different scaling, they behave probabilistically like
the $r$ eigenvalues of an $r \times r$ matrix in the GUE
ensemble.

  The $r$-Airy process can also be viewed as an {\bf interpolation}
  between the Airy and Pearcey processes. The {\bf Pearcey process}
  \cite{TW-Pearcey, Okounkov, AvM-Pearcey,AvM-Pearcey-joint} is
  defined as the limit of non-intersecting Brownian motions (for large
  $n$), all leaving from $0$ at time $t=0$, with $(1-p)n$ paths forced
  to end up at $0$ and $pn$ paths forced to end up at $a\sqrt{n/2}$ at
  time $t=1$; see Figure 0.3. The Pearcey process then
describes this Brownian cloud of particles (for $n\rg
\iy$) near the time $t$ of bifurcation, where the
support of the average mean density goes from one
interval into two intervals, with stretched space and
time. Then the boundary of the support of the average
mean density of particles in $(y,t)$-space has a cusp.
In section \ref{inter} it will be shown sketchily how
the location of the cusp, when the proportion $p$ of
particles tends to $0$ as $r/n$, tends to the precise
place $(y_0,t_0)$ where the $r$-Airy process occurs.

\begin{remark} To simplify the notation, one often writes
 \be
\BP(\sup\mathcal {A}^{(r)}(\tau)\leq x)
: =\BP( \mathcal{A}^{(r)}(\tau)
 \cap (x,\iy)=\emptyset)\label{notation}
  .\ee
\end{remark}





\begin{remark} The joint probabilities for the $r$-Airy
processes for a finite number of times can be defined in a similar way and lead to a
matrix Fredholm determinant, to be discussed in a later
paper.
\end{remark}

This phenomenon is closely related to statistical work
by Baik-Ben Arous-P\'ech\'e \cite{BBP},
Baik \cite{Baik} and P\'ech\'e \cite{Peche}.
 ~Indeed, consider a (complex) Gaussian population
$\overrightarrow{y}\in \BC^N$, with covariance matrix
$\Sigma$. Given $M$ samples
$\overrightarrow{y_1},\ldots,\overrightarrow{y_M}$, the
(centered) sample covariance matrix $S:=\frac{1}{M}
X\bar X^{\top}$, where
$$
 X=\left(
 \overrightarrow{y_1}-\frac{1}{M} \sum_1^M
 \overrightarrow{y_i},\ldots,
 \overrightarrow{y_M}-\frac{1}{M} \sum_1^M
 \overrightarrow{y_i}\right)
 ,$$
 is a positive definite matrix and is an estimator of the true covariance matrix
$\Sigma$. One may test the statistical hypothesis that
$\Sigma=I$ or that $\Sigma$ has all eigenvalues $=1$,
except for a few outliers. When all the eigenvalues of
$\Sigma$ are $=1$, then the limit distribution of the
largest eigenvalue of the sample covariance matrix $S$,
for $N$ and $M$ tending to $\iy$ in the same way, is
given by the Tracy-Widom distribution. Then
Baik-BenArous-P\'ech\'e \cite{BBP} noticed that, this is
still so, if the eigenvalues of $\Sigma$ contain some
outliers, which are not too large. There is a critical
point at which and beyond which the largest eigenvalue
of the sample covariance matrix $S$ will be different
from the Tracy-Widom distribution. At this point of
phase transition Baik-BenArous-P\'ech\'e \cite{BBP} have
found an Airy-type distribution in $x$, which is given
by the Fredholm determinant $$ \det
{(I-K_\tau^{(r)})}_{_{(x,\iy)}}\Bigr|_{\tau=0} ,$$ where
$r$ denotes the number of eigenvalues of $\Sigma$ that
are equal to $1+\gamma^{-1}$, while all the others are
$=1$; $\gamma$ is such that $M/N=\gamma^2$ for $M$ and
$N$ very large. This distribution was further
generalized in \cite{BBP} to the case where $\tau \neq
0$ in the kernel $K_\tau^{(r)}$. Baik proved in
\cite{Baik} that the Fredholm determinant of
$K_\tau^{(r)}$ is a genuine probability distribution. In
the statistical problem above, the covariance matrix $S$
is positive definite and therefore its eigenvalues
satisfy Laguerre-type distributions; this idea was
extended to GUE-type distributions by S. P\'ech\'e
\cite{Peche}. In the present paper one finds that the
shift $\tau$ appearing in the kernel $K_\tau^{(r)}$ is
precisely the rescaled time of the non-intersecting
Brownian motion model! The arguments will appear in
sections \ref{sect1} and \ref{sect2}.


This paper also shows that the probability of the
$r$-Airy process or, what is the same, the Fredholm
determinant $\det\left(I-K_\tau^{(r)}\right)_{(x,\iy)}$
satisfies a non-linear PDE in $x$ and $\tau$, depending
on the number $r$ of outliers, as established in section
\ref{sect6}:


\begin{theorem}\label{Theo2}
 The logarithm of the probability
 $$Q(\tau,x):=\log\BP(\sup\mathcal {A}^{(r)}(\tau)\leq x)
 = \log\det {(I-K_\tau^{(r)})}_{_{(x,\iy)}}$$
satisfies the following non-linear PDE\footnote{The
Wronskian $\{f,g\}_x$ with regard to the variable $x$ is
defined as $f'g-fg'$.}, with both, the function
$Q(\tau,x)$ and the PDE, being invariant\footnote{The
invariance under the involution is obvious for the
equation; for the function $Q(\tau,x)$, see Lemma 7.5.}
under the involution $(\tau,x,r)\rg (-\tau,x,-r)$,
\be \label{PDE0}\ee
   \bean
 &&\left\{\frac{\pl^3 Q}{\pl \tau\pl x^2}
 ,
  \displaystyle{\left[
   \begin{array}{l} \left( r-{\frac {\partial ^{2}Q}{\partial \tau\partial x}} \right) ^{2}
\left( {\frac {\partial ^{3}Q}{\partial {x}^{3}}}
 \right) +
\left( r-{\frac {\partial ^{2}Q}{\partial \tau\partial x}} \right)
\frac{\pl}{\pl \tau}\left(2\tau\frac{\pl^2Q}{\pl x\pl \tau}+\frac{\pl^2Q}{\pl \tau^2}\right)
\\
  +  {\frac {\partial ^{3}Q}{\partial \tau\partial x^2}}    \left( 2r{\frac {\partial ^{2}Q}{
\partial {x}^{2}}} +2{\frac {\partial Q}
{\partial \tau}}  -xr\right)
+\frac{\pl}{\pl x}\left(\frac{1}{4}(\frac {\partial ^{2}Q}{\partial {\tau}^{2}})^2+\tau
 \frac {\partial ^{2}Q}{\partial {\tau}^{2}} \frac {\partial ^{2}Q}{\partial {\tau}\partial x} \right)
 \end{array}\right]}
  \right\}_x
  \\\no\\
\lefteqn{~~~~~~~ -\frac{1}{2}\left( {\frac {\pl
^{3}Q}{\pl \tau\pl x^2}} \right) ^{2}\left({\frac
{\partial ^{3}Q}{\partial {\tau}^{ 3}}}
 -4\frac{\pl^2Q}{\pl \tau\pl x}  \frac {\pl ^{3}Q}{\pl
x^3}\right)=0,} 
\no\eean
with ``initial condition", given by the log of the
Tracy-Widom distribution,
 \be
 Q_0(x):=\lim_{\tau\rg -\iy} Q(\tau,x):=
 \log\BP(\sup\mathcal {A}(\tau)\leq x)
  =-\int_x^{\iy}(\al-x)g^2(\al)d\al, \label{ic}\ee
where
  $g(\al)$ is the Hastings-MacLeod solution of Painlev\'e
  II,
  \be
  g''=\al g+2g^3,~~~~\mbox{with}~~~ g(\al)\cong
  \frac{
  e^{-\frac{2}{3}  \al^{\frac{3}{2}}}}{2\sqrt \pi \al^{1/4}}
\mbox{~~~for~~~}\al\nearrow \infty .\label{Painleve}\ee

\end{theorem}


\begin{remark} Obviously, the PDE (\ref{PDE0})
has the following structure
  \be
 \frac{\pl}{\pl
x}\left(\frac{\ldots}{\frac{\pl^{3}Q}{\pl \tau
  \pl x^2}}\right)
=  \frac{\partial^{3} Q}{\partial \tau^3 } -4
\frac{\partial^{2} Q}{\partial \tau
 \partial x} \frac{\partial^{3} Q}{ \partial x^3}
~.%
\label{structure}\ee 
\end{remark}


\begin{remark} The following simple recipe gives the PDE for $\BP(\mathcal {A}^{(r)}(\tau)\cap
E=\emptyset)$ for a general set $E=\cup_1^\ell
[x_{2i-1},x_{2i}]$, replacing the PDE (\ref{PDE0}) for
$E=(x,\iy)$; indeed, perform the replacements  
  $$\frac{\pl}{\pl x}\mapsto\sum_i \frac{\pl}{\pl
x_i}~~\mbox{ and }~~ x\left(\frac{\pl}{\pl
x}\right)^k\mapsto\left(\sum_i x_i\frac{\pl }{\pl
x_i}\right)\left(\sum_i  \frac{\pl}{\pl
x_i}\right)^{k-1},$$ with the understanding that
differentiation must always be pulled to the right.

\bigbreak

Although the average mean density of the particles is insensitive to the presence of outliers,
the presence of the $r$ particles forced to reach the
target $a>0$ at $t=1$ is already felt, when $t\rg 0$ in
the $t$-scale; that is when $\tau\rg -\iy$ in the
$\tau$-scale. The net effect is that it pulls the edge
of the cloud of particles in the average towards the right to first
order like $|r/\tau|$; i.e., the more so when $r$ gets
large; this edge then behaves like the Airy process
shifted in space, up to and including order $1/\tau^4$.The
PDE in Theorem \ref{Theo2} is a convenient instrument to
extract the remote past asymptotics, as shown in
Section~\ref{sect7} and stated in the theorem below.
\end{remark}

\begin{theorem}\label{Theo3}

The PDE with the initial condition $Q_0(x)$ as in
(\ref{ic}) admits the asymptotic solution, for $\tau\rg
-\iy$, of the form \bean
 Q(\tau,x)&=&\sum_0^{\iy}\frac{Q_i(x)}{\tau^i}
   =
   Q_0\left((x+\frac{r}{\tau})(1+\frac{r}{3\tau^3})+
    \frac{r^2}{4\tau^4}\right)
   + \frac{r}{5\tau^5}{\mathcal{F}}_5
   +O(\frac{1}{\tau^6}).\eean
For the probability itself, one has\footnote{with 
 $ \mathcal{F}_5:=
  x^2Q_0'+4xQ_0+Q_0^{\prime 2}
   +10\int_x^{\iy}Q_0
    -
   6\int_x^{\iy}dy\int_y^{\iy}du Q_0^{\prime\prime 2}
   .$ }
   \bean
   \lefteqn{\BP(\sup\mathcal {A}^{(r)}(\tau)\leq x)}
  \\
   &=&
    \BP\left(\sup\mathcal {A}(\tau)\leq (x+\frac{r}{\tau})(1+\frac{r}{3\tau^3})+
    \frac{r^2}{4\tau^4}\right)
    \left(1+\frac{r}{5\tau^5}\mathcal{F}_5
     +O(\frac{1}{\tau^6})\right)
     .\eean
     For $\tau\rg -\iy$, the mean and variance of the right edge of the process
      behave as
      \bean
      {\mathbb E}
       (\sup \mathcal {A}^{(r)}(\tau))&=&
 {\mathbb E}(\sup\mathcal {A}^{(0)}(\tau))\left(1-\frac{r}{3\tau^3}
 \right)
 -\frac{r}{\tau}-\frac{r^2}{4\tau^4} +O(\frac{1}{\tau^5})
      \\
       \mbox{\em var}(\sup \mathcal {A}^{(r)}(\tau))
       &=&
 \mbox{\em var}(\sup\mathcal {A}^{(0)}(\tau))
 \left(1-\frac{2r}{3\tau^3}\right)
  +O(\frac{1}{\tau^5}).\eean

\end{theorem}

The probability (\ref{Brownian target}) is related, via
a change of variables, to a Gaussian matrix model with
external potential, \bean\BP_{n} (\alpha , \tilde{E})
 &=& \frac{1}{Z_n}\int_{\HR_n(\tilde E)}
 dM e^{-\frac{1}{2}\Tr (M^2-2AM)}
 ,\eean
where $\HR_n(\tilde E)$ denotes the Hermitian matrices
with eigenvalues in $\tilde E\subset {\mathbb R}$ and with a
diagonal matrix \be
 A:= \left(\begin{array}{cccccc}
   \al\\
   &\ddots& & & &{\bf O}\\
   & &\al\\
   & & &0\\
   &{\bf O}& & &\ddots\\
   & & & & &0
\end{array}\right)\begin{array}{l}
\updownarrow r\\
\\
\\
\\
\updownarrow n-r
\end{array}.
  \label{diagmatrix}\ee
  The relationship between the Karlin-McGregor
  non-intersecting Brownian motions and matrix models
   has been developped
  by Johansson \cite{Johansson1} and, for the Gaussian
  matrix model with external potential, by
  Aptekarev-Bleher-Kuijlaars \cite{AptBleKui}.
  The latter model has come up in many other situations;
  among others, see
  \cite{Pastur,Brezin2,Brezin3,Brezin4,Zinn1,Zinn2,Johansson2,
  BleKui1, TW-Pearcey, AvM-Pearcey}.
  Following the method of Adler-van Moerbeke \cite{AvM-Pearcey}, using the multi-component KP hierarchy (section \ref{sect:KP}) and the
  Virasoro constraints (section \ref{sect:Vir}),
  it is shown in section
  \ref{sect:PDE} that this matrix model satisfies a
  non-linear PDE in $\al$ and the boundary points of $\tilde E$.
  The PDE of Theorem \ref{Theo2} then
  follows from making an asymptotic analysis on that
  PDE. The asymptotic behavior of the $r$-Airy process
  for $\tau\rg -\iy$ follows from solving the
  partial differential equation near $\tau=-\iy$ for the
  Airy-kernel initial condition; this process introduces
  some free constants, which then can be determined by the
  asymptotic properties of the $r$-Airy kernel for $x \rg \iy$. 
   It is an interesting question whether these results
  can be deduced via the Riemann-Hilbert methods, in the
  style of \cite{Deift}.



\section{A constrained Brownian motion with a few
outliers}\label{sect1}

The purpose of this section is to motivate the scaling
limit (\ref{r-AiryP}) leading to the definition of the
$r$-Airy process. The Airy process was originally
defined as an edge scaling limit of Dyson's
non-intersecting Brownian motions, in the same way that
the Tracy-Widom distribution was obtained as an edge
scaling limit of GUE. The non-intersecting Brownian
motions (\ref{Brownian target}) with target point $a=0$
can be transformed to the Dyson Brownian motion, as will
be explained in this section. This transformation will
be used in the case of a $a\neq 0$ target. In the course
of doing this, one must consider a Gaussian matrix model
with external potential. 
\subsection{Dyson Brownian motion} 
The Dyson
process (\cite{Dyson}) describes the motion of the eigenvalues
$\lb_i(t')$ of an $n\times n$ Hermitian random matrix $B(t')$ whose
real and imaginary entries perform independent
Ornstein-Uhlenbeck processes, given that the initial
distribution is given by invariant measure for the
process, namely
 $$
  Z^{-1}\int_{\mathcal {H}_n(E')} dB e^{-\Tr B^2}
  .$$
Then the process is stationary and its probability
distribution at any time is given by
\be\BP_{Dy}(\mbox{all}~\lb_i(t')\in E')=
 Z^{-1}\int_{\mathcal {H}_n(E')} dB e^{-\Tr B^2}=
 Z^{-1}\int_{\mathcal {H}_n(\sqrt{2}E')} dM e^{-\frac{1}{2}\Tr M^2}
 .\label{Dyson1}\ee
The probability for two times $0<t_1'<t_2'$ is then
given by (set $c'=e^{-(t'_2-t'_1)}$)
%
%
 \bea
\lefteqn{\BP_{Dy}(\mbox{all~}\lb_i(t'_1)\in E'_1,
       \mbox{~all~}\lb_i(t'_2)\in E'_2)} 
\no\\ \no\\
&=& P(\mbox{all~}(B(t'_1)
 \mbox{-eigenvalues})\in E'_1,\mbox{~all~}(B(t'_2)
  \mbox{-eigenvalues})\in E'_2) 
\no\\ \no\\
&=&  \int\!\!\!\int_{\mbox{\scriptsize{all~}}B _1
-\mbox{\scriptsize{eigenvalues~}}\in
E'_1\atop{\mbox{\scriptsize{all~}}B _2
 -\mbox{\scriptsize{eigenvalues~}}\in E'_2}}
 Z^{-1}\frac{dB_1dB_2}{(1-c'^2)^{n^2/2}}
e^{-\frac{1}{1-c^{\prime 2}}\Tr(B_1^2+B_2^2-2c'B_1B_2)}\no\\
\no\\
&=&  \int\!\!\!\int_{
 M_1\in
\mathcal {H}_n(\frac{\sqrt{2}E'_1}{\sqrt{1-c'^2}}) \atop{
M_2 \in \mathcal {H}_n (\frac{\sqrt{2}E'_2}{ \sqrt{1-c'^2}}}
)}
Z^{'-1}dM_1dM_2~e^{-\frac{1}{2}
   \Tr(M_1^2+M_2^2-2c'M_1M_2)}
\label{Dyson2}\eea 
%
%
\subsection{Constrained Brownian motion with target
$a>0$} As in the introduction, consider the $n$
non-intersecting \mbox{Brownian} particles on ${\mathbb R}$, all
starting from the origin at time $t=0$, where among
those paths, $1 \leq r \leq n$ are forced to end up at
the target $a > 0$, while the $(n-r)$ other paths return
to the position $x=0$ at time $t=1$. Remember
$\BP_{Br}^{0a}$ denotes the transition probability. Then
the probability that all the particles belong to some
window $E \subset {\mathbb R}$ at a given time $0 < t < 1$, can
be expressed in terms of a Gaussian matrix model with an
external potential, using the Karlin-McGregor formula
\cite{Karlin} for non-intersecting Brownian motions (see
\cite{Grabiner, Johansson1, AptBleKui, AvM-Pearcey}),
itself involving the Brownian transition probability
(\ref{Normal}),
namely:
%
%
%
%
\bean    
{ \BP_{Br}^{0 a} \left( \mbox{all
$x_j(t)\in E$} \right)
 }
 &:=&
\BP \left(\begin{tabular}{c|c}
& all $x_j(0) =0$\\
all $x_j(t)\in E$ &
 $r$ right paths end up at $a$ at $t=1$\\
  &
$(n-r)$ paths end up at $0$ at $t=1$
\end{tabular}\right)
\\
 &=&\lim_{\begin{array}{c}
  \mbox{  all $~\gamma_i\rightarrow 0$}
  \\
 \mbox{$\delta_1, \ldots , \delta_{n-r}\rightarrow 0$}\\
  \mbox{$\delta_{n-r+1}, \ldots , \delta_{n}
  \rightarrow a$}
  \end{array}} \int_{E^n}\frac{1}{Z_{n}(\ga,\delta)}
  \prod_1^n dx_i\no\\
   \no\\ &&~~~\hspace*{0cm}
  \det\left(p(t;\gamma_i,x_j)\right)_{1\leq i,j\leq
n}\det(p(1-t;x_{i'},\delta_{j'}))_{1\leq i',j'\leq n}
  ,
 \\
  &=& \BP_{n}\left(a \sqrt{\frac{2t }{1-t}};
E\sqrt{\frac{2}{t(1-t)}}\right)
  \eean
where\footnote{$\Dt_r(x)$ denotes the Vandermonde
determinant, with regard to the $r$ variables
$x=(x_1,\ldots,x_{r})$ and $\Dt_n(x,y)$ denotes the
Vandermonde determinant, with regard to the $n$
variables $x=(x_1,\ldots,x_{r})$ and
$y=(y_1,\ldots,y_{n-r})$.}
 \bean\BP_{n} (\alpha ,
\tilde{E})
 &=& \frac{1}{Z_n}\int_{\HR_n(\tilde E)}
 dM e^{-\frac{1}{2}\Tr (M^2-2AM)}
 \\
  &=& \frac{1}{Z'_{n}}
  \int_{ \tilde E  
   ^{n}
   }
   \Dt_{n}(x ,y )
 \left(\Dt_r(x )  \prod^r_{i=1}
  e^{
 -\frac{x_i^{   2}}{2}  +  { \al x_i
 }}
 dx_i\right)
\left(\Dt_{n-r}(y  )  \prod^{n-r}_{i=1}e^{ -\frac{y_i^{
 2}}{2} }dy_i  \right) , \eean
with $A$ as in (\ref{diagmatrix}).
  %
%
In short, the conditioned Brownian motion is related to
a Gaussian matrix model with an external potential $A$
as follows:
%
%
\begin{equation}
\BP_{Br}^{0a}\left(\mbox{all $x_j(t)\in E$}\right)
 = \BP_{n} \bigl(\al, \tilde E 
\bigr) ~\mbox{  with  }~
 \tilde E=E  \sqrt{\frac{2}{t\,( 1-t)}} ,~~\al=a
\sqrt{\frac{2\,t}{1-t}} .\label{CBM1}
\end{equation}
%
The joint probability for the constrained Brownian
motion at two times is related to a chain of two
Gaussian matrix models (see \cite{AvM-Pearcey-joint})
with an external potential $A$, which again by
Karlin-McGregor reads as follows: \be
  {\BP_{Br}^{0a}(\mbox{all~}x_i(t_1)\in
E_1,\mbox{all~}x_i(t_2)\in  E_2)}
=\BP_n\bigl(\alpha;c';\tilde
E_1,\tilde E_2\bigr),%
\label{CBM2}\ee
 where ($A$ is the same diagonal matrix as in
 (\ref{diagmatrix}))
  \bean
 \BP_n\left(\alpha;c';\tilde
E_1,\tilde E_2\right)&=&
 \frac{1}{Z_n}\int_{ \mathcal{H}_n(\tilde E_{1})\times\mathcal {H}_n(\tilde E_{2})
}\!\!\!\!\!\!\!e^{-\frac{1}{2}\Tr
(M_1^2+M_2^2-2c'M_1M_2-2AM_2)} dM_1  dM_2
 ,\eean
%
%
with
$$
\tilde  E_1=E_1 \sqrt{\frac{2t_2}{(t_2-t_1)t_1}} ,~~~~
 \tilde  E_2=E_2 \sqrt{\frac{2(1-t_1)}{(1-t_2)(t_2-t_1)}}
  $$
  \be
  c'=\sqrt{\frac{(1-t_2)t_1}{(1-t_1)t_2}}=
   \frac{\sqrt{\frac{t_1}{1-t_1}}}{\sqrt{\frac{t_2}{1-t_2}}}
 ,~~~~~~
 \al=a\sqrt{\frac{2(t_2-t_1)}{(1-t_2)(1-t_1)}}
  \label{map1}\ee
\subsection{Comparing the Dyson and constrained Brownian
motions with target $a=0$}
From the identities (\ref{Dyson1}) and (\ref{CBM1}), one
deduces the following set of identities for one time
$t'$ (in the time $t'$ of the Dyson process), setting
the point $a=0$,
 \be
\BP_{Dy}\left(\mbox{all \,$\lb_j(t')\in
E'$}\right)=\BP_n(\al, \tilde E)\Bigr|_{\al =0} =
\BP_{Br}^{0a}\left(\mbox{all \,$x_j(t)\in E$}\right)
\Bigr|_{a =0}\label{iso} ,\ee
  where
$$
\sqrt{2}E'=\tilde E=E \sqrt{\frac{2}{t\,( 1-t)}} .$$
Similarly, for two times $t'_1$ and $t_2'$, one deduces
from the identities (\ref{Dyson2}) and (\ref{CBM2}),
 \bean 
 {
\BP_{Dy}(\mbox{all~}\lb_i(t'_1)\in E'_1,
       \mbox{~all~}\lb_i(t'_2)\in E'_2)}
       &=&
       \BP_n\left(\alpha;c';\tilde
E_1,\tilde E_2\right)\Bigr|_{\al =0} \\
 &=&
 {\BP_{Br}^{0a}(\mbox{all~}x_i(t_1)\in
E_1,\mbox{all~}x_i(t_2)\in  E_2)}\Bigr|_{a =0}
 ,\eean
where, as follows from (\ref{Dyson2}) and (\ref{map1}),
one deduces
\bea
\frac{\sqrt{2}E'_1}{\sqrt{(1-c'^2)}}&=&\tilde E_1=E_1
\sqrt{\frac{2t_2}{(t_2-t_1)t_1}}
\no\\
\frac{\sqrt{2}E'_2}{\sqrt{(1-c'^2)}}&=&\tilde
E_2=E_2 \sqrt{\frac{2(1-t_1)}{(1-t_2)(t_2-t_1)}}
\label{A}\eea with $$
c'=\frac{e^{t'_1}}{e^{t'_2}}
   =\frac{\sqrt{\frac{t_1}{1-t_1}}}{\sqrt{\frac{t_2}{1-t_2}}}
 .$$
%
%
Thus the processes are related by a clock change
 \be
e^{t'_i}= \sqrt{\frac{t_i}{1-t_i}}\Longleftrightarrow
t_i=\frac{1}{1+e^{-2t'_i}}
  \Longleftrightarrow
  \frac{1}{\sqrt{t_i(1-t_i)}}=2\cosh
  t'_i,\label{time-change}\ee
 and thus $
1-c^{\prime 2} =\frac{t_2-t_1}{t_2(1-t_1)}.$
Comparing extremities in (\ref{A}), one finds
$$
E'_1=E_1 \sqrt{\frac{2t_2}{(t_2-t_1)t_1}}
 \sqrt{(1-c'^2)/2}=E_1\frac{1}{\sqrt{t_1(1-t_1)}}
$$
$$
E'_2=E_2 \sqrt{\frac{2(1-t_1)}{(1-t_2)(t_2-t_1)}}
 \sqrt{(1-c'^2)/2}=E_2\frac{1}{\sqrt{t_2(1-t_2)}}
.$$
This fact, combined with (\ref{time-change}), yields
\be E_i=E'_i \sqrt{t_i(1-t_i)}=\frac{E'_i}{2\cosh t'_i}
.\label{space-change} \ee
%
So, summarizing (\ref{time-change}) and
(\ref{space-change}), one has the relation between the
parameters of the Dyson process and constrained Brownian
motions,
\be e^{t'_i}= \sqrt{\frac{t_i}{1-t_i}}
 ,~~~~E_i\sqrt{\frac{2}{t_i(1-t_i)}}=\sqrt{2}E'_i
 .\label{space-time-change}
 \ee




\begin{theorem}\label{Theo:1.1}
  (Tracy-Widom \cite{TW-Airy},
  Adler-van Moerbeke \cite{AvM-Airy-Sine})
Taking a limit on the Dyson process, in an appropriate
 time and space scale,
 one finds the Airy process, which is stationary:
$$
\lim_{n\rightarrow \iy}
 \BP_{Dy}\left(\mbox{all}~ \lb_i\left(\frac{\tau}{n^{1/3}}\right) \in
  (-\iy, \sqrt{2n}+\frac{x}{\sqrt{2}
  n^{1/6}})\right)=\BP(\mathcal {A}(\tau)\leq x)=F(x),
$$
 $F(x)$ being the Tracy-Widom distribution.
Similarly the limit of the joint probability for the
Dyson process yields the joint probability for the Airy
process:
 $$ \lim_{n\rightarrow \iy}
 \BP_{Dy}\left(\begin{array}{l}
   \mbox{all}~ \lb_i\left(\frac{\tau_1}{n^{1/3}}\right) \in
  (-\iy, \sqrt{2n}+\frac{x_1}{\sqrt{2}
  n^{1/6}}) \\
   \mbox{all}~ \lb_i\left(\frac{\tau_2}{n^{1/3}}\right) \in
  (-\iy, \sqrt{2n}+\frac{x_2}{\sqrt{2}
  n^{1/6}})\end{array}\right)
  =
   \BP\left(\begin{array}{l} \mathcal {A}(\tau_1)\leq x_1
      ,\\
     \mathcal {A}(\tau_2)\leq x_2 \end{array}  \right)
    . $$
     Its logarithm  (setting
$s=(\tau_2-\tau_1)/2$)
$$
 H(s;x,y):=
  \log \BP\left(\mathcal {A}(\tau_1) \leq {x+y} , \mathcal {A}(\tau_2) \leq
   {x-y} \right),
 $$
satisfies the Airy PDE %
\be
  2s\frac{\pl^3 H}{\pl s\pl x \pl y}= \bigl( 2s^2\frac{\pl}{\pl y}-
y\frac{\pl}{\pl x}\bigr)\bigl(\frac{\pl^2 H}{\pl y^2}-
 \frac{\pl^2 H}{\pl x^2}\bigr)+ \left\{ \frac{\pl^2 H}{\pl x\pl y},
  \frac{\pl^2 H}{\pl x^2} \right\}_x.
 \label{Airy PDE1}\ee
\end{theorem}


\begin{corollary}\label{Cor:k-Airy} Taking an appropriate scaling limit on the constrained Brownian
motion, we have
 \bean { \lim_{n\rightarrow \iy}
\BP^{0a}_{Br}\left(\mbox{all}~
x_i\left(\frac{1}{1+e^{-2\tau/n^{1/3}}}\right) \in
  \frac{(-\iy, \sqrt{2n}+\frac{x}{\sqrt{2}
  n^{1/6}})}{2\cosh( \tau/n^{1/3})}\right)\Bigr|_{a=0}
  }=\BP(\mathcal {A}(\tau)\leq x)
  \eean
 and similarly for two times. \end{corollary}
\begin{proof} This statement follows immediately from Theorem
1.1 and the correspondence (\ref{iso}), upon using the clock
change (\ref{time-change}) and the space change
(\ref{space-change}), with the appropriate time and
space scalings of Theorem \ref{Theo:1.1}.\end{proof}

\bigbreak

\noindent Remembering the definition
 $$\BP_{n} (\alpha ,
\tilde{E})
 = \frac{1}{Z_n}\int_{\HR_n(\tilde E)}
 dM e^{-\frac{1}{2}\Tr (M^2-2AM)},$$
  for the diagonal matrix $A$ as in (\ref{diagmatrix}),
   we now state a Theorem of P\'ech\'e, which is closely
  related to the multivariate statistical problem
  mentioned in the introduction; see \cite{BBP}.

\begin{theorem}  \label{Peche}(P\'ech\'e \cite{Peche})
\be \lim_{n\rightarrow \iy}
 \BP_n\left(\rho\sqrt{n},
 (-\iy,2\sqrt{n}+\frac{x}{n^{1/6}})\right)=
 \left\{
  \begin{array}{l}
  F(x)\mbox{  for} ~~\rho <1\\ \\
  F^{(r)}(x)\mbox{  for} ~~\rho =1
  \end{array}
  \right.
  \label{PecheF}\ee
 where
 \be
 F^{(r)}(x)=\det (I-K^{(r)}_\tau{\chi_{
 }}_{(x,\iy)}(y))\Bigr|_{\tau=0}
 \label{r-Airy1}\ee
  is the Fredholm determinant of the kernel (see
  (\ref{k-Airy0}))
    \be
 K^{(r)}_\tau(x,y)=\int_0^{\iy}
 du A_r^-(x+u; \tau)A_r^+(y+u; \tau).
 \label{r-Airy kernel}\ee
 For $r=0$, (\ref{r-Airy1}) yields the
 Tracy-Widom distribution $F(x)$.

\end{theorem}

 The Airy process will now be deformed in a way
  which is compatible with P\'ech\'e's Theorem, especially concerning the target point $a$,
  $$
\BP^{(0a)}_{Br}\left(\mbox{all}~
x_i\left(\frac{1}{1+e^{-2\tau/n^{1/3}}}\right) \in
  \frac{(-\iy, \sqrt{2n}+\frac{x}{\sqrt{2}
  n^{1/6}})}{2\cosh( \tau/n^{1/3})}\right)
  .$$
   Since, from (\ref{CBM1}),
$$
 \BP_{Br}^{0 a} \left( \mbox{all $x_j(t)\in E$} \right)
=\BP_{n}\left(a \sqrt{\frac{2t }{1-t}};
E\sqrt{\frac{2}{t(1-t)}}\right)
$$
 holds, we have
\bean
 \lefteqn
 {\BP^{(0a)}_{Br}\left(\mbox{all}~
x_i\left(\frac{1}{1+e^{-2\tau/n^{1/3}}}\right) \in
  \frac{ (-\iy,\sqrt{2n}+\frac{x}{\sqrt{2}
  n^{1/6}})}{2\cosh( \tau/n^{1/3})}\right)
 }  \hspace*{4cm} \\
    && ~~~~~~~~~~~~~~~~~~~~~~~~~~~~~~
  =\BP_n\left(a\sqrt{2}e^{\tau/n^{1/3}};
 (-\iy,2\sqrt{n} +\frac{x}{n^{1/6}})
 \right).\eean
  Comparing this formula at time $t=1/2$, or what is the same at $\tau=0$,
   with formula (\ref{PecheF})
 suggests the choice \be a=\rho\sqrt{\frac{n}{2}}.\label{choice a}\ee
  We thus define the {\bf Airy process with $r$ outliers}
  (in short: $r$-Airy process) $\mathcal {A}^{(r)}(\tau)$
  by means of the Airy scaling as in Corollary
  \ref{Cor:k-Airy} and with the choice of $a$ above for $\rho=1$. Notice
  from (\ref{CBM1}) and using (\ref{space-time-change}),
  this can also be expressed in terms of the matrix
  model with external potential:
  \bean {\BP(  \mathcal {A}^{(r)}(\tau)\cap
    E=\emptyset)
  }
  &:=& \lim_{n\rightarrow \iy}
\BP^{(0,\sqrt{n/2})}_{Br}\left(\mbox{all}~
x_i\left(\frac{1}{1+e^{-2\tau/n^{1/3}}}\right) \in
  \frac{ \sqrt{2n}+\frac{E^c}{\sqrt{2}
  n^{1/6}} }{2\cosh( \tau/n^{1/3})}\right)
  \\ \\
  &=&
   \lim_{n\rightarrow \iy}
  \BP_{n}\left({}\sqrt{n}e^{\tau /n^{1/3}};
  2\sqrt{n} +\frac{E^c}{n^{1/6}}   \right).%
  \eean
  In a similar way, one defines the joint probability
  of $\mathcal {A}^{(r)}(\tau)$
  for any number of times. The next section deals with
  this limit expressed in terms of a Fredholm
  determinant.


 \section{The existence of the limit to the $r$-Airy kernel}\label{sect2}

The first part of this section deals with a sketch of the proof of
 Theorem \ref{Theo1}; details and rigor can be found in \cite{BBP} and
 \cite{Peche}. In the second part,
 the $r$-Airy kernel (\ref{r-Airy kernel})
  will be expanded for large time $\tau$.

\begin{proof}[Proof of Theorem \ref{Theo1}]
From the explicit Brownian motion transition
probability, one shows (see Johansson \cite{Johansson1}
and Tracy-Widom \cite{TW-Pearcey})

$$
\BP_{Br}^{0a}\left(\mbox{all $x_j(t)\in E$}\right) =
  {\det(I-H_n^{(r)})}_{L^2(E^c)}
  $$
  where
 \newline\noindent$
H_n^{(r)}(x,y)dy =
  $
\be -\frac{dy}{2 \pi^2 (1-t)} \int_{ \mathcal{D}} dz
\int_{\Gamma_{L}} dw ~e^{ -\frac{ tz^2}{1-t}  +
\frac{2xz}{1-t} + \frac{tw^2}{1-t} - \frac{2wy}{1-t} }
\bigg(\frac{w}{z}\bigg)^{n-r}
\bigg(\frac{w-a}{z-a}\bigg)^{r} \frac{1}{w-z}
;\label{C032} \ee $\mathcal{D}$ is a closed contour
containing the points $0$ and $a$, which is to the left
of the line $\Gamma_L:=L+i{\mathbb R}$ by picking $L$ large
enough. So, $\Re
(w-z)>0$.
 Consider now an arbitrary point $(y,t)$
 on the curve $\mathcal{C}$, parametrized by (\ref{tangency1})
  and the point $(\rho\sqrt{n/2},1)$,
 which is the point of intersection of the tangent to
 $\mathcal{C}$ at $(y,t)$ with the axis $(t=1)$; as pointed out before, it is convenient to parametrize $\rho$ by $\rho=e^{-\sigma}$ and thus
  \be
 (y,t)= \left(\frac{\rho \sqrt{2n}}{1+\rho^2},
 \frac{1}{ 1+\rho^2 }\right) = \left(\frac{\sqrt{2n}    }
   {2\cosh(\sigma)}  ,\quad  \frac{1}
 {1+e^{-2 \sigma }}\right)\in \mathcal{C}
.\label{tangency1}
 \ee
  Consider the Brownian motions with $r$ outliers forced to a point $\rho_0\sqrt{n/2}$ with $0<\rho_0\leq \rho$ at time $t=1$; here also parametrize $\rho_0$ by $\rho_0=e^{-\tau_0}$ and set $\alpha=\rho_0 / \rho=e^{\sigma-\tau_0}$.

The main issue is to compute the following limit, for $0<\rho_0\leq \rho$
\bean 
\lim_{n\rightarrow \iy}
\BP^{(0,\rho_0\sqrt{n/2})}_{Br}\left(\mbox{all}~
x_i\left(\frac{1}{1+e^{-2(\sigma+\frac{\tau}{n^{1/3}})}}\right) \in
  \frac{  \sqrt{2n}+\frac{E^c}{\sqrt{2}
  n^{1/6}} }{2\cosh(\sigma+ \frac{\tau}{n^{1/3}})}\right)
  ,\eean
 which dictates the space- and
  time-scale to be used in the kernel
 $H_n^{(r)}$ for large $n$. Let $t$ and $y$ be the time and space variables for the Brownian motion, which in terms of the new time and space scale $\mathcal{L}$ reads:
 \be \mathcal{L}:\quad
  t=\frac{1}{1+e^{-2(\sigma+\tau/n^{1/3})}},\quad
  x= \frac{
\sqrt{2n}+\frac{u}{\sqrt{2}
  n^{1/6}} }{2\cosh(\sigma+ \frac{\tau}{n^{1/3}})},\quad y= \frac{  \sqrt{2n}+\frac{v}{\sqrt{2}
  n^{1/6}}}{2\cosh(\sigma+ \frac{\tau}{n^{1/3}} )}
 , \label{substit}\ee
  with target point $a=\rho_0\sqrt{n/2}=e^{-\tau_0}\sqrt{n/2}$. Putting this rescaling in the integral (\ref{C032}) suggests changes of integration variables
 $$ z:=\tilde z\sqrt{\frac{n}{2}}e^{-\sigma-\tau/n^{1/3}}\mbox{    and    }
  w:=\tilde w\sqrt{\frac{n}{2}}e^{-\sigma-\tau/n^{1/3}}$$
 in the integral (\ref{C032}); the exponential will contain a function $ F(z)$, with Taylor series at $z=1$: \be F(z)=\frac{z^2}{2}-2z+\log z
  = F(1)+\frac{1}{3}(z-1)^3+ { O}(z-1)^4\label{F}
 . \ee   Also set $\zeta
 :=1+\frac{\gamma}{n^{1/3}}$ for some parameter $\gamma$ and
 \be
  Z_n=\left\{\begin{array}{lll}n^{-r/3}e^{nF(1)} &\mbox{ for }&\alpha=1\\
                                     (1-\alpha)^r e^{nF(1)}
                                     & \mbox{ for } &0<\alpha<1
                                     \end{array}\right.
     .\label{Zn}\ee
 Then, using,
 $$
\frac{1}{n^{1/3}(\tilde w-\tilde z)}=\int_0^{\iy}
 e^{-yn^{1/3}(\tilde w-\tilde z)}dy
 ,~~~~~~\mbox{for}~~\Re
(\tilde w-\tilde z)>0,$$
 one checks\footnote{ by elementary computation$$
 \frac{t}{1\!-\!t}=e^{2(\sigma+\frac{\tau}{n^{1/3}})}, \quad \frac{dy}{1-t} \frac{dz dw}{w-z}=
  \frac{n^{\frac{1}{3}}d\tilde z d\tilde w dv}{\tilde w-\tilde z}
 ,\quad  \frac{2xz}{1\!-\!t}=(2n+un^{\frac{1}{3}})\tilde z,\quad \frac{tz^2}{1\!-\!t} =\frac{n}{2}\tilde z^2.
   $$}, using the rescaling (\ref{substit}):
   \bean
 \lefteqn
{  \left.
 H_n^{(r)}\left(
x,y\right)dy
  \right|_{\mathcal{L}}
 }
   \\
  &=&
 -\left(\frac{n^{1/3}}{2\pi}\right)^2
 dv\int_{\mathcal{D}} d\tilde z\int_{\Gamma_L}
 d\tilde w
 e^{(-\frac{n}{2}\tilde z^2+\tilde z(2n+n^{1/3}u))}
 e^{-(-\frac{n}{2}\tilde w^2+\tilde w(2n+n^{1/3}v))}
 \\
 &&\left(\frac{\tilde w}{\tilde z}\right)^n
 \left(  \frac{\tilde w-\alpha e^{\tau/n^{1/3}}}{\tilde
 w}\right)^r
 \left(\frac{\tilde z-\alpha e^{\tau/n^{1/3}}}{\tilde z} \right)^{-r}
 \frac{1}{n^{1/3}(\tilde w-\tilde z)}
\\
%
&=&- e^{\zeta n^{1/3}(u-v)}dv\int_0^{\iy}dy \frac{Z_n
n^{1/3}}{2\pi}\int_\mathcal{D} d\tilde z~e^{-nF(\tilde z)~}
 \left(\frac{\tilde z}{\tilde
z-\alpha e^{\tau/n^{1/3}}}\right)^r
 e^{n^{1/3}(u+y)(\tilde z-  \zeta)}
\\
&&
\frac{ n^{1/3}}{2\pi
Z_n}\int_{\Gamma_L} d\tilde w~e^{ nF(\tilde w)~}
 \left(\frac{\tilde w}{\tilde
w-\alpha e^{\tau/n^{1/3}}}\right)^{-r}
 e^{-n^{1/3}(v+y)(\tilde w-  \zeta)}
.\eean
Conjugating the kernel, which leaves invariant the
Fredholm determinant, one finds
$$
 e^{  n^{1/3}(v-u)} \left.
H_n^{(r)}\left(x,y\right)\right|_{\LR}dy =-dv~e^{\gamma(u-v)}\int_0^{\iy}dy~\mathcal{I}_{\tau}^{(n)}(u+y)\mathcal{J}_{\tau}^{(n)}(v+y)
$$
with
\bea
 \label{Integrals}
 \mathcal{I}_{\tau}^{(n)}(x)&=&\frac{n^{1/3}Z_n}{2\pi}\int_\mathcal{D} d\tilde z~e^{-nF(\bar z)}\left( \frac{\tilde
z}{\tilde z-\alpha e^{\tau/n^{1/3}}}\right)^re^{
n^{1/3}x(\tilde z-\zeta)} \\
\no\\
\mathcal{J}_{\tau}^{(n)}(y)&=&\frac{n^{1/3}}{2\pi~Z_n}\int_{\Gamma_L}d\tilde
w~e^{nF(\tilde w)}\left( \frac{\tilde w}{\tilde
w-\alpha e^{\tau/n^{1/3}}}\right)^{-r}e^{ -n^{1/3}y(\tilde
w-\zeta)} .\no\eea
Using the Taylor series (\ref{F}) for $F(z)$ and the value (\ref{Zn}) for $Z_n$, 
%
%
 one is led naturally to pick a new variable $u$ such that
$ \tilde z-1=\frac{u}{n^{1/3}}. $ Then
\bean \lefteqn
{n^{1/3}Z_nd\tilde z~e^{-nF(\tilde z)}\left(\frac{\tilde z}{\tilde z-\alpha e^{\tau/n^{1/3}}}\right)^re^{n^{1/3}
x(\tilde z-\zeta)}}
\\
&=&n^{1/3}Z_n
\frac{du}{n^{1/3}}\left(
e^{-nF(1)-\frac{u^3}{3}}+ {
O}\left(\frac{1}{n^{1/3}}\right)\!\!\right)
\\
&&
\hspace*{4cm} \left(\frac{u+n^{1/3}}{u+n^{1/3}\!-\!\alpha \left(
1\!+\!\frac{\tau}{n^{1/3}}\!+\!\ldots\right)n^{1/3}}
\right)^re^{x(u-\gamma)}\\
\\
&=&\left\{\begin{array} {lll}e^{-x\gamma}du~e^{-\frac{u^3}{3}+xu}\left(
\frac{1}{u-\tau}\right)^r+\mbox{lower order terms}&\mbox{ for } &
 \alpha=1\\
  e^{-x\gamma}du~e^{-\frac{u^3}{3}+xu}+\mbox{lower order terms}&\mbox{ for } &
0< \alpha<1.
\end{array}\right.
\eean
%
%
Upon remembering the definition of the functions $A^\pm _r(x,\tau)$ and the Airy function $A(x)$ and also the fact that $\alpha=\rho_0/\rho$, the rigorous saddle point argument, given in
\cite{BBP,Peche}, yields
\bea \lefteqn{ \lim_{n\rg\iy}\mathcal{I}_{\tau}^{(n)}(x)=}\no\\
&&
\left\{\begin{array}{l} \frac{1}{2\pi}
 \int_{\mathcal{D}'}\frac{1}{(z-\tau)^r}e^{-\frac{z^3}{3}}e^{
 x(z-\gamma)}dz=
\mathcal{I}_{\tau}(x):=e^{-x\gamma}A^-_r(x,\tau)\\
\qquad\qquad\qquad\qquad\qquad\qquad \mbox{ for }\quad
\rho_0=\rho
  \\
   \frac{1}{2\pi}
 \int_{\mathcal{D}'} e^{-\frac{z^3}{3}}e^{
 x(z-\gamma)}dz= \mathcal{I}_{ }(x):=e^{-x\gamma}A(x)\\
 \qquad\qquad\qquad\qquad\qquad\qquad\mbox{ for }\quad 0<\rho_0<\rho
   \end{array} \right.\label{I}\\ \no\\
\lefteqn{\lim_{n\rg\iy} \mathcal{J}_{\tau}^{(n)}(y)=}\no
\\& &
\left\{ \begin{array}{l}\frac{1}{2\pi}
\int_{\mathcal{D}''}
(z-\tau)^r
e^{\frac{z^3}{3}}e^{-y(z-\gamma)}dz=\mathcal{J}_{\tau}(y):=-e^{y\gamma}A^+_r(y,\tau),\\
\qquad\qquad\qquad\qquad\qquad\qquad\mbox{for }\quad
\rho_0=\rho
 \\
 \frac{1}{2\pi}
\int_{\mathcal{D}''}
e^{\frac{z^3}{3}}e^{-y(z-\gamma)}dz=\mathcal{J}(y):=-e^{y\gamma}A(y),\\
\qquad\qquad\qquad\qquad\qquad\qquad\mbox{for }\quad
0<\rho_0<\rho
\end{array}\right.
\label{J}
 \eea
  where
 $\mathcal{D}'$ is a contour running from $\iy
 e^{4i\pi/3}$ to $\iy
 e^{2i\pi/3}$, with an indentation to the right of $\tau$
 , such that $\tau$ is to the left of the contour,
and where
  $\mathcal{D}''$ is a contour running from $\iy
 e^{-i\pi/3}$ to $\iy
 e^{i\pi/3}$, with an indentation to the right of $\tau$,
 such that $\tau$ lies also to the left of the contour.
 Upon rotating the two contours and deforming $\mathcal{D}''$
 slightly, since the integrand is pole-free, one
 gets the final identities in the equations above.

 Therefore
 $ \BP(\sup \mathcal {A}^{(r)}(\tau)\cap
 E=\emptyset)=\det(I-K_\tau^{(r)})_E ,$
with
\bean
  \lefteqn{
 {K_\tau^{(r)}(u,v)dv}
=\lim_{n\rg\iy}
 e^{n^{1/3}(v-u)}
H_n^{(r)}\left(x,y\right)dy\Bigr|_{\LR}
 }\\
    &=&
    \left\{\begin{array}{l}
   dv
 \int^{\iy}_0 dw~A^-_r(u+w,\tau)A^+_r(v+w,\tau)
 \qquad\mbox{  for  } \rho_0=\rho \\ \\
   dv
 \int^{\iy}_0 dw~A (u+w,\tau)A (v+w,\tau)
 \qquad\mbox{  for  } 0<\rho_0<\rho.
\end{array}\right.\eean
 Baik \cite{Baik} has shown that the Fredholm determinant of
the $r$-Airy kernel is a probability distribution, i.e.,
$$\lim_{x\rg \pm \iy}
\det\left(I-K_\tau^{(r)}\right)_{(x,\iy)}=\left\{{1
\atop 0 .}\right.$$ This establishes Theorem
\ref{Theo1}.\end{proof}



\begin{remark} In this section one lets $\tau\rg -\iy$, which
implies that $-i\tau$ remains above the contour $C$ and
is thus compatible with the contour mentioned above.
Letting $\tau\rg +\iy$ would require a drastic change of
the functions $A^{\pm}_r$.
\end{remark}

The next statement concerns the asymptotic
behavior of the $r$-Airy kernel for $\tau\rg -\iy$,
   \be
 K^{(r)}_\tau(u,v)=\int_0^{\iy}
 dw A_r^-(u+w; \tau)A_r^+(v+w; \tau)
 ,\label{r-kernel2}\ee
 where (remember)
 \be
A_r^\pm(u; \tau) =
 \int_C ~e^{\frac{1}{3} ia^3+iau }
 \left( {\mp ia-\tau}\right)^{\pm r}
  \frac{da}{2\pi}
 ,\label{k-Airy}\ee
 where $C$ is a contour running from $\iy
 e^{5i\pi/6}$ to $\iy
 e^{i\pi/6}$, such that $-i\tau$ lies above the contour.
%
This limit is compatible with the contour $C$
 appearing in the definition of the functions
 $A_r^{\pm}$, since then $-i\tau$ remains
 above the contour $C$, as required.

 \begin{lemma}\label{lemma:kernels} Given the
 ``initial condition" $$\lim_{\tau\rg
 -\iy}K^{(r)}_\tau(u,v)=K^{(0)}(u,v):=
 \frac{A(u)A'(v)-A'(u)A(v)}{u-v}=
  \mbox{``Airy kernel"}
 ,$$
   the kernel
 $K^{(r)}(u,v)$ behaves asymptotically for
 $\tau\rg -\iy$, as
$$ K^{(r)}(u,v)
=
 K_{0}+\frac{K_1^{(r)}}{\tau}+\frac{K_2^{(r)}}{\tau^2}+
 \frac{K_3^{(r)}}{\tau^3}+\ldots
,$$
 where\footnote{Whenever $\pl /
 \pl w$ appears with a negative exponent in the formula below, it is set $=0$.} for $n\geq 1$,
 {
 \bean
{K_n^{(r)}(u,v)}
&=&-\frac{r^n}{n!}\left(\frac{\pl}{\pl w}\right)^{n-1}A(u+w)A(v+w)\Big|_{w=0}\\
\\
& &-\frac{r^{n-1}}{2(n-2)!}\left(\frac{\pl}{\pl
w}\right)^{n-2}\!\!\Big(A'(u+w)A(v+w)\!-\!A
(u+w)A'(v+w)\Big)\Big|_{w=0}\\
\\
& &-\frac{r^{n-2}}{(n-3)!}\left(
\begin{array}{l}
\frac{3n-1}{24}\left(\frac{\pl}{\pl w}\right)^{n-1}A(u+w)A(v+w)\\
-\frac{n-1}{2}\left(\frac{\pl}{\pl
w}\right)^{n-3}A'(u+w)A'(v+w)
\end{array}\right)\Bigg|_{w=0}
 \\
 & &
+~(\mbox{polynomial of degree $n-3$ in $r$}).
 \eean}
Although the kernel $K_\tau^{(r)}(u,v)$ involves
integration, the terms $K_i^{(r)}(u,v)$ in the expansion
never involve integration, they are quadratic in the
Airy function and its derivatives; also the
$K_i^{(r)}(u,v)$ are polynomials in $r$ of degree $i$,
divisible by $r$, with alternately symmetric and
skew-symmetric coefficients in $u$ and $v$, the top
coefficient being symmetric.

\end{lemma}

\medbreak

\begin{proof} In order to expand the kernel (\ref{r-kernel})
with regard to $\tau$ for $\tau\rightarrow -\iy$, set
the expressions (\ref{k-Airy}) into the kernel
(\ref{r-kernel}), which then becomes a triple integral.
Set $\al=ia$ and $\beta=ib$ and consider the following
Taylor expansions about $\tau= -\iy$, {
 \bea
{\frac{1}{(\al-\tau)^r(-\beta-\tau)^{-r}}}
 &=&\left(1+\frac{\frac{\al+\beta}{\tau} }{ 1-\frac{\al}{\tau}}\right)^r \no\\
\no\\
 &=&
  1+\sum_{n=1}^\iy \frac{1}{\tau^n}
     \sum_{j=0}^{\left[\frac{n-1}{2}\right]}
   \left(\begin{array}{l}
   (r(\al+\beta))^{n-2j}Q_{2j}(\al,\beta)\\
   +(r (\al+\beta))^{n-2j-1} (\al-\beta)\tilde
   Q_{2j}(\al,\beta)\end{array}
   \!\!\!\!\!\right)
%
 \no\\
 &=&1\!+\!\sum_{n=1}^\iy \frac{1}{\tau^n}
  \left(\!\!
  \begin{array}{l}
  \frac {(r(\al+\beta))^n}{n!}
  +
    \frac {(r (\al+\beta))^{n-1}(\al-\beta)}{2(n-2)!}
   \\
   +
    \frac {(r(\al+\beta))^{n-2}}{(n-3)!}
    \left(\frac{(3n-1)(\al+\beta)^2}{24}
     -\frac{(n-1)\al\beta}{2}\right)+\ldots
  \end{array}\!\!\!\!\!\right)
%
\label{powerseries},\eea}
where $Q_{2j}$ and $\tilde Q_{2j}$ are symmetric
homogeneous polynomials of degree $2j$ in the arguments,
since the first expression is invariant under the
involution $r\mapsto -r$ and $\al\mapsto -\beta$. The
coefficients of $1/\tau^n$ are divisible by $r$, for the
simple reason that for $r=0$, the expression above
equals $1$.
%
%

   Also notice
multiplication by $ia$ of the integrand in the kernel
(\ref{r-kernel}),
 \bean K^{(r)}_\tau(u,v)\!\!\!&=&\!\!\!\int_0^{\iy}\!\!
 dw \!\!\int_C   e^{\frac{1}{3} ia^3+ia(w+u) }
 \left(\frac{1}{ia-\tau}\right)^r
 \frac{da}{2\pi}
  \int_C ~e^{\frac{1}{3} ib^3+ib(w+v) }
 \left( {-ib-\tau}\right)^r
  \frac{db}{2\pi}\\
  \!\!\!&=&\!\!\!
  \int_0^{\iy}
 \!\!
 dw \!\!\int_C  \int_C 
  \frac{dadb}{4\pi^2}
e^{\frac{1}{3} ib^3+ib(w+v)}  e^{\frac{1}{3}
ia^3+ia(w+u) }(1+\frac{r}{\tau}(ia+ib)+\ldots)
  \eean
 can be realized by taking $\frac{\pl K^{(0)}
}{\pl u}$ and similarly multiplication of the integrand by $ib$ is realized by taking
$\frac{\pl K^{(0)} }{\pl v}$; thus we have the following
recipe
\bean 
ia &\leftrightarrow& \frac{\pl K^{(0)}
}{\pl
u}=\int_0^{\iy}dwA'(u\!+\!w)A(v\!+\!w)   \\
  ib &\leftrightarrow& \frac{\pl K^{(0)} }{\pl
v}=\int_0^{\iy}dwA(u\!+\!w)A'(v\!+\!w), \eean and so in
particular, \bean
{(ia)^{k_1}(ib)^{k_2}(ia\!+\!ib)^n}
 &\Leftrightarrow&
  \int_0^{\iy}\left(\frac{\pl}{\pl u}\!+\!\frac{\pl}{\pl v}\right)^n
 \left(\frac{\pl }{\pl u }\right)^{k_1}
 \left(\frac{\pl}{\pl v}\right)^{k_2}
  A(u\!+\!w)A(v\!+\!w)dw\\
   &&=
   -\left(\frac{\pl}{\pl
   w}\right)^{n-1}A^{(k_1)}(u+w)A^{(k_2)}(v+w)\Bigr|_{w=0}
  . \eean
Notice that, since $ia+ib$ factors out of every term in
the expansion (\ref{powerseries}), the kernels obtained
never contain integration. In addition, since $ia-ib$
factors out of every other term, every other term in
$K^{(r)}_i(u,v)$ must be skew; in particular it vanishes
for $u=v$.
 One then reads off the $K^{(r)}_i(u,v)$'s from the expansion
(\ref{powerseries}) and the recipe above, upon using
occasionally the differential equation $A''(x)=xA(x)$
for the Airy function, thus
 ending the proof of Lemma \ref{lemma:kernels}.
\end{proof}

\begin{remark} As an example, we give explicit expressions for
the first few $K_i^{(r)}(u,v)$'s:
{
\bea
&& \\
\label{Ki's} K_0^{(r)}(u,v)&=&K^{(0)}(u,v)=\frac{A(u)A'(v)-
A'(u)A(v)}{u-v}
  \no\\ 
K_1^{(r)}(u,v)&=&-r~A(u)A(v)
\no\\
K_2^{(r)}(u,v)&=&-\frac{r^2}{2}(A'(u)A(v)+A(u)A'(v))+\frac{r}{2}(u-v)K^{(0)}(u,v)
\no\\ 
K_3^{(r)}(u,v)&=&-\frac{r^3}{6}(A''(u)A(v)+2 A
'(u)A'(v)+A(u)A''(v)) +\frac{r^2}{2}(v-u)A(u)A(v)
\no\\
& &-\frac{r}{3}\left( A''(u)
A(v)+A(u)A''(v)-A'(u)A'(v)\right).
\no   
  %
\eea}
\end{remark}

In order to find the PDE for the transition probability, one will need an estimate on how the actual transition probability for the finite problem converges for $n\rg \iy$. This will be used in (\ref{6.3}).

\begin{corollary}\label{cor:estimate2}
 For $x\in {\mathbb R} $ sufficiently large, one has for some constant $C>0$,
\bea
\lefteqn{\hspace*{-2cm} \left|\log \BP^{(0,\rho_0\sqrt{\frac{n}{2}})}_{Br}\left(\mbox{all}~
x_i\left(\! \frac{1}{1\!+\!e^{-2(\tau_0\!+\!\frac{\tau}{n^{1/3}})}}\!\right)
\! \leq \!
  \frac{  \sqrt{2n}+\frac{x}{\sqrt{2}
  n^{1/6}}}{2\cosh(\tau_0\!+\! \frac{\tau}{n^{1/3}})}\right)
\right. }     \label{estimate1} \\
 &&\no
 ~~~\left.-\log \BP(\sup \mathcal {A}^{(r)}(\tau)\leq x)\right|  \leq Cn^{-1/3}.
  \eea
\end{corollary}

\begin{proof}
For any trace class operators $K_n$ and $K_\iy$, set
\footnote{Define the three norms on a Hilbert space: the sup, the trace and the Hilbert-Schmidt norms, with $ \vv T \vv \leq  \vv T \vv_{_{HS}} \leq  \vv T \vv_{_{\tr}}$ :
$$\vv T \vv=\sup_{v \in \mathcal {H}}\frac{|Tv|}{|v|}
=\sup |\lb_i|,~
 \vv T \vv_{_{\tr}}=\Tr (T^*T)^{1/2}=
  \sum |\lb_i|,~ \vv T \vv_{_{HS}}=
  (\Tr T^*T)^{1/2}=(\sum |\lb_i|^2)^{1/2}.$$}
  $$
 M:= \max ( \vv K_n  \vv, \vv K_\iy  \vv).
  $$
  Then assuming $M<1$,
  one checks, using $K_n^i-K_\iy^i=\sum_{\ell=0}^{i-1}
   K_\iy^\ell (K_n-K_\iy)K_n^{i-1-\ell}$ and $\vv AB\vv_{_{\tr}}\leq \vv A\vv~ \vv B\vv_{_{\tr}}$ and $\vv AB\vv_{_{\tr}}\leq\vv A\vv_{_{\tr}} \vv B\vv
   $,
\bea
  \label{estimate} \hspace*{.4cm}\left| \log \det (I-K_n)-\log \det (I-K_\iy)\right|
  &=&
    \left| \tr(\log (I-K_n)-\log  (I-K_\iy))\right|
   \\
     &\leq&
   \sum_1^{\iy}
       \left|
        \tr(\frac{K_n^i-K_\iy^i}{i})
        \right|
        \no\\
        &\leq&
        \sum_1^{\iy} \frac{1}{i}\vv
      K_n^i-K_\iy^i    \vv_{_{\tr}}
      \no\\
       &\leq&
      \vv
      K_n-K_\iy   \vv_{_{\tr}}
      \sum_0^{\iy}
       M^i
       \no\\
        &\leq&
         \frac{\vv
     K_n-K_\iy   \vv_{_{\tr}}}{1-M}.
    \no
\eea
Setting (see notation (\ref{Integrals}))
\bean
 K_n(u,v)&:=&  \int_0^\iy dy \mathcal{I}_\tau^{(n)} (u+y)
                      \mathcal{J}_\tau^{(n)} (v+y)\\
  K_\iy(u,v)&:=&
                       e^{(v-u)\gamma} \int_0^\iy dy A_r^-(u+y,\tau)  A_r^+(v+y,\tau)
            ,  \eean
one checks
\bean
\lefteqn{\left |
K_n(u,v)-K_\iy(u,v) \right |
}\\
&\leq &
 \left |
 \int_0^\iy dy \Bigl(\mathcal{I}_\tau^{(n)} (u+y)-e^{-(u+y)\gamma} A_r^-(u+y,\tau)\Bigr)
   \mathcal{J}_\tau^{(n)} (v+y)\right |
 \\
&&+\left|\int_0^\iy dy e^{-(u+y)\gamma}A_r^-(u+y,\tau)
 \Bigl(\mathcal{J}_\tau^{(n)} (v+y)-e^{(v+y)\gamma} A_r^+(v+y,\tau)\Bigr)\right|
.\eean
  An argument similar to the one of Baik-BenArous-P\'ech\'e \cite{BBP,Peche} shows that
for given $x_0\in {\mathbb R}$, there are constants $C>0,~N>0$,
\bean
 \left | \mathcal{I}_\tau^{(n)} (x)-e^{-x\gamma} A_r^-(x,\tau) \right |
  &\leq& \frac{Ce^{-cx}}{n^{1/3}} \\
  \left | \mathcal{J}_\tau^{(n)} (x)-e^{x\gamma} A_r^+(x,\tau) \right |
  &\leq& \frac{Ce^{-cx}}{n^{1/3}} ~,~~\mbox{for $n\geq N$ and $x\geq x_0$}.
     \eean
%
 Viewing the functions in the integrals on the right hand side
 as kernels representing H\"ankel-like
integral operators on $(x_0,\iy)$, one has, using the
inequality $\vv AB  \vv_{_{\tr}}\leq \vv A
\vv_{_{HS}} \vv B \vv_{_{HS}}$ for
Hilbert-Schmidt operators,
  \bean
   \vv K_n-K_\iy  \vv_{_{\tr}}
    &\leq &
    \vv   \mathcal{I}_\tau^{(n)} (u+y)-e^{-(u+y)\gamma} A_r^-(u+y,\tau)    \vv_{_{HS}}   \vv     \mathcal{J}_\tau^{(n)} (v+y)       \vv_{_{HS}}
  \\  \\ && +
     \vv   e^{-(u+y)\gamma}A_r^-(u+y,\tau)     \vv_{_{HS}}
      \vv    \mathcal{J}_\tau^{(n)} (v+y)-e^{(v+y)\gamma} A_r^+(v+y,\tau)       \vv_{_{HS}}
  \\
  &\leq & 2C'n^{-1/3},
      \eean
     where $$C'= \max (\vv  e^{-(u+y)\gamma}A_r^-(u+y,\tau)   \vv_{_{HS}}, \sup_{n\geq N}\vv    \mathcal{J}_\tau^{(n)} (v+y)    \vv_{_{HS}})Ce^{-cx_0}.$$

     Then, from (\ref{estimate}), one has that
      $$\left| \log \det (I-K_n)-\log \det (I-K_\iy)\right|\leq \frac{2C'}{1-M} n^{-1/3}
      .$$
   Remembering the representation of the probabilities in the statement (\ref{estimate1}) in terms of Fredholm determinants
   establishes Corollary \ref{cor:estimate2}.\end{proof}






\section{An integrable deformation of Gaussian random
ensemble with external source and 3-component
KP}\label{sect:KP}

The connection between the Gaussian random ensemble with
external source and the multi-component KP hierarchy is
explained in \cite{AvM-Pearcey} and \cite{AvM-Mops}. The
main ideas are sketched in this section. For the multicomponent KP hierarchy, see [30].
%
%
\subsection{Two sets of weights and the $p+q$-KP
hierarchy} Define two sets of weights
$$
 \psi_1(x),\ldots,\psi_q(x)~~~\mbox{and}~~~
 \varphi_1(y),\ldots,\varphi_p(y),~~\mbox{with}~x,y\in
 {\mathbb R},
$$
and deformed weights depending on {\it time} parameters
$s_\al=(s_{\al 1},s_{\al 2},\ldots)$ ($1\leq \al\leq q$)
and $t_\beta=(t_{\beta 1},t_{\beta 2},\ldots)$ ($1\leq
\beta \leq p$), denoted by
\be
  \psi_\al^{-s}(x):=\psi_\al(x)e^{-\sum_{k=1}^\iy s_{\al k}x^k}
   ~~~~~\mbox{and}~~~~~
  \varphi_\beta^{t}(y):=\varphi_\beta (y)e^{\sum_{k=1}^\iy t_{\beta
   k}y^k}.
\label{weights}\ee
That is, each weight goes with its own set of times. For
each set of positive integers\footnote{$\vert m
\vert=\sum_{\al=1}^{q}m_\al$ and $\vert
n\vert=\sum_{\beta=1}^{p}n_\beta$.}
$
  m=(m_1,\ldots,m_q),~n=(n_1,\ldots,n_p)\mbox{~ with~}\vert m\vert=\vert n\vert,
$
consider the determinant of a moment matrix $T_{mn}$ of
size $\vert m\vert=\vert n\vert$, composed of $pq$
blocks of sizes $m_in_j$; the moments are taken with
regard to a (not necessarily symmetric) inner product
$\INN$
\bea
  \label{Tmn_def}  \lefteqn{\hspace*{.5cm}\tau_{mn}(s_1,\ldots,s_q;t_1,\ldots,t_p)}
  \hskip1cm \\ 
   &&  \hspace*{-1.5cm}
  := \det\left(\!\!\!\!\begin{array}{ccc}
  \Bigl(\!
  \inn{x^i\psi^{-s}_1(x)} {y^j\varphi^{t}_1(y) }
  \Bigr)_{0\leq i<m_1\atop{0\leq j<n_1}}&\ldots&\left(
  \inn{x^i\psi^{-s}_1(x)} {y^j\varphi^{t}_p(y) }
  \!\right)_{0\leq i<m_1\atop{0\leq j<n_p}}\\
  \vdots& &\vdots\\
  & &\\
  \left(\!\inn{x^i\psi^{-s}_q(x)}
  {y^j\varphi^{t}_1(y) } \right)_{0\leq i<m_q\atop{0\leq j<n_1}}&\ldots&\left(\inn{x^i\psi^{-s}_q(x)}
  {y^j\varphi^{t}_p(y)}\! \right)_
  {0\leq i<m_q\atop{0\leq j<n_p}}
  \end{array}\!\!\!\!\right).
 \no
\eea


We now state a non-trivial Theorem involving a relationship between the determinants of the block moment matrices above, by increasing or decreasing the sizes of the blocks by one. Modifying the size $n_\beta$ in $n=(n_1,\ldots, n_p)$ by $1$ is indicated by $n\mapsto n\pm e_\beta$, where $e_\beta=(0,\ldots,0, 1,0,\ldots,0)$, with $1$ at place $\beta$. The proof and many simple examples can be found in \cite{AvM-Mops}:

\begin{theorem}\label{Theo:p+q-KP}{\em (Adler, van Moerbeke and Vanhaecke \cite{AvM-Mops})}
Then the block matrices $\tau_{mn}$ satisfy the $(p+q)$-KP hierarchy; to be precise, the functions $\tau_{mn}$ satisfy the bilinear relations\footnote{The integrals are contour integrals along a small circle about $\iy$, with formal Laurent series as the integrand. Also, for $z \in \BC$, we define $[z^{-1}]:= (\frac{z^{-1}}{1},\frac{z^{-2}}{2},\frac{z^{-3}}{3},\ldots)$. For a given polynomial $p(t_1,t_2,\dots)$, the Hirota symbol between functions $f=f(t_1,t_2,\ldots)$ and
$g=g(t_1,t_2,\ldots)$ is defined by
$
  p(\frac{\pl}{\pl t_1},\frac{\pl}{\pl t_2},\dots)f\circ g:=
  p(\frac{\pl}{\pl y_1},\frac{\pl}{\pl y_2},\dots)
  f(t+y)g(t-y)\Bigl|_{y=0}.
$
We also need the elementary Schur polynomials
${\gs}_{\ell}$, defined by
$e^{\sum^{\iy}_{1}t_kz^k}:=\sum_{k\geq 0} {
\gs}_k(t)z^k$ for $\ell\geq 0$ and ${\gs}_{\ell}(t)=0$
for $\ell<0$; moreover, set
$
  \gs_{\ell}(\tilde \pl_t):={\gs}_{\ell}(\frac{\pl}{\pl t_1},
  \frac{1}{2}\frac{\pl}{\pl t_2},\frac{1}{3}\frac{\pl}{\pl t_3},
  \ldots).
$
 }
    \be \label{tau-bilinear identity}\ee
   \vspace*{-.7cm}
      \bean
  \sum_{\beta=1}^p\oint_\infty
  (-1)^{\sigma_\beta(n)} \tau_{m,n-e_{\beta}}
    (t_{\beta}-[z^{-1}])\tau_{m^*,n^*+e_{\beta}}(t^*_{\beta}+[z^{-1}])
          e^{\scriptscriptstyle\sum_1^{\iy} (t_{\beta k}-t_{\beta k}^*)z^k}z^{n_\beta-n_\beta^*-2}\,dz=\nonumber\\
  \sum_{\al=1}^q\oint_\infty
 (-1)^{\sigma_\al(m)}\,\tau_{m+e_\al,n}(s_{\al}-
  [z^{-1}])\tau_{m^*-e_\al
    ,n^*}(s^*_{\al}+[z^{-1}])
    e^{\sum_1^{\iy} (s_{\al k}-s_{\al k}^*)z^k}\,
    z^{m_{\al}^*-m_\al-2}\,dz,
  \eean
  for all $m,n,m^*,n^*$ such that
  $\vert m^*\vert=\vert n^*\vert+1$
and $\vert m\vert=\vert
  n\vert-1$ and all $s,t,s^*,t^*
\in \BC^{\iy}$ and where
   $
     \sigma_\al(m)={\sum_{\al'=1}^\al (m_{\al'}
-m_{\al'}^*)}\quad\hbox{and}\quad
     \sigma_\beta(n)={\sum_{\beta'=1}^\beta(n_{\beta'}-n_{\beta'}^*)}.
   $%
   \end{theorem}

Computing the residues in the contour integrals above,
the functions $\tau_{mn}$, with $\vert m\vert=\vert
n\vert$, satisfy the following PDE's 
in terms of the Hirota symbol, defined in the footnote 11:
%
{
 \bea
  \tau_{mn}^2\frac{\pl^2}{\pl t_{\beta,\ell+1}\pl t_{\beta',1}}
  \log \tau_{mn}
  &=&
  \gs_{\ell+2\delta_{\beta\beta'}}\bigl(\tilde\pl_{t_{\beta}}
  \bigr) \tau_{m,n+e_{\beta}-e_{\beta'}}
  \circ\tau_{m,n+e_{\beta'}-e_{\beta}}  \nonumber\\
  \tau_{mn}^2 \frac{\pl^2}{\pl s_{\al,\ell+1}\pl
  s_{\al',1}}\log
  \tau_{mn}
  &=&\gs_{\ell+2\delta_{\al\al'}}(\tilde\partial_{s_{\al}})
     \tau_{m+e_{\al'}-e_{\al},n}
     \circ\tau_{m+e_{\al}-e_{\al'},n} 
\no\\ \label{introa}%
  \tau_{mn}^2 \frac{\pl^2}{\pl s_{\al,1}\pl t_{\beta,\ell+1}}
  \log \tau_{mn}
  &=&-\gs_{\ell}(\tilde\partial_{t_{\beta}})\tau_{m+e_{\al},
  n+e_\beta}\circ\tau_{m-e_{\al},n-e_{\beta}}\nonumber\\
  \tau_{mn}^2\frac{\pl^2}{\pl t_{\beta,1}\pl s_{\al,\ell+1}}
  \log \tau_{mn}
  &=&-\gs_{\ell}(\tilde\partial_{s_{\al}})
  \tau_{m-e_{\al},n-e_\beta}\circ\tau_{m+e_{\al},n+e_{\beta}}.
\eea}
%
%
\subsection{Gaussian ensemble with external source}
Consider an ensemble of $n\times n$ Hermitian matrices
with an external source, given by a diagonal matrix
 $
 A=\diag(a_1,\ldots,a_n)
  $
 and a general potential $V(z)$, with density
  $$
  \BP_n(M\in [M,M+dM])=
  \frac{1}{Z_{n}}
  e^{-\Tr(V(M)-AM) }dM.
  $$
For a subset $E\subset {\mathbb R}$, 
%
the following probability can be transformed by the
Harish-Chandra-Itzykson-Zuber formula, with
$D:=\diag(z_1,\ldots,z_n$), $$\Dt_n(z):=\prod_{1\leq
i<j\leq n}(z_i-z_j),$$ and all distinct $a_i$,
 \bea
 \label{1.1}  \BP_n(\mbox{spectrum} ~M\subset E)
 &=&\frac{1}{Z_n} \int_{\mathcal{H}_{n}(E)} e^{-\Tr(V(M)-AM) }dM
  \\
   && \no\\
  &=&\frac{1}{Z_n} \int_{E^n} \Dt^2_n(z)\prod_1^n
e^{-V(z_i)}dz_i\int_{U(n) }e^{\Tr AUD
U^{-1}}dU    \no\\
&=&\frac{1}{Z'_n} \int_{E^n} \Dt^2_n(z)\prod_1^n
e^{-V(z_i)}dz_i\frac{\det[e^{a_iz_j}]_{1\leq i,j\leq
n}}{\Dt_n(z)\Dt_n(a)}  \no\\
  &=&\frac{1}{Z''_n}
\int_{E^n}\Dt_n(z)\det[e^{-V(z_j)+a_iz_j}]_{1\leq
i,j\leq n}\prod_1^n dz_i,
\no\eea
%
with $a_i\neq a_j$ and the Vandermonde
$\Dt_n(z)=\prod_{1\leq i<j\leq n}(z_i-z_j)$. The formula
remains valid in the limit, when some $a_i$'s coincide,
upon making differences of rows and dividing by the
appropriate $(a_i-a_j)$'s. In the following Proposition,
we consider a general situation, of which (\ref{1.1})
with $A=\diag (a,\ldots,a,0,\ldots,0)$ is a special
case, by setting $\varphi^+=e^{az}$ and $\varphi^-=1$.
 Consider the Vandermonde determinant 
 $\Dt_n(x,y)
 :=\Dt_n(x_1,\ldots,x_{k_1}
 ,y_1,\ldots, y_{k_2}). $
Then we have the following (see \cite{AvM-Pearcey}):

\begin{proposition}
Given an arbitrary potential $V(z)$ and arbitrary
functions
 $\varphi^+(z)$ and $\varphi^-(z)$, define
 ($n=k_1+k_2$)
\bean
(\rho_1,\ldots,\rho_n)&:=&e^{-V(z)}\left(\varphi^+(z),z\varphi^+(z)
,\ldots,   z^{k_1-1}\varphi^+(z),  \right.
 \\
 &&\hspace*{2cm} \left.\right.
\left.~\varphi^-(z),
~z\varphi^-(z),\ldots,z^{k_2-1}\varphi^-(z)\right)
 .
 \eean
 We have
 \bea
 \label{1.3}\lefteqn{\hspace*{.6cm}\frac{1}{n!}\int_{E^n}\Dt_n(z) \det(\rho_i(z_j))_{1\leq i,j\leq
n}
   \prod_1^n
    dz_i}
     \\
&=&\frac{ 1}{k_1!k_2!}\int_{E^n}\Dt_n(x,y)
\Dt_{k_1}(x)\Dt_{k_2}(y)\prod_1^{k_1}\varphi^+(x_i)
e^{-V(x_i)}dx_i
\prod_1^{k_2}\varphi^-(y_i)e^{-V(y_i)}dy_i
\no\\
&=& \det\left(\begin{array}{c}
\left(\displaystyle{\int_{E}}z^{i+j}\vp^+(z)e^{-V(z)}\right)_{{\begin{array}{l}
     0\leq i\leq k_1-1\\
     0\leq j\leq k_1+k_2-1
     \end{array}}} \\
  \\
\left(\displaystyle{\int_{E}}z^{i+j}\vp^-(z)e^{-V(z)}\right)_{{\begin{array}{l}
     0\leq i\leq k_2-1\\
     0\leq j\leq k_1+k_2-1
     \end{array}}}
\end{array}\right)
 .\no
 \eea

\end{proposition}


\subsection {Adding extra-variables $t,~s~u,~ \mbox{and}~ \beta$}

We add the extra-variables $t=(t_1,t_2,\ldots),~~
                       s=(s_1,s_2,\ldots),~~
                       u=(u_1,u_2,\ldots)~~\mbox{and}~~
                       \beta$ in the
exponentials, as follows (
$n=k_1+k_2$),
\bea
   V(z):=\frac{z^2}{2}-\sum_1^{\iy}t_i z^i  \hspace*{2cm}
  \no\\
   \varphi^+(z)= e^{ az+\beta z^2
  -\sum_1^{\iy}s_iz^i},
   ~~~~~~~\varphi^-(z)= e^{
  -\sum_1^{\iy}u_iz^i}.\label{extra-var}
 \eea
 The determinant of the moment matrix (\ref{Tmn_def}) with regard to
  the inner-product
  $\la f,g\ra=\int_E f(z)g(z)e^{-z^2/2}dz$, with $p=1$,
  $q=2$, $n_1=k_1+k_2$, $m_1=k_1$, $m_2=k_2$, and
   $\varphi_1(x)=1,~~\psi_1(y)=e^{ay+\beta y^2},
   ~~\psi_2(y)=1 $
   is the same as the determinant
   (\ref{1.3}), with the expressions $V(z)$ and
   $\varphi^{\pm}(z)$ as in (\ref{extra-var}), and setting
    $s_{1i}:=s_i$, $s_{2i}=u_i$,
   $t_{1i}=t_i$. Therefore by virtue of Theorem
   \ref{Theo:p+q-KP}, the expression below satisfies
   the 3-KP hierarchy, since $p+q=3$, namely,
%
%
%
%
%
   \bea
    {\tau_{k_1k_2}(t,s,u;\al,\beta;E)}
    &:=&
   \det
    \left( \!\!\begin{array}{c}
\!
\left(\displaystyle{\int_{E}}z^{i+j}e^{-\frac{z^2}{2}+\al
z+\beta
 z^2}e^{\sum_1^{\iy}(t_k-s_k)z^k}dz\!\!
  \right)_{\!\!\scriptsize{\begin{array}{l}
      0\leq i\leq k_1-1\\
      0\leq j\leq k_1\!+\!k_2\!-\!1
      \end{array}}}
  \\
 \!\!\!\left(\displaystyle{\int_{E}}z^{i+j}e^{-\frac{z^2}{2}
 }
 e^{\sum_1^{\iy}(t_k-u_k)z^k}dz
 \right)_{\scriptsize{\begin{array}{l}
       0\leq i\leq k_2-1\\
      0\leq j\leq k_1 + k_2 - 1
      \end{array}}}
 \end{array}\!\!\!\!\!\!\right)
%
 \no\\
 \label{1.5}  &=&\frac{1}{k_1!k_2!}
  \int_{E^{n }}\Dt_{n }(x,y)
  \prod_{j=1}^{k_1}e^{\sum_1^{\iy}t_i x_j^i  }
   \prod_{j=1}^{k_2}e^{\sum_1^{\iy}t_i y^i_j  }
  \\
& &~~~  \left(\Dt_{k_1}(x)\prod^{k_1}_{j=1}
  e^{-\frac{x_j^2}{2}+\al x_j+\beta
x_j^2 }
e^{-\sum_1^{\iy}s_ix^i_j}dx_j\right)
 \no\\ &&~~~
\left(\Dt_{k_2}(y)    \prod^{k_2}_{j=1}
e^{-\frac{y_j^2}{2}
} e^{-\sum_1^{\iy}u_iy^i_j}dy_j\right)
  \no.~~~~~\eea
%

\begin{corollary}\label{bilinear-identity}
%

The functions $\tau_{k_1k_2}(t,s,u)$ satisfy the
identities
%
%
%
%
 {\begin{equation}
  \frac{\partial}{\partial
t_{1}} \log \frac{\tau_{k_{1}+1, k_{2}}}{\tau_{k_{1}-1,
k_{2}}} \ = \frac{  \frac{\partial^{ 2}}{\partial t_{2}
\partial s_{1}}  \log  \tau_{k_{1},  k_{2}}
 }{  \frac{\partial^{ 2}}{\partial t_{1}
\partial s_{1}}   \log  \tau_{k_{1}, \,k_{2}}  },~~~
%
-  \frac{\partial}{\partial s_{1}} \log
\frac{\tau_{k_{1}+1, \,k_{2}}}{\tau_{k_{1}-1, \,k_{2}}}
   =  \frac{  \frac{\partial^{ 2}}{\partial t_{1}
\partial s_{2}}  \log  \tau_{k_{1},  k_{2}}  }{
\frac{\partial^{ 2}}{\partial t_{1}
\partial s_{1}}   \log  \tau_{k_{1}, k_{2}} }
 \label{C008}
 \end{equation}
 \begin{equation}
\frac{\partial}{\partial t_{1}} \log \frac{\tau_{k_{1},
k_{2}+1}}{\tau_{k_{1},  k_{2}-1}}
  = \frac{  \frac{\partial^{ 2}}{\partial t_{2}
\partial u_{1}}  \log  \tau_{k_{1},  k_{2}}
 }{  \frac{\partial^{ 2}}{\partial t_{1}
\partial u_{1}}  \log  \tau_{k_{1},  k_{2}}  }
%
,~~~~- \frac{\partial}{\partial u_{1}} \log
\frac{\tau_{k_{1}, k_{2}+1}}{\tau_{k_{1},  k_{2}-1}}
  = \frac{  \frac{\partial^{2}}{\partial t_{1}
\partial u_{2}}  \log  \tau_{k_{1},  k_{2}}  }{
\frac{\partial^{ 2}}{\partial t_{1}
\partial u_{1}}   \log  \tau_{k_{1},  k_{2}}  } .\label{C010}
\end{equation}
}

\end{corollary}

\begin{proof}
 The
bilinear identities (\ref{tau-bilinear identity}) imply
the PDE's (\ref{introa}) for $$
\tau_{k_1k_2}(t,s,u):=\tau_{k_1,k_2,k_1+k_2}(t,s,u),$$
(in the notation of (\ref{Tmn_def}), setting
$m=(k_1,k_2)$ and $n=k_1+k_2$) expressed in terms of
Hirota's symbol, for $j=0,1,2,\ldots$,
{
\bea
\gs_j(\tilde\pl_t)\tau_{k_1+1,k_2}\circ\tau_{k_1-1,k_2}&=&
 -\tau^2_{k_1k_2}
\frac{\pl^2}{\pl  s_1\pl t_{j+1}}\log\tau_{k_1k_2}
  \label{bilinear identity1}
\\
\gs_j(\tilde\pl_s)\tau_{k_1-1,k_2}\circ\tau_{k_1+1,k_2}&=&
 -\tau^2_{k_1k_2}
\frac{\pl^2}{\pl t_1\pl s_{j+1}}\log\tau_{k_1k_2}
  \label{bilinear identity2}
 \no\\
 \gs_j(\tilde\pl_t)\tau_{k_1,k_2+1}\circ\tau_{k_1,k_2-1}&=&
 -\tau^2_{k_1k_2}
\frac{\pl^2}{\pl  u_1\pl t_{j+1}}\log\tau_{k_1k_2}
 \no\\
\gs_j(\tilde\pl_u)\tau_{k_1,k_2-1}\circ\tau_{k_1,k_2+1}&=&
 -\tau^2_{k_1k_2}
\frac{\pl^2}{\pl t_1\pl u_{j+1}}\log\tau_{k_1k_2}. \no\eea}
In particular for $j=0$, one finds the following
expressions
 \bea \frac{\pl^2\log \tau_{k_1 ,k_2 }}{\pl
t_1\pl s_1} &=&
   -\frac{\tau_{k_1 +1,k_2 }\tau_{k_1 -1,k_2 }}
         {\tau_{k_1 ,k_2 }^2}
       \no \\
         \frac{\pl^2\log \tau_{k_1 ,k_2 }}{\pl
t_1\pl u_1} &=&
   -\frac{\tau_{k_1 ,k_2+1 }\tau_{k_1 ,k_2-1 }}
         {\tau_{k_1 ,k_2 }^2}
\label{14}
 \eea
 and another set of expressions for $j=1$. Then taking
 appropriate ratios of these expressions yields the
 formulae of Corollary \ref{bilinear-identity}.\end{proof}



\section{Virasoro constraints}\label{sect:Vir}

 Define the differential operators $\mathcal{B}_m$ involving the boundary points of the set $E$,
$$
\mathcal{B}_m := \sum_{i=1}^{2r} b_i^{m+1} \frac{\pl}{\pl
b_i} , \quad\mbox{for} \quad
E=\bigcup_{i=1}^{r}~[b_{2i-1},b_{2i}] \subset {\mathbb R}  ,$$
 and the differential operators $\BV_{-1}$ and $\BV_{0}$ involving differentiation with respect to the auxiliary variables $t_i, ~s_i, ~u_i$ and $\beta$:
 \bean
\BV_{-1}&:=&  - \frac{\partial}{\partial
t_{1}} - 2 \beta \frac{\partial}{\partial s_{1}} + \sum_{i\geq2} \left( it_{i} \frac{\partial}{\partial
t_{i-1}} + is_{i} \frac{\partial}{\partial s_{i-1}} +
iu_{i} \frac{\partial}{\partial u_{i-1}} \right)
\\
 &&+k_1(t_1-s_1)+k_2(t_1-u_1)
+\alpha  k_{1}
\\
\BV_{0}&:=&   -\frac{\partial  }{\partial
t_{2}}    - 2\beta\frac{\partial }{\partial
s_{2}} - \alpha\frac{\partial  }{\partial s_{1}}
  + \sum_{i\geq1} \left( it_{i} \frac{\partial}{\partial
t_{i}} + is_{i} \frac{\partial}{\partial s_{i}} + iu_{i}
\frac{\partial}{\partial u_{i}} \right)
\\
&&+ k_{1}^{2} + k_{2}^{2} + k_{1} k_{2}
 \label{C005}.
\eean

 \begin{theorem}\label{Theo:Virasoro}
The integral \,$\tau_{k_1k_2}(t,s,u ; \alpha, \beta ;
E)$\, defined in (\ref{1.5}) satisfies :
 \be
  \mathcal{B}_m \,\tau_{k_1k_2} \,=\, \BV_m
  \tau_{k_1k_2} ~~~~~\mbox{for}~~ m=-1 \mbox{ and } 0,
   \label{Virasoro constraints}\ee
  with $\BV_m$ as in (\ref{C005}).  \end{theorem}

\begin{proof}  Let  $I_{k_1k_2} (x,y) $ be the integrand of the integral $\tau_{k_1k_2}(t,s,u ; \alpha, \beta ; E)$\, defined in (\ref{1.5}):
$$
\tau:=\tau_{k_1k_2}(t,s,u ; \alpha, \beta ; E)=\int_{E^n} I_{k_1k_2} (x,y) \prod_1^{k_1}dx_i
 \prod_1^{k_2}dy_j.
 $$
To obtain (\ref{Virasoro constraints}), one uses the fundamental Theorem of calculus and the commutation relation $$\sum_1^n (\frac{\pl}{\pl x_i} x_i- x_i \frac{\pl}{\pl x_i} )=n;$$ one uses the fact that for any Vandermonde, $$\sum_1^k \frac{\pl}{\pl x_i}\Dt(x_1,\ldots,x_k)=0$$  and $$\sum_1^k  x_i\frac{\pl}{\pl x_i}\Dt(x_1,\ldots,x_k)=\frac{n(n-1)}{2}\Dt(x_1,\ldots,x_k).$$  Also one uses the fact that the sum of the $x_i$ and $y_i$ derivatives of the integrand $I_{k_1,k_2}$ translates into the $t_i,~s_i$ and $u_i$-derivatives of $I_{k_1,k_2}$, i.e., the auxiliary parameters $t,s,u ; \alpha, \beta$ were precisely added for this very purpose! So we find:
\bean
 \mathcal{B}_{-1}\tau
 &=&  \int_{E^n}  \left(\sum_1^{k_1}\frac{\pl}{\pl x_j}+
     \sum_1^{k_2}\frac{\pl}{\pl y_j}\right) I_{k_1,k_2}    \prod_1^{k_1}dx_i     \prod_1^{k_2}dy_j
      \\
    &=&   \int_{E^n}  \BV_{-1} (I_{k_1,k_2})    \prod_1^{k_1}dx_i     \prod_1^{k_2}dy_j
    \\
    &=&   \BV_{-1} \left( \int_{E^n}  I_{k_1,k_2} \prod_1^{k_1}dx_i     \prod_1^{k_2}dy_j\right)
=   \BV_{-1}   \tau
 \eean
   and
 \bean
 \mathcal{B}_{0}\tau
 &=&  \int_{E^n}  \left(\sum_1^{k_1}\frac{\pl}{\pl x_j}x_j+
     \sum_1^{k_1}\frac{\pl}{\pl y_j}y_j\right) I_{k_1,k_2}    \prod_1^{k_1}dx_i     \prod_1^{k_2}dy_j
      \\
    &=&   \int_{E^n} \BV_{0} (I_{k_1,k_2})    \prod_1^{k_1}dx_i     \prod_1^{k_2}dy_j
    \\
    &=&    \BV_{0} \left( \int_{E^n}  I_{k_1,k_2} \prod_1^{k_1}dx_i     \prod_1^{k_2}dy_j\right)
=  \BV_{0}  \tau,
 \eean
  establishing Theorem \ref{Theo:Virasoro}. Aother way of computing this has appeared in [2].\end{proof}

  \bigbreak

\noindent We also have the following identities, valid
when acting on $\tau_{k_{1}k_{2}} (t,s,u;\alpha, \beta;
E)$:
\begin{equation}
 \frac{\pl}{\pl t_n} = -\frac{\pl}{\pl s_n} - \frac{\pl}{\pl u_n}
, \label{C046} \end{equation}
\be
\frac{\partial}{\partial s_{1}} =
-\frac{\partial}{\partial \alpha} ,~~~~~
\frac{\partial}{\partial t_{1}} =
\frac{\partial}{\partial \alpha} -
\frac{\partial}{\partial u_{1}},~~~~~
\frac{\partial}{\partial s_{2}} =
-\frac{\partial}{\partial \beta}, ~~~~~
\frac{\partial}{\partial t_{2}} =
\frac{\partial}{\partial \beta} -
\frac{\partial}{\partial u_{2}}  \label{C004} .\ee

\begin{corollary}
On the locus $\mathcal{L} = \{ t = s = u = \beta = 0
\}$, the function \\ $f := \log \tau_{k_{1}k_{2}}
(t,s,u;\alpha, \beta; E)$ satisfies the Virasoro
constraints :
 \bea
\frac{\partial f}{\partial s_{1}} &=& -\frac{\partial
f}{\partial \alpha} ,\qquad\qquad \frac{\partial
f}{\partial s_{2}} = -\frac{\partial f}{\partial \beta}
\label{C003}
\\
 \frac{\partial f}{\partial t_{1}} &=&  -
\mathcal{B}_{-1} f + \alpha k_{1},\qquad
\frac{\partial f}{\partial u_{1}} = \left(
\mathcal{B}_{-1} + \frac{\partial}{\partial \alpha}
\right) f - \alpha k_{1}
\no\\
\frac{\partial f}{\partial t_{2}} &=& \left(
-\mathcal{B}_{0} + \alpha \frac{\partial}{\partial
\alpha} \right) f + \bigg( k_{1}^{2} + k_{2}^{2} + k_{1}
k_{2} \bigg) ~~~~~~
\no\\
\frac{\partial f}{\partial u_{2}} &=& \left(
\mathcal{B}_{0} - \alpha \frac{\partial}{\partial
\alpha} + \frac{\partial}{\partial \beta} \right) f -
\bigg( k_{1}^{2} + k_{2}^{2} + k_{1} k_{2} \bigg)
\no%
\eea
\vspace*{-.5cm}
\bea
  \label{C002}
  \frac{\partial^{2} f}{\partial t_{1} \partial
u_{1}}& =& -\mathcal{B}_{-1} \left( \mathcal{B}_{-1} +
\frac{\partial}{\partial \alpha} \right) f - k_{2}
 \\
\frac{\partial^{2} f}{\partial t_{1} \partial u_{2}}& =&
-\mathcal{B}_{-1} \left( \mathcal{B}_{0} -
\alpha\frac{\partial}{\partial \alpha} +
\frac{\partial}{\partial \beta} \right) f + 2 \left(
\mathcal{B}_{-1} + \frac{\partial}{\partial \alpha}
\right) f - 2\alpha k_{1}\no\\
%
\frac{\partial^{2} f}{\partial t_{2} \partial u_{1}}& =&
\left( -\mathcal{B}_{0} + \alpha\frac{\partial}{\partial
\alpha} + 1 \right) \left( \mathcal{B}_{-1} +
\frac{\partial}{\partial \alpha} \right) f  - 2\alpha
k_{1} ~~~~~~~~~~~~~~~~~\no \eea
\vspace*{-.5cm}
\bea \label{C001}
\frac{\partial^{2} f}{\partial t_{1} \partial s_{1}}
&=& \mathcal{B}_{-1} \left( \frac{\pl}{\pl \alpha} \right)
f - k_{1}
\\%
 \frac{\partial^{2} f}{\partial t_{1} \partial
s_{2}}& =&\left( \mathcal{B}_{-1}  
\frac{\partial}{\partial \beta}   - 2  
\frac{\partial }{\partial \alpha}
 \right)f  \no\\
  \frac{\partial^{2} f}{\partial t_{2} \partial
s_{1}} &=& \left( \mathcal{B}_{0} - \alpha
\frac{\partial}{\partial \alpha} -1 \right) \left(
\frac{\partial}{\partial \alpha} \right) f ~.~~~
 \no\eea
\end{corollary}

\begin{proof} Upon dividing by $\tau$, equations (\ref{C003})
are a direct consequence of (\ref{C005}) and
(\ref{C004}), when evaluated on the locus $\mathcal{L}$. To derive
equations (\ref{C002}) and (\ref{C001}), we use the fact
that the boundary operators
$$\mathcal{B}_{m}=\sum_{j=1}^{2r} b_{j}^{m+1}
\frac{\partial}{\partial b_{j}}$$ commute with pure
\mbox{time-differential} operators. For example, the
calculation of the first equation in (\ref{C002}) goes
as follows. We know from (\ref{C005}) and (\ref{C004})
that :
$$
-\mathcal{B}_{-1} f = \left( \frac{\partial f}{\partial
t_{1}} - \alpha k_{1} + L_{1}f + \ell_{1} \right) ,~~~
\left( \mathcal{B}_{-1} + \frac{\partial}{\partial
\alpha} \right) f = \left( \frac{\partial f}{\partial
u_{1}} + \alpha k_{1} + L_{2}f + \ell_{2} \right) ,
$$
where the $L_{i}$ are linear differential operators
vanishing on $\mathcal{L}$ (commuting with
$\mathcal{B}_{m}$) and the $\ell_{i}$ are functions
vanishing on $\mathcal{L} ~~ (i=1,2)$. Therefore:
\bean
 \lefteqn{-\mathcal{B}_{-1} \left( \mathcal{B}_{-1} +
\frac{\partial}{\partial \alpha} \right) f
\Bigr|_{\mathcal{L}}}\\&&= -\mathcal{B}_{-1} \left(
\frac{\partial f}{\partial u_{1}} + \alpha k_{1} + L_{2}
f + \ell_{2} \right) \Bigr|_{\mathcal{L}}\\
 &&= -\mathcal{B}_{-1}
\left( \frac{\partial f}{\partial u_{1}} + L_{2} f
\right) \Bigr|_{\mathcal{L}} = - \left(
\frac{\partial}{\partial u_{1}} + L_{2} \right)
\mathcal{B}_{-1} f \Bigr|_{\mathcal{L}} = - \left(
\frac{\partial}{\partial u_{1}} \right) \mathcal{B}_{-1}
f \Bigr|_{\mathcal{L}}
\\
&&= \frac{\partial}{\partial u_{1}} \left(
\frac{\partial f}{\partial t_{1}} - \alpha k_{1} +
L_{1}f + \ell_{1} \right) \Bigr|_{\mathcal{L}} = \left(
\frac{\partial^{2} f}{\partial t_{1} \partial u_{1}} +
\frac{\partial}{\partial u_{1}} L_{1}f + \frac{\partial
\ell_{1}}{\partial u_{1}} \right) \Bigr|_{\mathcal{L}}.
\eean


Since  $$L_{1} = -2 \beta
\frac{\partial}{\partial \alpha} - \sum_{i\geq2} \left(
i t_{i} \frac{\partial}{\partial t_{i-1}} + i s_{i}
\frac{\partial}{\partial s_{i-1}} + i u_{i}
\frac{\partial}{\partial u_{i-1}} \right)$$ and
$\ell_{1} = k_{1} s_{1} + k_{2} u_{1} -
\left( k_{1} + k_{2} \right) t_{1} $, one checks
that, along the locus $\LR$, one has $$\frac{\partial}{\partial u_{1}}L_1f=0\qquad \mbox{    ~~~and~~~   }\qquad\frac{\partial
\ell_{1}}{\partial u_{1}} =k_2,$$ yielding the expression for $\frac{\partial^{2} f}{\partial t_{1} \partial
u_{1}}$ in (\ref{C002}).
%
A similar procedure \mbox{applies} to establish all the
identities above. \end{proof}


\section{A PDE for the Gaussian ensemble with \mbox{external source}}
\label{sect:PDE}

Consider the Gaussian Hermitian random matrix ensemble
$\mathcal {H}_n$ with external source $A$, given by the
diagonal matrix (\ref{diagmatrix}) (set $n=k_1+k_2$)
 %
  and density
  \be
  \frac{1}{Z_n} 
  e^{-\Tr(\frac{1}{2} M^2-AM) }dM.
  \ee
    Given a disjoint union of intervals
 $
 E:= \bigcup^r_{i=1}[b_{2i-1},b_{2i}]\subset {\mathbb R}
   ,$
%
%
%
define the algebra of differential operators, generated
by
 \be\mathcal{B}_k=\sum_{i=1}^{2r}b_i^{k+1}\frac{\pl}{\pl b_i}.\ee
Consider the following probability:
  \be
\BP_n(\al;E):=\BP(~\mbox{all eigenvalues}~\in E) =
 \frac{1}{Z_n} \int_{\mathcal {H}_{n}(E)}
   e^{-\Tr(\frac{1}{2} M^2-AM) }dM,
  \label{0.4}
   \ee
where $\mathcal {H}_{n}(E)$ is the set of all Hermitian
matrices with all eigenvalues in $E$. The purpose of
this section is to prove the following theorem:

\begin{theorem} \label{Theo:PDE}
The log of the probability $\BP_n(a;E)$
   satisfies
  a  \underline {fourth-order} PDE in $a$ and in the endpoints $b_1,..., b_{2r}$
  of the set $E$, with \underline {quartic
  non-linearity}:
\bea && \bigg( F^{+}    \mathcal{B}_{-1} G^{-}   + F^{-}
  \mathcal{B}_{-1} G^{+}   \bigg) \bigg( F^{+}
\mathcal{B}_{-1} F^{-}   - F^{-}    \mathcal{B}_{-1}
F^{+}   \bigg)
\no\\
&&- \bigg( F^{+} G^{-} + F^{-} G^{+} \bigg) \bigg( F^{+}
   \mathcal{B}_{-1}^{ 2} F^{-}   - F^{-}
\mathcal{B}_{-1}^{ 2} F^{+}   \bigg)  =  0 \label{C017a}
, \eea
 or what is the same
 \begin{equation}
\det \left( \begin{array}{cccc}
 G^+& \mathcal{B}_{-1}F^+  & - F^+ &0\\
 -G^-& \mathcal{B}_{-1}F^-  & - F^- &0\\
 \mathcal{B}_{-1} G^+ &\mathcal{B}_{-1}^2F^+&0&-F^+\\
 -\mathcal{B}_{-1} G^- &\mathcal{B}_{-1}^2F^-&0&-F^-\\
 \end{array}\right) =0 , \label{C021}
\end{equation}
  where
$$
  F^{+}  = \mathcal{B}_{-1}
 \frac{\partial}{\partial \alpha}   \log
\BP_{n} - k_{1} ,\qquad
 F^{-} =  -\mathcal{B}_{-1} \left( \mathcal{B}_{-1}  +
\frac{\partial}{\partial \alpha} \right) \log  \BP_{n} -
k_{2}
$$

\vspace*{-.6cm}

\bean 
  H_{1}^{+} & =& 4  \frac{\partial}{\partial
\alpha}   \log \BP_{n} + 4 \alpha k_{1} +
\frac{4k_{1}k_{2}}{\alpha}
 ,\\
 H_{1}^{-}
 &=&  -  2 \left( \mathcal{B}_{0} - \alpha
\frac{\partial}{\partial \alpha} + 1 \right)
\frac{\partial}{\partial \alpha}   \log \BP_{n}
 -   \frac{4k_{1}k_{2}}{\alpha}
\\
H_{2}^{+} &=& 2 \left( \!\mathcal{B}_{0}\! -\! \alpha
\frac{\partial}{\partial \alpha}\! -\! 1\! - 2 \alpha
\mathcal{B}_{-1} \!\right) \frac{\partial}{\partial
\alpha}   \log \BP_{n}
 \\
 H_{2}^{-} &=&  -  2 \left( \mathcal{B}_{0}\! -\! \alpha
\frac{\partial}{\partial \alpha} \!-\! 1 \right)   \bigg(
\mathcal{B}_{-1} \!+\! \frac{\partial}{\partial \alpha}
\bigg) \log
 \BP_{n}
 \\ \
 2G^{\pm}&=&\left\{H_1^{\pm},F^{\pm}\right\}_{\mathcal{B}_{-1}}\mp
 \left\{H_2^{\pm},F^{\pm}\right\}_{\pl/ \pl \al} .
 \eean 
\end{theorem}

\begin{remark}   It is not surprising that the PDE
(\ref{C017a}) has exactly \emph{the same form} as the
PDE \mbox{derived} in \cite{AvM-Pearcey} and
\cite{AvM-Pearcey-joint}, associated to the Gaussian
Unitary Ensemble with an \mbox{external} source, in the
case where the source matrix admits \emph{two}
eigenvalues of \mbox{\emph{opposite}} signs. The only
difference is that the expressions for the functions
$F^{\pm}, H_{1}^{\pm}, H_{2}^{\pm}$ and $G^{\pm}$
\mbox{obtained} here, differ from those in
\cite{AvM-Pearcey}. The reason is that corresponding
$\tau$-functions satisfy the same integrable equation
(the 3-KP hierarchy, as in Section 3.3), whereas the Virasoro constraints \mbox{leading}
to (\ref{C017a}) are different. In particular, there is
no more involution relating the variables $s_{k}$ and
$u_{k}$.
\end{remark}

\begin{proof}[Proof of Theorem~\ref{Theo:PDE}]
Remember:
\begin{displaymath}  \BP_n (\alpha , E) := \BP_n (\textrm{spec}  M \subseteq E) =
\frac{ \tau_{k_1k_2} (t, s, u ; \alpha, \beta ; E) } {
\tau_{k_1k_2} (t, s, u ; \alpha, \beta ; {\mathbb R}) } \Bigr|_{
\mathcal{L} \equiv \{ t = s = u = \beta = 0 \} }.
\end{displaymath}
The denominator, that is the integral (\ref{1.5}) over
the whole range, can be expressed in terms of moments,
which contain standard {Gaussian} {integrals} (the
reader is referred to the \mbox{Appendix $1$} in
\cite{AvM-Pearcey}), leading to an exact evaluation,
with $c_{k_{1}k_{2}}$ a constant, depending on $k_{1} ,
k_{2}$ only:
\begin{equation}
\tau_{k_{1},k_{2}} (t, s, u ; \alpha, \beta ;{\mathbb R})
\Bigr|_{\mathcal{L}} = c_{k_{1}k_{2}} \alpha^{k_{1}
k_{2}}e^{ k_{1} \left( \alpha^{2} / 2 \right)}.
~~~~~~~~~~~ \label{C018}
\end{equation}
%
Consequently (with $C_{k_{1}k_{2}}$ a constant,
depending on $k_{1} , k_{2}$ only):
\begin{equation}
\log \BP_{n} (\alpha, E) = \log \tau_{k_{1},k_{2}} (t,
s, u ; \alpha, \beta ; E) \Bigr|_{\mathcal{L}} -
\frac{k_{1}}{2} \alpha^{2} - k_{1} k_{2} \log
\left(\alpha\right) - C_{k_{1}k_{2}} \label{C020}.
\end{equation}

Then, we turn our attention to the numerator
$\tau_{k_1k_2} (t, s, u ; \alpha, \beta ; E)$, more
briefly noted as $\tau_{k_{1}k_{2}}$. On the locus
$\mathcal{L}$, by (\ref{C008}), and (\ref{C003}), one
finds:
\bea \label{C011} \frac{ \frac{\partial^{2}}{\partial t_{2}
\partial s_{1}} \log \tau_{k_{1}, k_{2}}
}{ \frac{\partial^{2}}{\partial t_{1}
\partial s_{1}}  \log \tau_{k_{1}, k_{2}} } &=& \frac{\partial}{\partial t_{1}}
\log \left( \frac{\tau_{k_{1}+1, k_{2}}}{\tau_{k_{1}-1,
k_{2}}} \right) = -\mathcal{B}_{-1} \log \left(
\frac{\tau_{k_{1}+1, k_{2}}}{\tau_{k_{1}-1, k_{2}}}
\right) + 2\alpha ~
 \\
\frac{ \frac{\partial^{2}}{\partial t_{1}
\partial s_{2}} \log \tau_{k_{1}, k_{2}}
}{ \frac{\partial^{2}}{\partial t_{1}
\partial s_{1}}  \log \tau_{k_{1}, k_{2}} } &=& -\frac{\partial}{\partial s_{1}} \log \left(
\frac{\tau_{k_{1}+1, k_{2}}}{\tau_{k_{1}-1, k_{2}}}
\right) = \frac{\partial}{\partial \alpha} \log \left(
\frac{\tau_{k_{1}+1, k_{2}}}{\tau_{k_{1}-1, k_{2}}}
\right) . \no \eea
%
%
In (\ref{C011}), acting with $  {\partial}/{\partial
\alpha} $ on the first expression, then acting with
$\mathcal{B}_{-1}$ on the second expression, and adding
the two, yields:

\begin{displaymath}
\frac{\partial}{\partial \alpha} \left( \frac{
\frac{\partial^{2}}{\partial t_{2}
\partial s_{1}} \log \tau_{k_{1}, k_{2}} }{
\frac{\partial^{2}}{\partial t_{1}
\partial s_{1}}  \log \tau_{k_{1}, k_{2}} } - 2\alpha \right)
+ \mathcal{B}_{-1} \left( \frac{
\frac{\partial^{2}}{\partial t_{1}
\partial s_{2}} \log \tau_{k_{1}, k_{2}} }{
\frac{\partial^{2}}{\partial t_{1}
\partial s_{1}}  \log \tau_{k_{1}, k_{2}} } \right)
= 0 . \end{displaymath}
This identity can conveniently be expressed as
Wronskians:
\begin{equation}
\left\{ \mathcal{B}_{-1} \frac{\partial}{\partial \beta}
\log \tau_{k_{1} k_{2}} \Bigr|_{\mathcal{L}} , F^{+}
\right\} _{\mathcal{B}_{-1}} = \bigg\{ H_{1}^{+} ,
\frac{F^{+}}{2} \bigg\}_{\mathcal{B}_{-1}} - \bigg\{
H_{2}^{+} , \frac{F^{+}}{2}
\bigg\}_{\frac{\partial}{\partial \alpha}} =: G^{+},
\label{C012}
\end{equation}
where
\bea \label{C013}
 F^{+}  &:=&  \frac{\partial^{2}}{\partial t_{1}
\partial s_{1}} \log \tau_{k_{1} , k_{2}}
   \\
H_{1}^{+} &:=& -2 \frac{\partial^{2}}{\partial t_{1}
\partial s_{2}} \log \tau_{k_{1} , k_{2}} + 2\mathcal{B}_{-1}\frac{\partial}{\partial \beta} \log
\tau_{k_{1} k_{2}} \no\\
H_{2}^{+} &:=& 2 \left( \frac{\partial^{2}}{\partial
t_{2}
\partial s_{1}} \log \tau_{k_{1} , k_{2}} -2\alpha
\frac{\partial^{2}}{\partial t_{1} \partial s_{1}} \log
\tau_{k_{1} , k_{2}} \right) .~~ \no\eea
%
%
In this way, by virtue of (\ref{C001}) and (\ref{C020})
one obtains, along the locus $\mathcal{L}$, explicit
expressions for $F^{+}$, $H_{1}^{+}$ and $H_{2}^{+}$
that are free of partials in $\beta$; namely:
{  \bea
 \label{Hplus} F^{+}& =& \mathcal{B}_{-1}
\left(\frac{\partial}{\partial \alpha} \right) \log
\tau_{k_{1},k_{2}} - k_{1} = \mathcal{B}_{-1}
\left(\frac{\partial}{\partial \alpha} \right) \log
\BP_{n} - k_{1} ~~
  \\
H_{1}^{+} &=& 4 \left(\frac{\partial}{\partial \alpha}
\right) \log \tau_{k_{1},k_{2}} = 4
\left(\frac{\partial}{\partial \alpha} \right) \log
\BP_{n} + 4 \alpha k_{1} + 4 \frac{k_{1}k_{2}}{\alpha}
   \no\\
H_{2}^{+}& =& 2 \left( \mathcal{B}_{0} - \alpha
\frac{\partial}{\partial \alpha} - 1 - 2 \alpha
\mathcal{B}_{-1} \right)  \bigg(
\frac{\partial}{\partial \alpha} \bigg) \log
\tau_{k_{1},k_{2}} + 4 \alpha k_{1} ~~
   \no\\
 &=&
 2 \left( \mathcal{B}_{0}\!  -  \!\alpha
\frac{\partial}{\partial \alpha}\!  -  \!1 \! -  \!2
\alpha \mathcal{B}_{-1} \right)  \bigg(
\frac{\partial}{\partial \alpha} \bigg) \log \BP_{n} .
 \no \eea}
Subsequently, one repeats exactly the same operations
for the Virasoro and KP-identities
involving the $t_{k}$ and $u_{k}$ variables. From
(\ref{C010}) and (\ref{C003}), one finds on
$\mathcal{L}$:
\bea \label{C014}
&&\hspace*{.7cm}\frac{ \frac{\partial^{2}}{\partial t_{2}
\partial u_{1}} \log \tau_{k_{1}, k_{2}}
}{ \frac{\partial^{2}}{\partial t_{1}
\partial u_{1}}  \log \tau_{k_{1}, k_{2}} } = \frac{\partial}{\partial t_{1}}
\log \left( \frac{\tau_{k_{1}, k_{2}+1}}{\tau_{k_{1},
k_{2}-1}} \right) = -\mathcal{B}_{-1} \log \left(
\frac{\tau_{k_{1}, k_{2}+1}}{\tau_{k_{1}, k_{2}-1}}
\right)   \\
&&\hspace*{.7cm}\frac{ \frac{\partial^{2}}{\partial t_{1}
\partial u_{2}} \log \tau_{k_{1}, k_{2}}
 }{  \frac{\partial^{ 2}}{\partial t_{1}
\partial u_{1}}   \log  \tau_{k_{1},  k_{2}} } = - \frac{\partial}{\partial u_{1}} \log \left(
\frac{\tau_{k_{1},  k_{2}+1}}{\tau_{k_{1},  k_{2}-1}}
\right) = - \left( \mathcal{B}_{-1} +
\frac{\partial}{\partial \alpha} \right) \log \left(
\frac{\tau_{k_{1},  k_{2}+1}}{\tau_{k_{1},  k_{2}-1}}
\right) . \no\eea

So, in (\ref{C014}), acting on the first equation with
$\left( \mathcal{B}_{-1} + \frac{\partial}{\partial
\alpha} \right)$ and on the second equation with
$\mathcal{B}_{-1}$ , then subtracting the two yields:
\begin{displaymath}
\mathcal{B}_{-1} \left( \frac{  \left( \frac{\partial^{
2}}{\partial t_{2} \partial u_{1}} - \frac{\partial^{
2}}{\partial t_{1} \partial u_{2}} \right) \log
 \tau_{k_{1},  k_{2}}
  }{  \frac{\partial^{2}}{\partial t_{1}
\partial u_{1}}   \log  \tau_{k_{1},  k_{2}}  } \right)
 +  \frac{\partial}{\partial \alpha} \left( \frac{
\frac{\partial^{ 2}}{\partial t_{2}
\partial u_{1}}  \log  \tau_{k_{1}, k_{2}}  }{
\frac{\partial^{ 2}}{\partial t_{1}
\partial u_{1}}  \log  \tau_{k_{1},  k_{2}}  } \right)
= 0  , \end{displaymath}
or equivalently (remember the brackets are Wronskians) :
\begin{equation}
-\left\{ \mathcal{B}_{-1} \frac{\partial}{\partial
\beta} \log \tau_{k_{1} k_{2}} \Bigr|_{\mathcal{L}}
,F^{-} \right\} _{\mathcal{B}_{-1}} =  \bigg\{ H_{1}^{-}
, \frac{F^{-}}{2} \bigg\}_{\mathcal{B}_{-1}} + \bigg\{
H_{2}^{-} , \frac{F^{-}}{2}
\bigg\}_{\frac{\partial}{\partial \alpha}} \,\,=:\,\,
G^{-} \label{C015} \end{equation}
in terms of the functions
{
\bea
   \label{C016}  F^{-} &:=&
\frac{\partial^{\,2}}{\partial t_{1}
\partial u_{1}} \,\log \tau_{k_{1} , k_{2}}
 \\
H_{1}^{-} &:=&  2  \left( \,\frac{\partial^{
2}}{\partial t_{2}
\partial u_{1}} -
\frac{\partial^{ 2}}{\partial t_{1} \partial u_{2}}
\right) \log \tau_{k_{1},k_{2}} - 2 \mathcal{B}_{-1}
\frac{\partial}{\partial \beta} \log
\tau_{k_{1} k_{2}}  \no\\
 H_{2}^{-} &:=& 2  \left(
\frac{\partial^{ 2}}{\partial t_{2}
\partial u_{1}} \right)  \log \tau_{k_{1} , k_{2}}  .
 \no\eea}
Using the Virasoro constraints \,(\ref{C002}), \,as well
as (\ref{C020}), we obtain \mbox{explicit} formulae for
\,$F^{-}$, \,$H_{1}^{-}$ and $H_{2}^{-}$ (which do not
contain partials in $\beta$):
    %
 \bea
  F^{-} &=&
-\mathcal{B}_{-1} \left( \mathcal{B}_{-1} +
\frac{\partial}{\partial \alpha} \right) \log
\tau_{k_{1},k_{2}} - k_{2} 
\no\\
&=& -\mathcal{B}_{-1} \left(
\mathcal{B}_{-1} + \frac{\partial}{\partial \alpha}
\right) \log  \BP_{n} - k_{2}  \no\\
H_{1}^{-} &=& -  2 \left( \mathcal{B}_{0} - \alpha
\frac{\partial}{\partial \alpha} + 1 \right) \bigg(
\frac{\partial}{\partial \alpha} \bigg) \log
\tau_{k_{1},k_{2}}  \no\\
  &=&
   -  2 \left(
\mathcal{B}_{0} - \alpha \frac{\partial}{\partial
\alpha} + 1 \right) \bigg( \frac{\partial}{\partial
\alpha} \bigg) \log  \BP_{n} - 4
 \frac{k_{1}k_{2}}{\alpha}\no\\
H_{2}^{-}  &=&  -  2 \left( \mathcal{B}_{0} - \alpha
\frac{\partial}{\partial \alpha} - 1 \right)   \bigg(
\mathcal{B}_{-1} + \frac{\partial}{\partial \alpha}
\bigg) \log \tau_{k_{1},k_{2}} - 4  \alpha  k_{1}
 \no\\ 
  &=&
  -  2 \left( \mathcal{B}_{0} - \alpha
\frac{\partial}{\partial \alpha} - 1 \right) \, \bigg(
\mathcal{B}_{-1} + \frac{\partial}{\partial \alpha}
\bigg) \log
 \BP_{n} .\label{Hminus} \eea 
Equations (\ref{C012}) and  (\ref{C015})\, form a linear
system in  $$\BB_{-1}  \frac{\partial  \log
 \tau_{k_{1} k_{2}}}{\partial  \beta}
\Bigr|_{\mathcal{L}} \qquad\mbox{and}\qquad\BB_{-1}^{2}
 \frac{\partial  \log  \tau_{k_{1} k_{2}}}{\partial
 \beta} \Bigr|_{\mathcal{L}},$$ which can be solved to
extract the quantities:
\bea \mathcal{B}_{-1} \,\frac{\partial \log
\tau_{k_{1},k_{2}}}{\partial \beta} \Bigr|_{\mathcal{L}}
&=& \frac{ G^{-} F^{+} + G^{+} F^{-} }{ -F^{-} \left(
\mathcal{B}_{-1} F^{+} \right)  +\, F^{+} \left(
\mathcal{B}_{-1} F^{-} \right)  }   \label{C034} \\
\mathcal{B}_{-1}^{ 2} \,\frac{\partial \log
\tau_{k_{1},k_{2}}}{\partial \beta} \Bigr|_{\mathcal{L}}
 &=&  \frac{ G^{-} \left( \mathcal{B}_{-1} F^{+}
\right) + G^{+} \left( \mathcal{B}_{-1} F^{-} \right) }{
-F^{-} \left( \mathcal{B}_{-1} F^{+} \right) + F^{+}
\left( \mathcal{B}_{-1} F^{-} \right)  }  . \label{C035}
\eea
Finally, subtracting the second relation from
 $\mathcal{B}_{-1}$ of the first equation,
 establishes the expected PDE (\ref{C017a}) and Theorem \ref{Theo:PDE}.
%
%
%
To prove the second equation (\ref{C021}), set 
$$
 X:=\left. \mathcal{B}_{- 1}\frac{\pl  }{\pl \beta }
    \log \tau_{k_1k_2}\right|_{\LR};
 $$
then the matrix in (\ref{C021}) annihilates the column
$(
      1,
      X,
      \mathcal{B}_{-1} X,
      \mathcal{B}_{-1}^2X
      )^{\top}$, and thus the determinant vanishes, concluding the proof of Theorem
\ref{Theo:PDE}.\end{proof}

\section{A PDE for the transition probability of the $r$-Airy process}
\label{sect6}


\begin{proof}[Proof of Theorem \ref{Theo2}]
 Remember from section \ref{sect1}
 the identity
 \begin{multline*}
\lim_{n\rightarrow \iy}
\BP^{(0,\rho\sqrt{n/2})}_{Br}\left(\mbox{all}~
x_i\left(\frac{1}{1+e^{-2\tau/n^{1/3}}}\right) \in
  \frac{(-\iy, \sqrt{2n}+\frac{x}{\sqrt{2}
  n^{1/6}})}{2\cosh( \tau/n^{1/3})}\right) = \\*
     \lim_{n\rightarrow \iy}
  \BP_{n}\left({\rho}\sqrt{n}e^{\tau /n^{1/3}};
  2\sqrt{n} +\frac{(-\iy,x)}{n^{1/6}}   \right),%
  \end{multline*}
  which for $0\leq \rho\leq 1$ leads by
 Theorem~\ref{Theo1} to a phase transition at $\rho=1$, for which
  the expression above reduces to
  $\BP(\sup \mathcal {A}^{(r)}(\tau)\leq x)$,
  according to ({\ref{r-AiryP}). The
  above scaling suggests the choice $z=n^{-1/6}$ as small
  parameter and considering the map
  $ (\tau,x) \longmapsto(\al,b) , $
given by \bean
 \al={\rho}\sqrt{n}e^{\tau /n^{1/3}}=
    \frac{\rho}{z^3}e^{\tau z^2},\qquad
 b=
   2\sqrt{n} +\frac{x}{n^{1/6}}
   =\frac{2}{z^3}+xz  
 \eean
 with inverse map
  $(\al,b)\longmapsto (\tau,x),  $
 given by
 $$
 \tau=\frac{1}{z^2}\log\left(\frac{\al
 z^3}{\rho}\right),~~~
 x =\frac{b}{z}-\frac{2}{z^4}
 .$$
 Setting
\bea  \label{C024}
 \tilde Q(\tau,x) &:=&
    \log\BP_{n}\left({\rho}\sqrt{n}e^{\tau /n^{1/3}};
  2\sqrt{n} +\frac{(-\iy,x)}{n^{1/6}}   \right)
    \\
   &=&
  \log\BP_{n}\left(\frac{\rho}{z^3}e^{\tau z^2};
   \frac{2}{z^3}+(-\iy,x)z  \right)
   ,\no\eea
yields, via the inverse map, \be
 \tilde Q\left(\frac{1}{z^2}\log\bigl(\frac{\al
 z^3}{\rho}\bigr),\frac{b}{z}-\frac{2}{z^4}
 \right)=\log\BP_n(\al,(-\iy,b)).\label{C025}
  \ee
From Corollary \ref{cor:estimate2} (section \ref{sect2}) it follows that for $z\rg 0$,
 \be
 \tilde
 Q(\tau,x)=Q(\tau,x)+O(z^2),
 \label{6.3}\ee
with $Q(\tau,x)$ independent of $z$.

Since we have shown that $\log \BP_{n} (\alpha,
(-\infty, b))$ satisfies the PDE (\ref{C017a}) of
Theorem~\ref{Theo:PDE}, with $\mathcal{B}_{-1} \equiv
\frac{\partial}{\partial b}$ and $\mathcal{B}_{0} \equiv
b \frac{\partial}{\partial b}$ , it follows that, to get
a PDE for the limiting case, we just need to estimate:
{ \footnotesize \begin{equation}
\left.\left\{\begin{array}{l} \Bigl(F^ + \BB_{-1}G^-+F^-
\BB_{-1}G^+
\Bigr)     \Bigl(F^+ \BB_{-1}F^- -F^- \BB_{-1}F^+ \Bigr) \\
 - \Bigl(F^+ G^- +F^- G^+
\Bigr)   \Bigl(F^+ \BB_{-1}^2F^- -F^- \BB_{-1}^2F^+
\Bigr)
\end{array}\right\}
 \right|_{\footnotesize\begin{array}{l}
     \alpha  \mapsto  \left( \rho / z^{ 3} \right)
      e^{ \tau  z^{2} } \\
     b \,\mapsto  x z  +  \frac{2}{z^{ 3}} \\
     n  \mapsto  \frac{1}{z^{ 6}}\end{array} }
\label{C027}.\end{equation} }
To do this, the various expressions in the bracket must
be computed in terms of the function $\tilde Q (\tau, x)
$.
By (\ref{C024}) and (\ref{C025}), one immediately gets:
$$
 \frac{\partial  \log \BP_{n}}{\partial  \alpha}
= \frac{1}{\alpha z^2} \bigg(\frac{\partial  \tilde
Q}{\partial \tau}\bigg) = \frac{z}{\rho} e^{- \tau z^2}
\bigg( \frac{\partial  \tilde Q}{\partial \tau} \bigg)
 , \qquad \frac{\partial  \log
 \BP_{n}}{\partial  b} = \frac{1}{z}
 \bigg(\frac{\partial \tilde Q}{\partial
  x}\bigg) .
$$
Hereafter, to shorten notation, we will write the partials as 
$$
\tilde Q_{\tau x} \equiv \frac{\partial^{2}}{\partial
\tau
\partial x} \tilde Q ,~\tilde Q_{\tau \tau x}  \equiv
\frac{\partial^{3}}{\partial \tau^{2} \partial x} \tilde
Q,~~ \mbox{ etc.}
$$ 
Without taking a limit yet, but expanding
asymptotically the expressions in powers of $z$, we find
%
%
%
{\footnotesize
 \bean
  \bigg(F^+\BB_{-1}F^- -F^-\BB_{-1}F^+\bigg)
   &=&  \frac{ \tilde Q_{\tau x x} }{\rho z^7}
- \frac{\tau \tilde  Q_{\tau x x} }{\rho z^5}
  +\frac{\tilde
Q_{xxx}(r\rho-\tilde Q_{\tau x})+\tilde Q_{\tau
xx}(\tilde
Q_{xx}+\frac{\tau^2}{2})}{\rho z^3}  + O(\frac{1}{z}) \\
  \bigg(F^+\BB_{-1}^2F^- -F^-\BB_{-1}^2F^+\bigg)
     &=&\frac{ \tilde Q_{\tau x x x} }{\rho z^8}
- \frac{\tau  \tilde Q_{\tau x x x} }{\rho z^6}
 +
   \frac{\tilde Q_{xxxx}(r\rho-\tilde Q_{\tau x})+\tilde
Q_{\tau xxx}(\tilde Q_{xx}+\frac{\tau^2}{2})}{\rho z^4}
 + O(\frac{1}{z^2}). \eean}
In order to compute the expansions of $G^{\pm}$ (which
are respectively defined in (\ref{C012}) and
(\ref{C015}))  and of $\BB_{-1} G^{\pm}$, we need the
asymptotics of $H_i^{\pm}$, as defined in (\ref{Hplus})
and (\ref{Hminus}). In the end, one finds
%
%
{\footnotesize \bean
 F^+ G^- + F^- G^+
 &=&  \frac{2 \tilde Q_{\tau x
x} }{\rho^2} \bigg[  2 \tilde Q_{\tau x}  +
r\left(\rho-1\right)^2 \bigg] \frac{1}{z^{10}}
 \\ &&-  \frac{2}{\rho^3 z^{8}}
 \left(   \begin{array}{l}
   \tilde Q_{\tau x}  \bigg( \left(\rho-1\right)
   \tilde  Q_{\tau \tau x}
    + 4 \rho \tau
 \tilde  Q_{\tau x x}  \bigg)
    \\ \\
  -  2 r \rho  (\rho-1 ) ( \tilde Q_{\tau \tau x}
     +
       \tau    \tilde Q_{\tau x x})
     \end{array}
     \!\! \right)
    \! +\! \frac{\mathcal{T}}{z^{6}}\! +O(\frac{1}{z^4})
     \\
     F ^+ ~ \BB_{-1}G^-+F^-\BB_{-1}G^+&=&  \frac{2  \tilde Q_{\tau x x x}
}{\rho^2} \bigg[ 2  \tilde Q_{\tau x}  +
r\left(1-\rho\right)^2 \bigg] \frac{1}{z^{11}}
 \\ &&- \frac{2}{\rho^3 z^{9}}
 \left( \!\!  \begin{array}{l}
   \tilde Q_{\tau x}  \bigg( \left(\rho\!-\!1\right) \tilde  Q_{\tau \tau x x}
   + 4 \rho \tau
  \tilde Q_{\tau x x x}  \bigg)
    \\ \\
    -  2 r \rho  (\rho-1 )  (\tilde  Q_{\tau \tau x x}
     + \tau  \tilde Q_{\tau x x x})
     \end{array}
     \! \!\!\!\right)\!+\! \frac{\mathcal{T}'}{z^7}
   \!\! + \!\! O(\frac{1}{z^{5}})
  ,\eean}
where $\mathcal{T}$ and $\mathcal{T}^\prime$ are given by the
following expressions and where $\mathcal{T}_1$ and $\mathcal{T}_1'$ denote further expressions in the derivatives of
$\tilde Q$,
 {\footnotesize \bean  \mathcal{T}
 &=&
 2r^2\tilde Q_{xxx}
 +2r(\tilde Q_{\tau xx}(\tilde Q_{xx}+\tau ^2-x)+2\tau \tilde
 Q_{\tau \tau x}+2\tilde Q_{\tau x}+\tilde Q_{\tau \tau \tau })
 -2\tilde Q_{\tau x}^2(\tilde Q_{xxx}+2)\\
&&
  -2\tilde Q_{\tau x}(\tau \tilde
  Q_{\tau \tau x}+\tilde Q_{\tau \tau \tau }-\tilde Q_{\tau xx}
  (\tilde Q_{xx}+4\tau ^2+x))
 +2\tilde Q_{\tau xx}\tilde Q_{\tau }+ \tilde Q_{\tau \tau }
  \tilde Q_{\tau \tau x}
   +(\rho-1)\mathcal{T}_1,
 \eean}
 {
 \bean \mathcal{T}'
  \!&=&\!2r^2 \tilde Q_{xxxx}
 +2r(\tilde Q_{\tau xxx}(\tilde Q_{xx}+\tau ^2-x)+2\tau \tilde
 Q_{\tau \tau xx}+\tilde Q_{\tau xx}+\tilde Q_{\tau \tau \tau x})
 \\
 &&-2\tilde Q_{\tau x}^2\tilde Q_{xxxx}-2\tilde Q_{\tau x}
 (\tau \tilde Q_{\tau \tau xx}+\tilde
 Q_{\tau \tau \tau x}+2\tilde Q_{\tau xx}
 -\tilde Q_{\tau xxx}
  (\tilde Q_{xx}+4\tau ^2+x))
   \\
  &&
+  \tilde Q_{\tau \tau x}^2+ \tilde Q_{\tau \tau }\tilde
Q_{\tau \tau xx}
  -\tilde Q_{\tau xx}\tilde Q_{\tau \tau \tau }
  +2\tilde Q_{\tau xxx}\tilde Q_{\tau }+(\rho-1)\mathcal{T}_1'.
\eean}}
%
%
 %
Consequently, using the two leading orders for the
expressions above, one obtains for small $z$:
\be \label{expansion6}\ee
\vspace*{-1cm}
{\footnotesize\bean
 0 &=&  \!\!\!\!\!\!\!\left.\left\{\!\!\begin{array}{l}
\Bigl(F^+\BB_{-1}G^-+F^-\BB_{-1}G^+ \Bigr)
\Bigl(F^+\BB_{-1}F^- -F^-\BB_{-1}F^+ \Bigr)  \\
 \,-\,
 \Bigl(F^+ G^- +F^- G^+
\Bigr) \Bigl(F^+\BB_{-1}^2F^- -F^-\BB_{-1}^2F^+ \Bigr)
\end{array}\!\!\!\right\}
 \right|_{\begin{array}{l}
     \alpha \,\mapsto\, \left( \rho / z^{ 3} \right)  e^{ \tau \,z^{2} } \\
     b  \mapsto  x z + \frac{2}{z^{ 3}} \\
     n  \mapsto  \frac{1}{z^{ 6}}\end{array} }
 \\  
 &=& \frac{4(\rho-1)}{\rho^4 z^{16}}
     \bigg( 2 r\rho  -  \frac{\partial^{2}\tilde Q}{\partial \tau \partial x}
    \bigg)  \left\{\frac{\partial^{3}  \tilde Q}{\partial \tau^2 \,\partial x} \,,\,
   \frac{\partial^{3}\tilde  Q }
   {\partial \tau  \partial x^2} \right\}
   _{x}  +
    \frac{\mathcal{E}(\tilde Q)
    +(\rho-1)\mathcal{F}(\tilde Q)}{z^{14}}
    + O\left(\frac{1}{z^{12}}\right) ,
   \no
    \eean }
where $\mathcal{E}(\tilde Q)$ is given by
  \bea     \mathcal{E}(\tilde Q) &=&  2 \bigg( r - \frac{\partial^{2} \tilde
Q}{\partial \tau
\partial x} \bigg)^{2} \bigg\{\frac{\partial^{3}
\tilde Q}{\partial \tau
 \partial x^2} ,  \frac{\partial^{3}   \tilde Q}{\partial x^3}
 \bigg\}_{x} \no\\
 & &+ 2 \bigg( r - \frac{\partial^{2}  \tilde Q}
 {\partial \tau  \partial x}
\bigg)  \bigg\{ \frac{\partial^{3} \tilde  Q}{\partial
\tau
 \partial x^2}  ,  \frac{\partial}{\partial \tau}
\Bigg( \frac{\partial}{\partial \tau}  \bigg(
\frac{\partial \tilde  Q}{\partial \tau} + \tau
\frac{\partial
 \tilde Q}{\partial x} \bigg)  -  x \frac{\partial^2
 \tilde Q}{\partial x^2} \Bigg)
 \bigg\}_{x} \no\\
 &&+ \bigg\{ \frac{\partial^{3} \tilde  Q}
 {\partial \tau  \partial x^2}
 ,
\frac{\partial^{3}  \tilde Q}{\partial \tau^2 \partial
x} \bigg( 2r \tau + \frac{\partial^2   \tilde
Q}{\partial \tau^2} \bigg) \bigg\}_{x} +
\bigg(\frac{\partial^3 \tilde Q}{\partial \tau
\partial x^2}\bigg)^{2}  \frac{\partial}{\partial \tau}
\bigg( \frac{\partial^2   \tilde Q}{\partial \tau^2} + 2
 \frac{\partial   \tilde Q}{\partial x} \bigg)  ,
 \label{PDE6}   \eea
%
%
where $\mathcal{F}(\tilde Q)$ is a similar expression,
that will not be needed, and where the bracket is a
Wronskian associated to the ``space''
operator ${\partial}/{\partial x}$. 

\vspace{.4cm}

Hence, \underline{\em for $0\leq \rho < 1$}, taking the
limit $z\rightarrow 0$ above and using (\ref{6.3}),
yields the equation
 $$\bigg( 2 r\rho  -  \frac{\partial^{2}
Q}{\partial \tau \partial x}
    \bigg)  \left\{\frac{\partial^{3}    Q}{\partial \tau^2
    \partial x} ~,~
   \frac{\partial^{3}  Q }
   {\partial \tau  \partial x^2} \right\}
   _{x}  =0
   $$
which is trivially satisfied; indeed, from other
considerations we know that $e^Q$ is the Tracy-Widom
distribution, which of course is $\tau$-independent.


However, \underline{\em in the critical case $\rho=1$}, the
leading term has order $1/ z^{14}$, with coefficient
$\mathcal{E}(\tilde Q(\tau,x;z))$, as in (\ref{PDE6}). Then
taking a limit in (\ref{expansion6}) when $z\rg 0$ and
using (\ref{6.3}), forces upon us the equation
$$\lim_{z\rg 0}\mathcal{E}(\tilde Q(\tau,x;z))
=\mathcal{E}(Q(\tau,x))=0,$$
 which an easy computation shows can be written
 as
 \begin{equation} \label{one-time PDE}
\begin{split}
&\left\{\frac{\pl^3 Q}{\pl \tau\pl x^2},
  \left[\begin{array}{l}
   r^{2}
  \frac{\pl^3 Q}{\pl x^3}
 +r\left(\frac{\pl^2}{\pl \tau^2}(\frac{\pl Q}{\pl \tau}
  +2\tau \frac{\pl Q}{\pl x}
 )
   -\frac{\pl^3 (x Q)}{\pl \tau\pl x^2}
   +2\left\{
   \frac{\pl^2Q}{\pl \tau\pl x},\frac{\pl^2Q}{\pl x^2}
   \right\}_x
 \right) \\
  +\frac{1}{2}
 \frac{\pl^3Q}{\pl \tau^2\pl x}\frac{\pl^2 Q}{\pl \tau^2}
  -\frac{\pl^2Q}{\pl \tau\pl x}\frac{\pl^3Q}{\pl \tau^3}
  +\left(\frac{\pl^2Q}{\pl \tau\pl x}\right)^2\frac{\pl^3Q}
  {\pl x^3}+
  \left\{ \frac{\pl^2Q}{\pl \tau\pl x},
   \frac{\pl (\tau Q)}{\pl \tau^2}
  \right\}_x
  \end{array} \right]
  \right\}_x
  \\
& -\frac{1}{2}\left( {\frac {\pl
^{3}Q}{\pl \tau\pl x^2}} \right) ^{2}\left({\frac
{\partial ^{3}Q}{\partial {\tau}^{ 3}}}
 -4\frac{\pl^2Q}{\pl \tau\pl x}  \frac {\pl ^{3}Q}{\pl
x^3}\right)=0,
\end{split}
\end{equation}   
 and further rewritten as
   equation (\ref{PDE0}), ending the proof of Theorem~0.4.
\end{proof}
%

\begin{remark} This ``phase transition'' at $\rho=1$ is
completely analogous to the results found in \cite{BBP}
and \cite{Peche} for small rank perturbations of random
Hermitian matrices.
\end{remark}

\section{Remote past asymptotics}\label{sect7}

The aim of this section is to study the behavior of the
$r$-Airy process $\mathcal {A}^{(r)}(t)$ for $t\rg -\iy$, as
stated in Theorem \ref{Theo3}. In this section $\tau$
will be systematically replaced by $t$. This theorem
will be rephrased as Theorem \ref{theo:one-time
expansion}, which includes some additional details.




\begin{theorem}\label{theo:one-time expansion}

%
%
The log of the probability of the $r$-Airy process $
 Q(t,x)= \\  \log\BP  (\sup\mathcal {A}^{(r)}(t)\leq x)
$ admits an asymptotic expansion, as $t\rightarrow
-\iy$, having the following form
$$Q(t,x)=Q_0(x)+\frac{1}{t}Q_1(x)+\frac{1}{t^2}Q_2(x)+...
,$$ for the initial condition
%
%
$$
\lim _{t\rightarrow -\iy}Q(t,x):=Q_0(x):=
 \log\BP(\sup\mathcal {A}(t)\leq x)=
-\int_x^{\iy}(\al-x)g^2(\al)d\al ,$$
 and where
{\be %
 \label{Qi}\ee}
\vspace*{-1cm}{\footnotesize\bean Q_1&=&rQ'_0
,~\quad Q_2=\frac{r^2}{2!}Q''_0
 ,~\quad Q_3=\frac{r^3}{3!}Q'''_0+\frac{r}{3 } xQ'_0,
 ~\quad Q_4=
  \frac{r^4}{4!}
   Q_0^{iv}+\frac{r^2}{3}xQ_0''+\frac{7r^2}{12} Q_0',~~~
 \\ %
%
   Q_5&=&
   \frac{r^5}{5!}
   Q_0^{v}+\frac{r^3}{3~2!}xQ_0'''+\frac{7r^3}{12} Q_0''
   +\frac{r}{5} \mathcal{F}_5
~~~~~~~~~~~~~~~~~~~~~~~~~~~~~~~~~~~~~~~~~~~~~~~~
\\
Q_6&=&\frac{r^6}{6!}
   Q_0^{vi}+\frac{r^4}{3~3!}xQ_0^{iv}+
   \frac{7r^4}{12~2!} Q_0'''+{\frac {r^2}{5}}  \left(
 \mathcal{F}_5'+\frac{5}{18}(x^2Q_0''+13(x+c_6)Q_0' ) \right)
  \\
 \vdots\\
Q_n&=&
   \frac{r^n}{n!}
   Q_0^{(n)}+\frac{r^{n-2}x}{3~(n-3)!} {Q_0^{(n-2)}}
    +\frac{7r^{n-2}}{12(n-4)!}  {Q_0^{(n-3)}}
   +\sum_{i=2}^{\left[\frac{n-1}{2}\right]}{r^{n-2i}}
   Q_{n,n-2i}(x)
   %
   \eean}
   for some constant $c_6$ and with
    $$
\mathcal{F}_5:=  x^2Q_0'+4xQ_0+Q_0^{\prime 2}
   +10\int_x^{\iy}Q_0
   -
   6\int_x^{\iy}dy\int_y^{\iy}du Q_0^{\prime\prime 2}
   .$$
Also
   \bean
   \BP(\sup\mathcal {A}^{(r)}(t)\leq x)&=&
    \BP\left(\sup\mathcal {A}(t)\leq (x+\frac{r}{t})(1+\frac{r}{3t^3})+
    \frac{r^2}{4t^4}\right)
    \\
    &&~~~~~~~~~~~~~~~~~~~~\times ~\left(1+\frac{r\mathcal{F}_5}{5t^5}
     +O(\frac{1}{t^6})\right)
     .\eean
     The mean and variance of the right edge of the process
      behave as
      \bean
      {\mathbb E}
       (\sup \mathcal {A}^{(r)}(t))&=&
 {\mathbb E}(\sup\mathcal {A}^{(0)}(t))\left(1-\frac{r}{3t^3}\right)
 -\frac{r}{t}-\frac{r^2}{4t^4} +O(\frac{1}{t^5})
      \\
       \mbox{\em var}(\sup \mathcal {A}^{(r)}(t))
       &=&
 \mbox{\em var}(\sup\mathcal {A}^{(0)}(t))\left(1-\frac{2r}{3t^3}\right)
  +O(\frac{1}{t^5}).\eean

     \end{theorem}

\bigbreak

\noindent Remember the $r$-Airy kernel
   $$
 K^{(r)}_t(u,v)=\int_0^{\iy}
 dw A_r^-(u+w; t)A_r^+(v+w; t)
 ,$$ as in (\ref{r-kernel}),
 where $A^{\pm}_r(u;\tau)$ is given by (\ref{k-Airy0}),
%
%
 where $C$ is a contour running from $\iy
 e^{5i\pi/6}$ to $\iy
 e^{i\pi/6}$, such that $-it$ lies above the contour.
 In this section one lets $t\rg -\iy$, which, of course,
implies that $-it$ will remain above the contour $C$ and thus
this limit is compatible with the contour just
mentioned. Letting $t\rg +\iy$ would require a drastic
change of the functions $A^{\pm}_r$.
 In this section the subscript $t$ will occasionally be omitted from
 the $r$-Airy kernel $K^{(r)}_t(u,v)$.
Note that $K^{(0)}_t(u,v)=K^{(0)}(u,v)$ is the Airy
kernel, which is independent of $t$,
 $$K^{(0)}(u,v):=\frac{A(u)A'(v)-A'(u)A(v)}{u-v}
  =
 \int^{\iy}_{0}A(w+u)A(w+v)dw
 ,$$
  where $$A(u)=A_0^{\pm}(u)=\int_C   e^{\frac{1}{3} ia^3+iau }
 \frac{da}{2\pi}$$ is the Airy function, satisfying the
 ordinary
differential equation $ A'' (u)=uA(u)$ and behaving
asymptotically as
 \be
A(x)=\frac{e^{-\frac{2}{3}x^{3/2}}}{2\sqrt{\pi}x^{1/4}}
(1+\sum_1^{\iy}\frac{\al_i}{(x^{3/2})^i}+\ldots),~~~\mbox{as
$x\rg \iy$.}
\label{Airy-asymptotics}
\ee
 The ODE and $\lim_{x\rg \iy}A(x)=0$ imply the following formulae,
upon differentiation by $x$,
 \bea 
  \int_x^{\iy}A^2(u)du&=&(A^{'2}-A A'')(x) \no\\
  \int_x^{\iy}A^{'2}(u)du&=&
  -\frac{1}{3}\left( (A^2)'+x(A^{'2}-AA'')\right)(x)
. \label{RemF}\eea

\noindent Also remember the Tracy-Widom distribution
\cite{TW-Airy}
   \bean
      \det
  \left(I-K^{(0)}\right )_x=\exp\left(-\int^{\iy}_{x}(\al-x)g^2(\al)
d\al \right)
, \eean
 where
  $g(\al)$ is the Hastings-MacLeod solution (\ref{Painleve})
   of Painlev\'e II.
%
%
The following shorthand notation will be used for
integers $\ell\geq 1$,
$$
O(A^\ell)=O\left(  \left(\frac{
  e^{-\frac{2}{3}  x^{\frac{3}{2}}}}{2\sqrt \pi x^{1/4}}
  \right)^\ell x^k\right),~~\mbox{for $x\rg \iy$},
   $$
whatever be the power $k\in {\mathbb R}$.

\begin{lemma}\label{lemma:Qprime}
Given
$$
Q_0(x)=\log \det \left(I-K^{(0)}\right
)_x=-\int_x^{\iy}(\al-x)g^2(\al)d\al
,$$
one checks \bean
 Q_0'(x)&=&\int_x^{\iy}
g^2(u)du=A^{'2}(x)-A''(x)A(x)
 +O(A^3)%
  \\
   Q_0^{(n)}(x)&=&-(g^{2 })^{(n-2)}(x)
   = -(A^{2})^{(n-2)}(x)+O(A^3),~~~~\mbox{for $n\geq 2$}.
  \eean
 \end{lemma}
\begin{proof} The estimates follow from (\ref{RemF}) and an
improved version of the estimate (\ref{Painleve}) by
Hastings-McLeod \cite{HM}, namely:
$$
g(x)= \mbox{A}(x)+ { 
O}\left(\frac{e^{-\frac{4}{3}x^{3/2}}}{x^{1/4}}\right)
\mbox{~for~}x\nearrow\iy,
$$
which is to be interpreted as a genuine asymptotic formula; i.e., it can be both integrated and differentiated.
 \end{proof}

  \bigbreak
\noindent For future use, one needs the following
estimates for the Airy function:

\begin{lemma}\label{lemma:Airy}For $x\rg \iy$, one has
the estimates
\bean
  \int_x^{\iy} A(u)du&=&
 \frac{e^{-\frac{2}{3}x^{3/2}}}{2\sqrt{\pi} x^{3/4}}
  (1+\sum_1^{\iy}\frac{c'_i}{(x^{3/2})^i} )=O(A)
\\
\int_x^{\iy}\!\! A^2(u)du&=&(A^{'2}-A''A)(x)=
 \frac{e^{-\frac{4}{3}x^{3/2}}}{8\pi x}
  (1+\sum_1^{\iy}\frac{c''_i}{(x^{3/2})^i} )=O(A^2)
 .\eean


\end{lemma}

\begin{proof} Upon using integration by parts and upon substituting the asymptotic formula (\ref{Airy-asymptotics}) for the Airy function, one computes for
instance,
 \bean
  \int_x^{\iy} A^2(u)du&=&
 \int_x^{\iy}\frac{-1}{8\pi u}
 (1+\sum_1^{\iy}\frac{\tilde c'_i}{(u^{3/2})^i})
d(e^{-\frac{4}{3}u^{3/2}})\\
&=&
 \frac{e^{-\frac{4}{3}x^{3/2}}}{8\pi x}
  (1+\sum_1^{\iy}\frac{\tilde c'_i}{(x^{3/2})^i} )-
  \!\frac{1}{8\pi}\!
   \int_x^{\iy}\frac{e^{-\frac{4}{3}u^{3/2}}}{ u^2}
  (1+\sum_1^{\iy}\frac{\tilde c''_i}{(u^{3/2})^i} ).
\eean Further terms in the expansion can be obtained by
differentiation by parts 
%
and similarly for the first expression, thus ending the proof
of Lemma~\ref{lemma:Airy}.\end{proof}

 Given a kernel $F(y,u)$ acting on
$L^2(E)$ with $E\subset {\mathbb R}$ and a bounded continuous
real function $f$ on $E$, define the norms
\bean \vv F\vv_1=\sup_{y\in E}\int_E
|F(y,u)|du~~~~\mbox{and}~~ ~~\vv f\vv_{\iy}=\sup_{y\in
E}|f(y)| .\eean
 If $f$ is a function of several variables, the sup is
 taken over all variables.
 Then
  \bean
\vv FG\vv_1&\leq&\sup_{y\in E}\int_E dz\int_E du |F(y,u)G(u,z)|\\
\\
  &\leq&\sup_{y\in E}\int_E du |F(y,u)|\sup_{u\in E}\int_E dz|G(u,z)|=\vv
F\vv_1\vv G\vv_1 \eean
Hence\footnote{Obviously the estimate below requires $\vv
F\vv_1<1$. In the application, this is achieved by restricting the domain of the operator to the interval $(x,\iy)$ for sufficiently large $x$.}
\bean \vv \sum_1^{\iy}F^n\vv_1&\leq&\sum_1^{\iy}\vv
F\vv_1^n=\frac{\vv F\vv_1}{1-\vv F\vv_1} \eean
  and
   \bean
\vv Ff\vv_{\iy}=\sup_{y\in E}\int_E |F(y,u)  f(u)|du
&\leq&\sup_{u\in E} |f(u)|\sup_{y\in E}\int_E |F(y,u)|du
=\vv F\vv_1\vv f\vv_{\iy}. \eean
Given the kernel $K^{(0)}$, define the resolvent kernel
$R$ by \be I+R:=(I-K^{(0)})^{-1}\label{resolvent}.\ee
Then readily
\be
 R-K^{(0)}= K^{(0)2}(I+K^{(0)}+K^{(0)2}+\ldots)
  = (I-K^{(0)})^{-1}K^{(0)2}.
 \label{R-K}\ee

 \bigbreak

\begin{lemma}\label{functA}One has the following
estimates \bean
 \vv K^{(0)}\vv_{1}=O(A^2),~~
 \vv K^{(0)}\vv_{\iy}=O(A^2),\qquad\mbox{and}\qquad
 \vv R-K^{(0)}\vv_{\iy}=O(A^4).
\eean
\end{lemma}
\begin{proof} Using the fact that $A(u)\geq 0$ is monotonically
decreasing for $u\geq 0$, and setting $E=(x,\iy)$,
\bean \vv K^{(0)}\vv_1&=&\sup_{u\in
(x,\iy)}\int_x^{\iy}dv\left|\int_0^{\iy} dw
~A(u+w)A(v+w)\right|\\
&\leq&\sup_{u\in (x,\iy)}\int_x^{\iy}dv~A(v)\int_0^{\iy} dw
~A(u+w)  \\
 \\
& \leq & \left(\int_x^{\iy}dv~A(v)\right)^2={ O}(A^2)
\eean by Lemma \ref{lemma:Airy}, while
\bean
\vv K^{(0)}\vv_{\iy}=\sup_{u,v\in (x,\iy)}|K(u,v)|
&=&\sup_{u,v\in (x,\iy)}\left|\int_0^{\iy}  A(u+w) A(v+w)dw\right| \\
\\
&\leq& \int_0^{\iy} A(x+w)^2dw\\
\\
&=&\int_x^{\iy}A(w)^2dw={ O}(A^2), \eean also by
Lemma \ref{lemma:Airy}.
%
%
Considering the function $R(\cdot,v)$ for fixed $v$, one
has, using (\ref{R-K}),
\bean \vv R(\cdot,v)-K^{(0)}(\cdot,v)\vv_{\iy}&=&\vv
(I-K^{(0)})^{-1}(K^{(0)2}(\cdot,v))\vv_{\iy} \\
\\
&\leq&\vv (I-K^{(0)})^{-1}\vv_1~\vv K^{(0)2}(\cdot,v)\vv_{\iy} \\
\\
&\leq&\frac{1}{1-\vv K^{(0)}\vv_1}\vv K^{(0)2}(\cdot,v)\vv_{\iy}\\
\\
&\leq&\frac{1}{1-\vv K^{(0)}\vv_1}\vv K^{(0)}\vv_1~\vv
K^{(0)}(\cdot,v)\vv_{\iy}. \eean
Hence, with $\vv K\vv_{\iy}:=\displaystyle{\sup_{u,v \in
E} |K(u,v)|}$,
\vspace{-.4cm}

$$
\vv R-K^{(0)}\vv_{\iy}:=\sup_{u,v\in
E}|R(u,v)-K^{(0)}(u,v)|\leq\frac{\vv K^{(0)}\vv_1}{
1-\vv K^{(0)}\vv_1}\vv K^{(0)}\vv_{\iy}={ O}(A^4) ,$$
ending the proof of Lemma \ref{functA}.
\end{proof}

The next point is to get an asymptotic expansion for the
Fredholm determinant $\det (I-K_t^{(r)})_{x,\iy}$ in
$t$, using the asymptotic expansion  of the kernel
$K_t^{(r)}(u,v)$ in $t$ (Lemma \ref{lemma:kernels}); this
leads to the next Lemma:
%

\begin{lemma}\label{lemma:F}
The following probability for the $r$-Airy process has
an asymptotic expansion in $1/t$ of the following form
 \bean
Q(t,x)&=& \log\BP(\sup\mathcal {A}^{(r)}(t)\leq x)=
\log\det\left(I-K_t^{(r)}\right)_x
= Q_0(x)+\sum^{\iy}_{n=1}\frac{Q_n(x)}{t^n} ,\eean
 where
 \be
 Q_n(x)=
 \sum_{i=0}^{\left[\frac{n-1}{2}\right]}
  r^{n-2i}Q_{n,n-2i}(x)
  =
   -\Tr K_n^{(r)}+O(A^4)
  ,
  \label{Qn}
  \ee
 where $Q_0=\log\BP(\sup\mathcal {A}(t)\leq x)$ and where the $Q_i\rightarrow 0$ and have all their
derivatives $\rightarrow 0$ for $x \rightarrow \iy$.
Moreover, for $x \rightarrow \iy$, one has
{  \bean
Q_1(x)&=& r(A^{'2}-AA''  +O(A^4))\\
Q_2(x)&=& -\frac{r^2}{2}(A^2+O(A^4))\\
Q_3(x)&=& -\frac{r^3}{3!}((A^2)'  +O(A^4))
-\frac{r}{3}
\bigl((A^2)'+3\int_x^{\iy}A'^2(u)du+O(A^4)\bigr)+O(A^4)\\
 \vdots
 \\
Q_n(x)&=& -\frac{r^n}{n!}((A^2)^{(n-2)} +O(A^4))
  \\
  &&
-\frac{r^{n-2}} {(n-3)!} \left( \frac{3n-1}{24}(A^{2})''
-\frac{n-1}{2}(A')^2 +O(A^4)\right)
 ^{(n-4)}
 \\
%
 && 
 +\sum_{i=2}^{\left[\frac{n-1}{2}\right]}
  r^{n-2i}T_{n,n-2i}(A)+O(A^4)
%
  ,\eean}
 with $T_{n,n-2i}(A)=$ quadratic polynomial of $A$
and $ ~A'$, with coefficients depending on $x$ $+$ $
\int_x^\iy $(quadratic polynomial of $A$ and $ ~A'$,
with coefficients depending on $x$).
\end{lemma}

\begin{proof} We shall always operate in $L^2(x,\iy)$, so that
occasionally the $x$ will be suppressed. Then, using the
asymptotics for the kernel $K_t^{(r)} $ as in Lemma
\ref{lemma:kernels}, one has the following:
 $$ I-K_t^{(r)} =
I-K^{(0)}-\frac{K^{(r)}_1}{t}-\frac{K^{(r)}_2}{t^2}-\ldots
=(I-K^{(0)})(I-\sum_{i\geq 1} \frac{L_i}{t^i})
  $$
  with (the resolvent operator $R$ of the Airy kernel is
defined in (\ref{resolvent}))
  \be
   L_i=(I-K^{(0)})^{-1}K_i^{(r)}=(I+R)K_i^{(r)}.
   \label{L}\ee
%
 Using $\log(1-z)=-z-\frac{z^2}{2}-\frac{z^3}{3}-\ldots$,
one finds
$$
Q(t,x)=\log\det\left(I-K_t^{(r)}\right)_x=\Tr\log(I-K_t^{(r)})=
\Tr\log(I-K^{(0)})+\sum^{\iy}_{i=1}\frac{Q_i}{t^i} ,$$
where
 \bea  \label{Q's}
  Q_1&=&-\Tr L_1,\quad Q_2=-\Tr
(L_2+\frac{L_1^2}{2}), \quad Q_3=-\Tr \Bigl(L_3+{L_1L_2}
+\frac{L_1^3}{3}\Bigr),
 \\~~Q_4&=&-\Tr\Bigl(L_4+\frac{1}{2}L_2^2+L_1^2L_2+L_1L_3
 +\frac{1}{4}L_1^4
\Bigr),\ldots.\no\eea
  More generally, the $Q_n$'s are weight-homogeneous polynomials of degree $n$
   in the $L_i$, with weight$(L_i)=i$, having the form
   below, which can further be expressed in terms of the
   $K_i^{(r)}$ and $R$, using
   expression (\ref{L}) for the $L_i$,
  \bea  \label{Qn7}
  Q_n&=&-\Tr
L_n+\Tr P_n(L_1,L_2,\ldots,L_{n-1})
 \\
&=&
 -\Tr K^{(r)}_n-\Tr(RK^{(r)}_n)+\Tr P_n(L_1,L_2,\ldots,L_{n-1})\no\\
 &=&
  -\Tr K^{(r)}_n+\Tr S_n(K^{(r)}_1,\ldots,K^{(r)}_{n-1},R)
  ,\no\eea
%
%
    with $P_k$ and $S_k$ polynomials of non-commutative variables
with no linear or independent terms, but with quadratic
terms and higher.
From Lemma \ref{lemma:kernels}, the kernels $K_i^{(r)}$
and hence the $L_i$ are polynomials in $r$ of degree $i$
having no constant terms; hence the $Q_n$'s, by their
weight-homogeneity, are polynomials of degree $n$ in
$r$, having no constant term and so (here one must
indicate the $r$-dependence)
 \be
 Q_r(t,x):=Q(t,x)=\sum_{n=0}^{\iy}
 \left(\frac{r}{t}\right)^n\sum_{i=0}^{n-1}
 \frac{Q_{n,n-i}(x)}{r^i}.\label{star}\ee
 The claim is that only the terms $Q_{n,n-2i}(x)$
 appear. Observe from
 (\ref{r-kernel}) and (\ref{k-Airy0}),
 that $
  K^{(-r)}_{-t}(u,v)=K^{(r)}_{t}(v,u)
   ,$
    and thus
    \bean
     Q_r(t,x)=\det (I-K_t^{(r)}(u,v))_x
           =\det (I-K_t^{(r)}(v,u))_x
           &=&\det (I-K_{-t}^{(-r)}(v,u))_x\no\\
           &=&Q_{-r}(-t,x)
           .\eean
So from (\ref{star}), one has
  $$
  Q_r(t,x)=\frac{1}{2}(Q_r(t,x)+Q_{-r}(-t,x)),$$
  implying that only the even terms appear in the
  sum $\sum_i$ in (\ref{star}), thus proving the
  statement (\ref{Qn}).

 Next we now proceed to estimate the
two traces above:

\noindent$\bullet$ At first
 $$
   \Tr S_k(K^{(r)}_1,\ldots,K^{(r)}_{k-1},R)=O(A^4)
 ,$$
 which we now illustrate on a typical example,
 like $\Tr( RK_1^{(r)})$
. Note $A(u)\geq 0$ for $u\geq 0$, and so by Lemma
\ref{functA},  \bean
 | \Tr RK^{(r)}_1|&=&|\Tr ((R-K^{(0)})K^{(r)}_1)|+|\Tr (K^{(0)}K^{(r)}_1)
|  \\
  &\leq& (|\!|R-K^{(0)}|\!|_{\iy}+|\!|K^{(0)}|\!|_{\iy})
   \int\!\!\!\!\int_x^{\iy}|K_1^{(r)}(u,v)|du dv\\
   &\leq& (O(A^4)+O(A^2))
    \int\!\!\!\!\int_x^{\iy}A(u)A(v)du dv\\
    &\leq&O(A^2)\Bigl(\int_x^{\iy} A(u)du\Bigr)^2 \\
    &\leq& O(A^4);
  \eean
  the last estimate follows from Lemma \ref{lemma:Airy}. More generally, the trace
  of a monomial of degree $\ell$ has order
  $O(A^{2\ell})$.

\noindent $\bullet$ Then we evaluate $\Tr K^{(r)}_n$; in
order to do so, it suffices to evaluate the kernels
$K_n^{(r)}$ of Lemma \ref{lemma:kernels} along the
diagonal, and to notice that a
skew-symmetric operator vanishes on the diagonal. Since the domain of the operator is unbounded, one needs to consider $K_n^{(r)} \chi_{_{(x,m)}}$; the trace is then obtained by integrating  on the diagonal and by taking the limit $m\rg \iy$, upon using the decay of the kernel at $\iy$. %
Therefore, on the diagonal, $K_n^{(r)}(u,v)$ is a
polynomial of degree $n$ in $r$, skipping every other
term,
\bean {K_n^{(r)}(u,u)}
&=&-\frac{r^n}{n!}\Big(A^2(u)\Big)^{(n-1)}\\
\\
&
&-\frac{r^{n-2}}{(n-3)!}\left(\frac{3n-1}{24}\Big(A^2(u)\Big)^{(n-1)}-\frac{n-1}{2}
\Big(A^{' 2}(u)\Big)^{(n-3)}\right)+ ...
 \eean
  That the
$Q_i\rightarrow 0$ and that $Q_0=\log\BP(\sup \mathcal {A}(t)\leq
x)$ has all their derivatives $\rightarrow 0$ for $x
\rightarrow \iy$ follows from the statement on $T_{n,n-2i}(A)$ (which are the coefficients of $r^{n-2i}$ appearing in $Q_n(x)$ as in the statement of Lemma \ref{lemma:F} )
and the asymptotics of the Airy function, from which
Lemma \ref{lemma:F}  follows. \end{proof}

%
%

%

%
%


 \begin{proof}[Proof of Theorem~\ref{theo:one-time expansion}]
 From Section~\ref{sect6}, we know that $Q(t,x)=\log   \BP(\sup\mathcal {A}^{(r)}(t)\leq x)$ satisfies the non-linear
  PDE (\ref{PDE0}); it is more convenient here to use version
  (\ref{one-time PDE}) of the equation.
  Also remember $Q(t,x)\rg 0$, when $x\rg \iy$. Then,
assigning weight$=1$ to both variables $t$ and $r$, one
readily checks that the PDE (\ref{one-time PDE}) can be
graded as follows:
  \bean
 \hspace{-.7cm} 0\!\!\!
   &=&   \left\{\frac{\pl^3Q}{\pl t\pl x^2}
  ,r^{2}
   \frac{\pl^3Q}{\pl
   x^3}\right\}~~~~~~~~~~~~~~~~~
   ~~~~~~~~~~~~~~~~~~~~~~~~~~~~~~~~~(\mbox{weight$=1$})
  \\
  &&+\left\{\frac{\pl^3Q}{\pl t\pl x^2},r\left(2\frac{\pl^2}{\pl t^2}(
     t\frac{\pl Q}{\pl x}
  )
    -\frac{\pl^3 (x Q)}{\pl t\pl x^2}
    +2\left\{
    \frac{\pl^2Q}{\pl t\pl x},\frac{\pl^2Q}{\pl x^2} \right\}
  \right)\right\}
  ~(\mbox{weight$=-1$})\\
  &&+  \mbox{other terms of weight$<-1$}.
  \eean
Since, by Lemma \ref{lemma:F}, the solution has the
following general form
  $$
  Q(t,x)=\sum_0^\iy \frac{Q_n(x)}{t^n}=
   \sum_0^\iy \left(\frac{r}{t}\right)^n\left( Q_{nn}(x)+
   \frac{1}{r^2}
Q_{n,n-2}(x)+ \frac{1}{r^4} Q_{n,n-4}(x)+\ldots\right)
,$$
  it follows that one can compute inductively all
  the $Q_{nn}(x)$ and then inductively all the $Q_{n,n-2}(x)$ and so
    the $Q_n$ will be as announced in Theorem
  \ref{theo:one-time expansion}.

Setting this solution in the PDE above, yields a series
of descending weights, namely $ 0=W_1+W_{-1}+W_{-3}+
..., $ which holds for $t\rg -\iy$, all $x\in {\mathbb R}$ and
all integers $r>0$; this implies $W_1=W_{-1}=W_{-3}=...
=0$; one then checks the explicit expressions
\bean
W_1&=&-\sum^{\iy}_{n=1}\frac{r^{n+2}}{t^{n+1}}\sum_{j=0}^{n-1}\left\{(n-j)Q''_{n-j,n-
j},Q'''_{jj}\right\}\\
\\
W_{-1}&=&\!-\sum^{\iy}_{n=1}\frac{r^{n}}{t^{n+1}}\sum_{j=0}^{n-1}\left(
\!\!\!\begin{array}{l}
\left\{(n-j)Q''_{n-j,n-j},\left(\begin{array}{l}
Q'''_{j,j-2}+2(j-1)^2Q'_{j-1,j-1}\\
+(j-1)x  Q''_{j-1,j-1}\\
-2\displaystyle{\sum_{\ell
+k=j-1}}\ell\{Q'_{\ell,\ell},Q''_{kk}\}
\end{array}\right)\right\}\\
\\
+\left\{(n-j)Q''_{n-j,n-j-2},Q'''_{jj}\right\}
\end{array}\!\!\!\right)
.\eean
Since this holds for all $t\searrow -\iy$ and $r>0$
integer, one must have for all $x$
 \be
\sum^{n-1}_{j=0}(n-j)\left\{Q''_{n-j,n-j},Q'''_{jj}\right\}=0,\qquad
n=1,2,... \label{leading} \ee with $Q_{00}=Q_0(x)$.


\noindent $\bullet$ For $n=1$, this is $
\{Q''_{11},Q'''_0\}=0, $ leading to
$
Q_{11}=c_0Q'_0+\al x+\beta.
$
Considering the asymptotics for $x\nearrow\iy$ and using
Lemmas \ref{lemma:Qprime} and \ref{lemma:F}, $\al$ and
$\beta$ must $=0$, leading to the equation
$$
0=Q_{11}-c_0Q'_0=(A^{'2}-AA'')(1-c_0)+{ O}(A^3)
=(1-c_0){ O}(A^2)+{ O}(A^3), $$ implying $c_0=1$
and so $Q_{11}=Q'_0$.


\noindent $\bullet$ For $n=2$, the equation, with the
previous data introduced, reads
\bean
0&=&\left\{2Q''_{22},Q'''_0\right\}+
\left\{Q''_{11},Q'''_{11}\right\}
=\left\{Q'''_0,-2Q''_{22}+Q^{iv}_0\right\} ,\eean which
upon solving leads to $
Q_{22}=\frac{1}{2}Q''_0+c_1Q'_0+\al'x+\beta'. $ For the
same reason as before $\al'=\beta' =0$. Then again using
Lemma \ref{lemma:Qprime} and Lemma \ref{lemma:F}, one
finds for $x\nearrow \iy$,
\bean
0=Q_{22}-\frac{1}{2}Q''_0-c_1Q'_0 
&=&-\frac{A^2}{2}+\frac{A^2}{2}-c_1(A^{' 2}-A'' A)+{  O}(A^3) 
\\ &=&c_1{  O}(A^2)+{  O}(A^3), \eean implying $c_1=0$
and thus $Q_{22}=\frac{1}{2}Q''_0$.

\noindent $\bullet$ By induction, assume $
Q_{ii}=\frac{1}{i!}Q_0^{(i)},\mbox{ for~} 0\leq i\leq
n-1. $ Then substituting this identity into equation
(\ref{leading}) and setting $
Q_{nn}=\frac{1}{n!}Q_0^{(n)}+R_n $ leads to pairwise
cancellations in equation (\ref{leading}) with only one
remaining contribution $ \{Q_0''', R''_n\}=0, $
with solution $ R_n=c_nQ'_0+\al'' x+\beta '' ,$ where
$\al '' =\beta '' =0$, and thus, by the asymptotics of
Lemmas \ref{lemma:Qprime} and \ref{lemma:F},
\bean
0=Q_{nn}-\frac{1}{n!}Q^{(n)}_0-c_nQ'_0
&=&-\frac{1}{n!}(A^2)^{(n-2)}+\frac{1}{n!}(A^2)^{(n-2)}-c_n(A^{'
2}-AA'')+{ O}(A^3)\\
&=&c_n{ O}(A^2)+{ O}(A^3) \eean leading to
$c_n=0$, completing the proof that
$Q_{nn}=\frac{1}{n!}Q^{(n)}_0$ for all $n=1,2,... $.
This proves the form of the leading term (coefficient of
$r^n$) in formulae (\ref{Qi}) for the $Q_n$'s. Since
from Lemma \ref{lemma:kernels}, from the form
(\ref{Qn7}) of the $Q_n$ and the fact that the
coefficients $K_n^{(r)}(u,v)$ in the expansion of
$K^{(r)}$ are divisible by $r$, the $Q_i$ themselves are
divisible by $r$. Since they skip every other degree in
$r$, this shows the formulae for $Q_1$ and $Q_2$; in
particular $Q_{20}=0$.

Setting this information $Q_{nn}=\frac{1}{n!}Q^{(n)}_0$
into the equation $W_{-1}=0$ and noticing that the
following term vanishes automatically, $ \sum_{\ell
+k=j-1}\ell\{Q'_{\ell\ell},Q''_{kk}\}=0, $ one finds for
$n\geq 3$,
\be \label{next} \ee
\vspace*{-1cm}
\bean
0&=&\sum^{n-1}_{j=2}\left\{\frac{Q_0^{(n-j+2)}}{(n-j-1)!},\left(Q'_{j,j-2}+\frac{1}{(
j-2)!}
\left(xQ_0^{(j-1)}+2(j-2)Q_0^{(j-2)}\right)\right)''\right\}\nonumber\\
\nonumber\\
& &-\sum^{n-1}_{j=0}\left\{\frac{Q_0^{(j+3)}}{j!},
(n-j)Q''_{n-j,n-j-2}\right\}
\no\\
 &=&\sum^{n-3}_{j=0}\left\{\frac{Q_0^{(j+3)}}{j!},
  \left(\begin{array}{l}
   Q'_{n-j-1,n-j-3}-(n-j)Q_{n-j,n-j-2}\\  \\
   +\frac{1}{(n-j-3)!}\left(xQ_0^{(n-j-2)}+2(n-j-3)Q_0^{(n-j-3)}\right)
   \end{array}
   \right)''\right\}
   \eean

\noindent $\bullet$ For $n=3$, by using the fact that
$Q_{20}=0$, the equation reads  $
\left\{Q'''_0,(xQ'_0-3Q_{31})''\right\}$ $=0 $ yielding $
Q_{31}=\left(\frac{x}{3}+c'_3\right)Q'_0+\al'''x+\beta'''
$ with $\al'''=\beta'''=0$. Thus, using Lemmas
\ref{lemma:F} and \ref{lemma:Qprime}, and the identity
(\ref{RemF}),
\bean
0&=&Q_{31}-\left(\frac{x}{3}+c'_3\right)Q'_0
\\
&=&-\frac{1}{3}\left((A^2)'+3\int^{\iy}_xA^{'2}(u)du\right)-
\left(\frac{x}{3}+c'_3\right)(A^{' 2}-AA'')+{ O}(A^3)
 \\
 &=& -\left( \int_x^{\iy}A^{'2}(u)du
  +\frac{1}{3}\left(
  (A^2)'+x(A^{'2}-AA'')\right)\right)-c'_3(A^{'2}-AA'')+{ O}(A^3)
\\
&=&-c'_3(A^{'2}-AA'')+{ O}(A^3)
 =c'_3{\bf O}(A^2)+{ O}(A^3)
  ,~~~~\mbox{using
  (\ref{RemF})}
  \eean
   yielding $c'_3=0$, and thus $Q_{31}=\frac{x}{3}Q'_0$.

\medbreak

\noindent $\bullet$ For $n=4$, using the formula for
$Q_{31}$, the equation reads
$$
0=\left\{Q'''_0,\left(-4Q_{42}+\frac{1}{3}\left(
xQ_0'\right)'+xQ''_0+2Q'_0\right)'' \right\},
$$
with solution
$
Q_{42}=\frac{1}{3}xQ''_0+\left(
\frac{7}{12}+c'_4\right)Q'_0+\al^{{\it iv}}x+\beta^{iv}
$
and thus $\al^{iv}=\beta^{iv}=0$, and by the same Lemmas
\ref{lemma:F} and \ref{lemma:Qprime},
\bean
0&=&Q_{42}-\frac{1}{3}xQ''_0-\left( \frac{7}{12}+c'_4\right)Q'_0\\
\\
&=&\frac{3}{2}A^{'2}-\frac{11}{24}(A^2)''
+\frac{1}{3}AA'' -\left(
\frac{7}{12}+c'_4\right)(A^{'2}-AA'')+
{ O}(A^3)\\
\\
&=&\frac{3}{2}A^{'2}-\frac{11}{12}(A^{'2}+AA'')+\frac{1}{3}AA''
-\left(
\frac{7}{12}+c'_4\right)(A^{'2}-AA'')+ { O}(A^3)\\
\\
&=&-c'_4(A^{'2}-AA'')+{ O}(A^3)
=c'_4{ O}(A^2)+{ O}(A^3), \eean implying
$c'_4=0$.

\noindent $\bullet$ Using induction, assume
$$
Q_{i,i-2}=\frac{x}{3(i-3)!}Q_0^{(i-2)}+\frac{7}{12(i-4)!}Q_0^{(i-3)}
$$
holds for $i=3,..., n-1$. Then setting \be
Q_{n,n-2}=\frac{x}{3(n-3)!}Q_0^{(n-2)}+\frac{7}{12(n-4)!}Q_0^{(n-3)}+R_n
\label{ansatz}\ee into equation (\ref{next}) gives the
simple equation for $R_n$, namely $ \{Q_0''',R''_n\}=0,
$ and so $R_n=c'_nQ'_0. $ Then, rewriting (\ref{ansatz})
and using the asymptotics for $Q_{n,n-2}$ (Lemma
\ref{lemma:F}), and for the derivatives $Q_0^{(i)}$
(Lemma \ref{lemma:Qprime}), and using the ODE for the
Airy function $xA=A''$, we get
\bean
0&=&Q_{n,n-2}-\frac{1}{12(n-3)!}(4xQ''_0+(3n-5)Q'_0)^{(n-4)}-c'_nQ'_0\\
\\
&=&\frac{1}{12(n-3)!}\left(\begin{array}{l}
6(n-1)A^{'2}-(3n-1)(A^{'2}+AA'')\\
+4AA''-(3n-5)(A^{'2}-AA'')
\end{array}
\right)
^{(n-4)}
 \\
 & &
-c'_n(A^{' 2}-AA'')+{ O}(A^3)
\\
&=&-c'_n(A^{' 2}-AA'')+{ O}(A^3)
=c'_n { O}(A^2)+{ O}(A^3) ,\eean implying
$c'_n=0$. Thus
%
   the $Q_n$'s are as announced in Theorem
 \ref{theo:one-time expansion}, namely
 $$
 Q_n=
   \frac{r^n}{n!}
   Q_0^{(n)}+\frac{r^{n-2}x}{3}\frac{Q_0^{(n-2)}}{(n-3)!}
    +\frac{7r^{n-2}}{12} \frac{Q_0^{(n-3)}}{(n-4)!}
   +{r^{n-4}}G(x)+\left(\begin{array}{l}
   \mbox{lower degree\!\!\!}\\
   \mbox{terms in $r$}
   \end{array}\right).
  $$
In the same fashion we compute $Q_{51}$ and $Q_{62}$;
for example, setting
$$ Q_5=
 \frac{r^5}{5!}
   Q_0^{(v)}+\frac{r^{3}x}{3}\frac{Q_0^{'''}}{2!}
    +\frac{7r^{3}}{12} \frac{Q_0^{''}}{1!}
   +rQ_{51}
   $$
   into the equation (\ref{one-time PDE}), one finds the following
   differential equation for $Q_{51}$, namely
   $$ \left\{Q_0''', 5Q_{51}''-8xQ_0''+4Q_0^{\prime\prime 2}  \right\}
+2Q_0^{\prime\prime\prime 2}(Q_0^{\prime\prime } +x) =0,
 $$
%
%
%
%
 which upon solving leads to
 $$Q_{51}=
\frac{1}{5}\left(\begin{array}{l}(x^2+c_5)Q_0'+4xQ_0
+Q_0^{\prime 2}
   \\+10\int_x^{\iy}Q_0-
   6\int_x^{\iy}dy\int_y^{\iy}du Q_0^{\prime\prime 2}
    \end{array}\right)
  =:\frac{1}{5} (\mathcal{F}_5+c_5Q'_0) ,$$
with a constant $c_5$, which has been shown by Aminul
Huq (private communication, 2008) to be $0$. Similarly one finds a differential
equation for $Q_{62}$ and upon solving one finds, for some integration constant $c_6$,
$$
Q_{62}=
 {\frac {1}{5}}  \left(
 \mathcal{F}_5'+\frac{5}{18}(x^2Q_0''+13(x+c_6)Q_0' ) \right)
 .$$

 \medbreak


 Assembling all the pieces, one notices that
two Taylor series in $Q_0$ and $Q_0'$ make their
appearance in the $1/t$-expansion of $Q(t,x)$, leading
to shifts in the argument of $Q_0(x)$ up to order $5$:
\be \label{Q}\ee
\vspace*{-1.1cm}
\bean
 Q(t,x)&=&\sum_0^{\iy}\frac{Q_i(x)}{t^i}\\
 &=&
  \sum_0^{\iy} \left(\frac{r}{t}\right)^n\frac{Q_0^{(n)}(x)}{n!}
 +\left(\frac{xr}{3t^3}+\frac{7r^2}{12t^4}\right)
  \sum_0^{\iy}
  \left(\frac{r}{t}\right)^n\frac{Q_0^{(1+n)}(x)}{n!}
  \no\\&& + \frac{r}{5t^5}\mathcal{F}_5+O(\frac{1}{t^6})
    \no\\
   &=&
   Q_0\Bigl(x+\frac{r}{t}\Bigr)+
    \Bigl(\frac{xr}{3t^3}+\frac{7r^2}{12t^4}\Bigr)
    Q_0'\Bigl(x+\frac{r}{t}\Bigr)+
    \frac{r}{5t^5}\mathcal{F}_5+O(\frac{1}{t^6})
    \no\\
   &=&
   Q_0\Bigl(x+\frac{r}{t}+\frac{xr}{3t^3}+\frac{7r^2}{12t^4}\Bigr)
   + \frac{r}{5t^5}\mathcal{F}_5+c_5Q_0')+O(\frac{1}{t^6})
  .\no\eean
 Exponentiating (\ref{Q}), remembering
that $e^{Q_0(y)}=
 \BP\left(\sup\mathcal {A}(t)\leq y\right)$,
taking a derivative $\frac{d}{dy}\BP\left(\sup\mathcal{A}(t)\leq y\right)$ and setting
$P_0(x):=\BP\left(\sup\mathcal {A}(t)\leq x\right)$ yields
\bea
 {\frac{d}{dx}
 \BP(\sup\mathcal {A}^{(r)}(t)\leq x)}
 &=&\Bigl(1+\frac{r}{3t^3}\Bigr)
 \frac{d}{dy}\BP\left(\sup\mathcal {A}(t)\leq
 y\right)\Bigr|_{y=
 (x+\frac{r}{t})(1+\frac{r}{3t^3})+
    \frac{r^2}{4t^4}}+O(\frac{1}{t^5})
  \no\\
  &=&
   P_0'  +\frac{r}{t}P_0'' +
    \frac{r^2}{2t^2}P_0'''
    +\frac{r}{6t^3}(r^2P_0^{(iv)}+(2xP_0')')
   \no \\
    &&
    +\frac{r^2}{24 t^4}(r^2 P_0^{(v)}+14P_0''+8(xP_0'')')
    +O(\frac{1}{t^5}).
   \label{7.7} \eea
For the moments with regard to the density $\frac{d}{dx}
 \BP(\sup\mathcal {A}^{(r)}(t)\leq x)$,
$$
\mu_{\ell}^{(r)}(t)=\int_{-\iy}^\iy   x^\ell
\frac{d}{dx}
 \BP(\sup\mathcal {A}^{(r)}(t)\leq x) dx ,
 $$
  one reads off from (\ref{7.7}) the following expansion
  for $\mu_{\ell}^{(r)}(t)$ in terms of the $t$-independent moments
  $$
\mu_{\ell}  =\int_{-\iy}^\iy   x^\ell \frac{d}{dx}
 \BP(\sup\mathcal {A} (t)\leq x) dx =
 \int_{-\iy}^\iy   x^\ell P_0^{\prime}(x) dx
 ,
 $$
  namely,
\bean \mu_{\ell}^{(r)}(t)&=&
 \mu_{\ell}  -\ell\frac{r}{t}\mu_{\ell-1} +
    \frac{r^2}{2t^2}\ell(\ell-1)\mu_{\ell-2}
    +\frac{r}{6t^3}(-r^2\ell(\ell-1)(\ell-2)\mu_{\ell-3}-2\ell\mu_\ell)
 \\&&   +\frac{r^2}{24 t^4}(r^2 \ell(\ell-1)(\ell-2)(\ell-3)
 \mu_{\ell-4}+\ell(8\ell-14)\mu_{\ell-1})
 +O(\frac{1}{t^5}).\eean
In particular, the mean and second moment behave as
$$
\mu_1^{(r)}(t)={ \mu_1}-{\frac {r}{t}}-  {\frac {r{
\mu_1}}{3{t}^{3}}}-  {\frac
{{r}^{2}}{4{t}^{4}}}+O\Bigl(\frac{1}{t^5}\Bigr)
$$
and
$$
\mu_2^{(r)}(t)=
 {  \mu_2}-  {\frac {2r{  \mu_1}}{t}}+{\frac
 {{r}^{2}}{{t}^{2}}}-
 {
\frac {2r{  \mu_2}}{3{t}^{3}}}+  {\frac {{r}^{2}{
\mu_1}}{6{t}^{4}}}+O\Bigl(\frac{1}{t^5}\Bigr).
$$
Hence the variance of the right edge of the process
      behaves as
%
     \bean \mbox{var}(\sup \mathcal{A}^{(r)}(t))=
     (\mu_2^{(r)}-\mu_1^{(r)2})(t)  &=&
    ( \mu_2^{ }-\mu_1^{
    2})\left(1-\frac{2r}{3t^3}\right)+O\Bigl(\frac{1}{t^5}\Bigr)\\&=&
 \mbox{var}(\sup\mathcal {A}^{(0)}(t))\left(1-\frac{2r}{3t^3}\right)
  +O\Bigl(\frac{1}{t^5}\Bigr),\eean
ending the proof of Theorem~\ref{theo:one-time expansion}.
\end{proof}

  \section{The $r$-Airy process, an interpolation between the Airy and Pearcey processes}\label{inter}


   Consider $n$ non-intersecting Brownian motions on ${\mathbb R}$, with
 $0<p<1$ and $b<a$:
%
%
\bean   
\BP^{(a,b)}_{n} \left(\begin{tabular}{c|c}
& all $x_j(0) =0$\\
all $x_j(t)\in E$ for $1\leq j\leq n$ &
 $pn$ paths end up at $a$ at $t=1$\\
  &
$(1-p)n$ paths end up at $b$ at $t=1$
\end{tabular}\right)
\eean
It is intuitive that, when $n\rg \iy$, the mean density of Brownian particles
has its support on one interval for $t\sim  0$ and on
two intervals for $t\sim  1$, so that a bifurcation
appears for some intermediate time $t_0$, where one
interval splits into two intervals. At this point the
boundary of the support of the mean density in $(x,t)$-space has a cusp; see Figure 0.3.
The Pearcey process describes this cloud of particles near the point of bifurcation, with time and space stretched in such a way that the outer particles appear infinitely far and such that the time horizon $t=1$ is at infinity. In \cite{AvM-asymmPearcey} it is shown that the same Pearcey process appears in the neighborhood of this cusp, independently of the target points $a$ and $b$, and the number $np$ of paths forced to $a$, showing ``universality" of the Pearcey process. It
is convenient to introduce the parametrization of $p$,
 \be
  p=\frac{1}{1+q^3} \mbox{~with~}  0< q< \iy . 
  \label{inter1}\ee
Setting for simplicity $b=0$, one has the following:

  \begin{proposition} \cite{AvM-asymmPearcey}
  For $n\rg \iy$, the cloud of Brownian particles lies
within a region, having a cusp at location
$(x_0\sqrt{n},t_0)$, with
\be
  x_0 =\frac{(2q-1)a}{q+1}t_0
,\qquad\qquad
t_0=\left({1+2a^2\frac{q^2-q+1}{(q+1)^2}}\right)^{-1}
. \label{inter2}\ee
  Moreover, the following Brownian motion probability tends to the
probability for the Pearcey process $\mathcal{P}(t)$:
 \bea
\lefteqn{\lim_{n\rg\iy}\BP_{BR}^{(0,a\sqrt{n})}
 \left(\mbox{all~}
x_j \left(t_0+ (c_0\mu)^2\frac{2\tau}
{n^{1/2}}\right)\in  {x_0}  n^{1/2}+c_0 A
\tau+c_0\mu\frac{E^c}{n^{1/4}}\right)
 }\no
 \\& &\hspace*{9cm}
  =\BP^\mathcal{P}\left(\mathcal{P}(\tau ) \cap E =
\emptyset\right), \label{Pearcey limit}\eea
 with constants expressed in terms of (\ref{inter1}) and (\ref{inter2}),
 $$
   \mu:=\left(\frac{q^2-q+1
 }{q}\right)^{1/4}>0,
\qquad 
c_0:=\sqrt{\frac{ t_0(1-t_0)}{2}}>0,
 $$ $$
 A := {q^{1/2}(1-\frac{x_0}{a})-q^{-1/2}\frac{x_0}{a}} .$$
\end{proposition}

The $r$-Airy process is an interpolation between the Pearcey process and the Airy process, which can easily be described by looking at Figure 0.3: 

\begin{theorem}\label{Interpolation Theorem} When $p\rg 0$ and for $n$ very large, such that $pn$ equals a fixed integer $r>0$, the tip $(x_0\sqrt{n},t_0)$ of the cusp (as given by (8.2)) moves towards the right hand boundary of the picture, and, in particular, to the tangency point of the line through $(a\sqrt{n},1)$ tangent to the curve $y =\sqrt{2nt(1-t)}$:
$$
 (x_0\sqrt{n},t_0)\rg\left(\frac{2a}{1+2a^2}\sqrt{n},\frac{1}{1+2a^2}\right)\in \mbox{curve}\{y=\sqrt{2nt(1-t)}\}
. $$
Also the Pearcey process near the cusp tends to the r-Airy process in the neighborhood of the point of tangency above.
\end{theorem}


  %


\begin{proof}
Indeed, letting $p\rg 0$, or what is the same from (\ref{inter1}), letting $q\rg \iy$, one sees from formula (\ref{inter2}) that the cusp is located at the point
 $$x_0\sqrt{n}= 2at_0\sqrt{n} \mbox{~~and~~} t_0=\frac{1}{1+2a^2}.$$
  This implies that the point $(y,t)=(x_0\sqrt{n},t_0)$
  belongs to the curve $y=\sqrt{2nt(1-t)}$ and that
   $$a\sqrt{n}=\sqrt{\frac{1-t_0}{2t_0}}\sqrt{n}=\rho_0\sqrt{\frac{n}{2}},$$
   establishing the first part of Theorem \ref{Interpolation Theorem}. That the Pearcey process tends to the $r$-Airy process will be done elsewhere.  \end{proof}


\section{Appendix}










The purpose of this appendix is to show that the first
few $Q_i$ of Theorem \ref{theo:one-time expansion} can
be obtained, {\em with tears}, by functional analytical
methods, in the style of Widom \cite{Widom}. The proof
requires many intricate identities involving the kernels
$K_i^{(r)}$, some of which can be found in Tracy-Widom
\cite{TW-Airy}. This section should convince the reader
of the usefulness of the PDE's in computing the
asymptotics for $t\rg -\iy$.

Remember the $L_i =(I+R_x)K_i^{(r)}$ from (\ref{L}),
where we now indicate the explicit dependency of the
resolvent $R_x=K^{(0)}(I-K^{(0)})^{-1}$ on $x$, since
all operators act on $L^2(x,\iy)$. Then $Q(t,x)$ has an
expansion, with $Q_i$'s given in (\ref{Q's}),
$$
Q=\Tr\log(I-K_t^{(r)})=\sum^{\iy}_{i=0}\frac{Q_i}{t^i}=
 \Tr\log(I-K^{(0)})+\sum^{\iy}_{i=1}\frac{Q_i}{t^i}
.$$
Throughout this section, we shall be using the
inner-product $$\la f,g\ra:= \int_{\mathbb R}
\raisebox{1mm}{$\chi$}{}_{(x,\iy)}(u)f(u)g(u)du .$$


\begin{proposition}\label{Theo:8.1}

\bean Q_1=rQ'_0 ,~~ Q_2=\frac{r^2}{2!}Q''_0
 ,~~
 Q_3=\frac{r^3}{3!}Q'''_0+\frac{r}{3 } xQ'_0.
\eean

\end{proposition}


\begin{lemma}\label{lemma:8.1}
$$
\Tr L_1^n=(-r)^n\la (I+R_x)A ,A \ra^n.
$$
\end{lemma}

\begin{proof} Indeed,
\bean
\lefteqn{\Tr((I+R_x)K_1)^n}\\
&=&(-r)^n\int_{(x,\iy)^n}du_1\ldots
du_n\left(\left((I+R_x) A(u_1)\right)A(u_2)\right)
\left(\left((I+R_x)A(u_2)\right)A(u_3)\right)
\\
& &\qquad\qquad\qquad \qquad\qquad\qquad\ldots
\left(\left((I+R_x)A(u_n)\right)
A(u_1)\right) 
\\
&=&(-r)^n\la (I+R_x)A(u),A(u)\ra^n .\eean\end{proof}

The
identities in the following Lemma can be found in or
deduced from Tracy-Widom \cite{TW-Airy}.

\begin{lemma} \label{TW}
\bea
R_x(x,x)&=&\la(I+R_x)A,A\ra \label{TW1}\\
\left(\frac{\pl}{\pl x}\!+\!\frac{\pl}{\pl
u}\!+\!\frac{\pl}{\pl v}\right)R_x(u,v)&=&-((I+R_x)A(u))
((I+R_x)A(v))\label{TW2}
\\
\left(\frac{\pl}{\pl x}+\frac{\pl}{\pl
u}\right)(I+R_x)A(u)&=&(I+R_x)A'(u)-(I+R_x)A(u)
 \la (I+R_x)A,A\ra .
 \label{TW3}\eea
%
%
\bea 2\la (I+R_x)A',A\ra -\la
(I+R_x)A,A\ra^2&=&-((I+R_x)A(x))^2 \label{TW4}
%
\\ 2\la (I+R_x)A,A''\ra -\la
(I+R_x)A',A'\ra&=&x\la(I+R_x)A,A\ra \label{TW4'}
\\%
  \frac{d}{dx}(I+R_x)A'(u)&=&-R_x(u,x)(I+R_x)A'(x)
\label{TW5}\eea
\bea \label{TW6}&&\frac{\pl}{\pl u}(I+R_x)A'(u)=u(I+R_x)A(u)-2\la
(I+R_x)A',A\ra
(I+R_x)A(u) \\
& &\qquad\qquad +\la (I+R_x)A,A\ra
(I+R_x)A'(u)+R(u,x)(I+R_x)A'(x).\no
\eea

\end{lemma}
%
\begin{lemma} \label{Q'}We have
\bean
Q'_0(x)&=&\la (I+R_x)A,A\ra\qquad \mbox{   and   } \qquad
Q''_0(x)=-((I+R_x)A(x))^2\\
\\
\frac{1}{2}Q'''_0(x)&=&-((I+R_x)A(x))((I+R_x)A'(x))
+((I+R_x)A(x))^2\la
(I+R_x)A,A\ra .\\
\eean
\end{lemma}
\begin{proof} One computes
\bean Q'_0(x)&=&
\frac{\pl}{\pl x}\Tr\log(I-K^{(0)}\raisebox{1mm}{$\chi$}{}_{(x,\iy)} )\\
\\
&=&-\frac{\pl}{\pl x}\Tr\left(K^{(0)}(u,v)+
\frac{1}{2}\int^{\iy}_xK^{(0)}(u,w)K^{(0)}(w,v)dw+\ldots\right)\\
\\
&=&-\frac{\pl}{\pl x}\left(\begin{array}{l}
\displaystyle{\int_x^{\iy}K^{(0)}(u,u)du+
\frac{1}{2}\int^{\iy}_x\int^{\iy}_x}K^{(0)}(u,w)K^{(0)}(
w,u)dudw\\
\\
\!\!+\displaystyle{\frac{1}{3}\int^{\iy}_x\!\!\!\int^{\iy}_x\!\!\!\int^{\iy}_x}K_0(u,w_1)K_0(w_1,w_2)
K^{(0)}(w_2,u)dudw_1dw_2+\ldots
\end{array}
\right)
\\
&=&K^{(0)}(x,x)+\int_x^{\iy}K^{(0)}(x,v)K^{(0)}(v,x)dv
\\
& &\qquad
+\int_x^{\iy}\!\!\!\int_x^{\iy}K^{(0)}(x,w_1)K^{(0)}
(w_1,w_2)K^{(0)}(w_2,x)dw_1dw_2+\ldots\\
\\
&=&R_x(x,x)
=\la (I+R_x)A,A\ra,\mbox{~using (\ref{TW1})} \eean by
the Neumann series for $R_x=(I-K^{(0)})^{-1}-I$.
 Moreover, using the previous result,
%
%
%
\bean Q''_0(x)=\frac{d}{dx}R_x(x,x)&=&
\left(\frac{\pl}{\pl x}+\frac{\pl}{\pl u}+\frac{\pl}{\pl
v} \right)R_x(u,v) \Bigl|_{u=v=x}
 \\
&=&-((I+R_x)A(u))(I+R_x)A(v))\Bigl|_{u=v=x}
~~ \mbox{using (\ref{TW2})}\\
&=&-((I+R_x)A(x))^2, \eean
 and again using the result just obtained and using identity
 (\ref{TW3}),
\bean
Q'''_0(x)&=&-2((I+R_x)A(x))\frac{d}{dx}(I+R_x)A(x)\\
\\
&=&-2((I+R_x)A(x))\left(\frac{\pl}{\pl x}+\frac{\pl}{\pl
u}\right)
(I+R_x)A(u)\Bigl|_{u=x}\\
\\
&=&-2((I+R_x)A(x))(I+R_x)A'(x)
+2((I+R_x)A(x))^2\la (I+R_x)A,A\ra
,\eean
 proving Lemma \ref{Q'}.\end{proof}

\begin{lemma}\label{Lemma:8.5}

\bean Q_1(x)&=&r\la (I+R_x)A,A\ra\\
 Q_2(x)&=&
  -\frac{r^2}{2}( (I+R_x)A(x))^2
\\
 Q_3(x)&=& \frac{r^3}{3}(((I+R_x)A(x))^2\la
(I+R_x)A,A\ra
-(I+R_x)A(x)(I+R_x)A'(x))\\&&+\frac{rx}{3}\la
(I+R_x)A,A\ra. \eean

\end{lemma}

\begin{proof} Indeed, by (\ref{Q's}) and (\ref{Ki's}),
\bea \label{Q1}
Q_1(x)=-\Tr L_1
&=&-\Tr K^{(r)}_1-\Tr R_xK^{(r)}_1  \\
&=&r\int_x^{\iy}du~A(u)\left(A(u)+
\int_x^{\iy}R_x(u,v)A(v)dv\right)\no\\
\no\\
&=&r\int_x^{\iy}A(u)((I+R)A)(u)du
=r\la (I+R_x)A,A\ra.\no
\eea
%
%
%
%
 %
 %
%
Computing $Q_2$ by (\ref{Q's}) involves $\Tr L_2$ and
$\Tr L_1^2$. Since $K_2^{(r)}(u,v)$ has a symmetric and
skew-symmetric part, and since $ I+R_x $ is symmetric,
and remembering the form of $K_2^{(r)}(u,v)$ in
(\ref{Ki's}), we have (by symmetry) and the fact that a
symmetric times a skew-symmetric operator is traceless,
 \bean
 \Tr L_2=\Tr(I+R_x)K^{(r)}_2&=&
 -\frac{r^2}{2}\Tr(I+R_x)(A'(u)A(v)+A(u)A'(v))
 \\
 &=&-r^2\la (I+R_x)A',A\ra .\eean
%
%
Hence, combining this result with the computation of
$\Tr L_1^2$ in Lemma \ref{lemma:8.1}, one finds, using
(\ref{TW4}),
\bea \label{Q2}
Q_2=-\Tr(L_2+\frac{L_1^2}{2})
%
&=&\frac{r^2}{2}(2\la (I+R_x)A',A\ra -\la (I+R_x)A,A\ra ^2)
 \\
&=&-\frac{r^2}{2}( (I+R_x)A(x))^2
.\no\eea
%
%
The computation of $Q_3$, by (\ref{Q's}), involves $\Tr
L_3$,~$\Tr L_1L_2$ and $\Tr L_1^3$. Using again the fact
that a symmetric times a skew-symmetric operator is
traceless, one reads off from the form of $K_3^{(r)}$
(see (\ref{Ki's})) the following:
%
%
\bean \Tr L_3 ={\Tr(I+R_x)K^{(r)}_3}
&=&-\frac{r^3}{3}(\la(I+R_x)A,A''\ra +\la (I+R_x)A',A'\ra )
\\
&&\qquad\qquad -\frac{r}{3}(2\la (I+R_x)A,A''\ra -\la
(I+R_x)A',A'\ra )
 \eean
%
%
%
and using
\vspace*{-.7cm}
 { \bean
 {(I+R_x)K^{(r)}_2}
&=&-\frac{r^2}{2}\left(((I+R_x)A'(u))A(v)
+((I+R_x)A(u))A'(v)\right )\\
\\
&&\qquad\qquad -\frac{r}{2}\left(((I+R_x)A'(u))A(v)-
((I+R_x)A(u))A'(v)\right ), \eean}
%
 one computes
%
%
%
{
 \bean
\Tr L_1L_2 &=& \Tr(I+R_x)K_1^{(r)}(I+R_x)K_2^{(r)}\\
&=&\frac{r^3}{2}\int\!\!\!\int_{(x,\iy)^2}((I+R_x)A)(w_1)A(w_2)
((I+R_x)A')
(w_2)A(w_1)dw_1dw_2\\
\\
& &+\frac{r^3}{2}\int\!\!\!
\int((I+R_x)A)(w_1)A(w_2)((I+R_x)A)(w_2)A'(w_1)dw_1dw_2
\\
\\
& &+\frac{r^2}{2}\int\!\!\int((I+R_x)A)(w_1)A(w_2)
((I+R_x)A')(w_2)A(w_1)dw_1dw_2\\
\\
& &-\frac{r^2}{2}\int\!\!\!\int((I+R_x)A)(w_1)A(w_2)
 ((I+R_x)A)(w_2)A'(w_1)dw_1dw_2
 \\
&=&\frac{r^3}{2}\left(\begin{array}{c}\la
(I+R_x)A,A\ra\la
(I+R_x)A',A\ra \\
\\
+\la (I+R_x)A,A\ra\la A',(I+R_x)A\ra\end{array}\right)\\
\\
& &+\frac{r^2}{2}\left(\begin{array}{c}\la
(I+R_x)A,A\ra\la
(I+R_x)A',A\ra \\
\\
-\la (I+R_x)A,A\ra\la A',(I+R_x)A\ra\end{array}\right)\\
\\
&=&r^3\la (I+R_x)A,A\ra \la (I+R_x)A',A\ra .\eean}
%
%
Putting the pieces together and using (\ref{Q's}) and
using Lemmas \ref{lemma:8.1} and \ref{Q'}, one obtains
{
 \bea \label{Q3}
 \qquad Q_3&=&-\Tr \Bigl(L_3+ {L_1L_2}
+\frac{1}{3}L_1^3\Bigr) \\
 &=&\frac{r^3}{3}\left(\begin{array}{c}\la
(I+R_x)A,A''\ra +\la
(I+R_x)A',A'\ra  \no\\
-3\la (I+R_x)A,A\ra \la (I+R_x)A',A\ra\no\\
+\la (I+R_x)A,A\ra ^3\end{array}\right)
 \\
 &&+\frac{r}{3}(2\la (I+R_x)A,A''\ra -
 \la (I+R_x)A',A'\ra )\no\\
\no\\
&=&\frac{r^3}{3}\left(\begin{array}{l}\la
(I+R_x)A,A''\ra +\la
(I+R_x)A',A' \ra  \no\\
-\la (I+R_x)A,A\ra\la (I+R_x)A',A\ra 
\no\\
+\la (I+R_x)A,A\ra \Bigl(\la  (I+R_x) A,A\ra^2-2\la
(I+R_x)A',A\ra \Bigr)
\end{array}\right)\no\\
\no\
& &+\frac{r}{3}(2\la (I+R_x)A,A''\ra -\la (I+R_x)A',A'\ra )\no\\
\no\\
&{=}&\frac{r^3}{3}(-(I+R_x)A(x)(I+R_x)A'(x)+
((I+R)A(x))^2\la
(I+R_x)A,A\ra )
 \no\\& &
+\frac{rx}{3}\la (I+R_x)A,A\ra
 ,
\no\eea}
%
%
using in the last equality (\ref{TW4}), (\ref{TW4'})
combined with
 (\ref{Identity3}) below. Then, using the differential equation $uA(u)=A''(u)$, one
checks:
\be  \label{Identity3}\ee
\vspace{-.8cm}
\bean
\lefteqn{\la (I+R_x)A',A'\ra} \\
&=&-A(x)(I+R_x)A'(x)-\la\frac{\pl}{\pl u}(I+R_x)A',A\ra
\mbox{~~(by integration by parts)}\no\\
&=&-A(x)(I+R_x)A'(x)-\la (I+R_x)A(u),uA(u)\ra
 +\la (I+R_x)A',A\ra \la (I+R_x)A,A\ra\no\\
& & -\la R(u,x),A(u)\ra
 (I+R)A'(x)
,~~~\mbox{using (\ref{TW6}),} \no
\\
 &=&-(I+R_x)A(x)(I+R_x)A'(x)-\la (I+R_x)A,A''\ra +\la (I+R_x)A',A\ra \la (I+R_x)A,A\ra.
\no\eean
%
%
\end{proof}

\begin{proof}[Proof of Proposition~\ref{Theo:8.1}]
 The formulae follow immediately from comparing the
  formulae of Lemmas~\ref{Q'} and~\ref{Lemma:8.5}.
\end{proof}





  \ack
  Pierre van Moerbeke thanks G\'erard Ben Arous for a
useful conversation (May 2006) and thanks Patrik Ferrari for a very interesting discussion (January 2008) concerning Theorem \ref{Theo1}. Mark Adler and Pierre van Moerbeke gratefully acknowledge the support of a National Science Foundation grant \#~DMS-07-04271. This work was partially done while PvM was a member of the Miller Institute for Research in Science, Berkeley, California. The support of National Science Foundation grant \#~DMS-07-0427, a European Science Foundation grant (MISGAM), a Marie Curie Grant (ENIGMA), a FNRS grant and a "Interuniversity Attraction Pole" (Belgium) grants are gratefully acknowledged.



\frenchspacing
\bibliographystyle{plain}

\vspace*{1cm}

adler@brandeis.edu

jonathan.delepine@uclouvain.be

pierre.vanmoerbeke@uclouvain.be

vanmoerbeke@brandeis.edu

\end{document}